\documentclass[final]{siamart1116}

\usepackage{amsfonts}
\usepackage{graphicx,epstopdf}
\usepackage{mathrsfs}
\usepackage{bm}
\usepackage{multirow}
\usepackage{hyperref} 
\hypersetup{colorlinks=false} %

\newsiamthm{claim}{Claim}
\newsiamremark{remark}{Remark}
\newsiamremark{expl}{Example}
\numberwithin{theorem}{section}

\allowdisplaybreaks

\title{Positivity-Preserving Analysis of Numerical 
	Schemes for Ideal 
	Magnetohydrodynamics} 

\author{
	Kailiang Wu%
	\thanks{Department of Mathematics, The Ohio State University, Columbus, OH 43210 USA (\email{wu.3423@osu.edu}).} 
}

\begin{document}
\maketitle


\begin{abstract}
	Numerical schemes {\em provably} preserving  
	the positivity of density and pressure
	are highly desirable  
	for ideal 
	magnetohydrodynamics (MHD), but the rigorous positivity-preserving (PP) analysis remains 
	challenging.    
	The difficulties mainly arise from the intrinsic complexity of the MHD equations as well as the indeterminate relation between the PP property and the divergence-free condition on the magnetic field. This paper presents the first 
	rigorous PP analysis of conservative schemes with the Lax-Friedrichs (LF) flux for one- and multi-dimensional ideal MHD. 
	The significant innovation is 
	the discovery of the theoretical connection between the PP property and a discrete divergence-free (DDF) condition. This connection is established through the generalized LF splitting properties, which are alternatives of the  usually-expected LF splitting property that does not hold for ideal MHD. The generalized LF splitting properties involve a number of admissible states strongly coupled by the DDF condition, making their derivation very difficult. We derive these properties via 
	a novel equivalent form of the admissible state set and an important inequality, which is skillfully constructed by technical estimates. Rigorous PP analysis is then presented for finite volume and discontinuous Galerkin schemes with the LF flux on uniform Cartesian meshes. In the 1D case, the PP property is proved for the first-order scheme with proper  
	numerical viscosity, and also for arbitrarily high-order schemes under conditions accessible by a PP limiter.  In the 2D case, we show that the DDF condition is necessary and crucial for achieving the PP property. It is observed that even slightly violating the proposed DDF condition may 
	cause failure to preserve 
	the positivity of pressure. 
	We prove that the 2D LF type scheme with proper numerical viscosity preserves both the positivity and the DDF condition. 
	Sufficient conditions are derived for 2D PP high-order schemes, and extension to 3D is discussed.  Numerical examples further confirm the theoretical findings.
\end{abstract}



\section{Introduction}
\label{sec:intro} 
Magnetohydrodynamics (MHD) play an important role in many fields including astrophysics, space physics and 
plasma physics, etc.
The $d$-dimensional ideal compressible MHD equations can be
written as 
\begin{equation}\label{eq:MHD}
\frac{{\partial {\bf U}}}{{\partial t}} + \sum\limits_{i = 1}^d {\frac{{\partial { {\bf F}_i}({\bf U})}}{{\partial x_i}}}  = {\bf 0},
\end{equation}
together with the divergence-free condition on the magnetic field
${\bf B}=(B_1,B_2,B_3)$, 
\begin{equation}\label{eq:2D:BxBy0}
\sum\limits_{i = 1}^d \frac{\partial B_i } {\partial x_i} =0,
\end{equation}
where $d=1, 2$ or $3$. In \eqref{eq:MHD}, the conservative vector ${\bf U} = ( \rho,\rho {\bf v},{\bf B},E )^{\top}$,
and ${\bf F}_i({\bf U})$ denotes the flux in the $x_i$-direction, $i=1,\cdots,d$, defined by
\begin{align*}
	{\bf F}_i({\bf U}) = \Big( \rho v_i,~\rho v_i {\bf v}  -  B_i {\bf B}   + p_{tot}  {\bf e}_i,~v_i {\bf B} - B_i {\bf v},~v_i(E+p_{tot} ) -  B_i({\bf v}\cdot {\bf B})
	\Big)^{\top}.
\end{align*}
Here $\rho$ is the density, the vector ${\bf v}=(v_1,v_2,v_3)$ denotes
the fluid velocity, 
$p_{\rm tot}$ is the total pressure consisting of 
the gas pressure $p$ and  magnetic pressure $p_m= \frac{|{\bf B}|^2}2$, 
the vector ${\bf e}_i$ represents the $i$-th row of the unit matrix of size 3, and
$E=\rho e + \frac12 \left( \rho |{\bf v}|^2 + |{\bf B}|^2 \right)$ is the total energy 
consisting of thermal, kinetic and magnetic energies with $e$ denoting the specific internal energy.
The equation of state (EOS) is needed to
close the system \eqref{eq:MHD}--\eqref{eq:2D:BxBy0}. 
For ideal gases it is given by 
\begin{equation}\label{eq:EOS}
p = (\gamma-1) \rho e,
\end{equation}
where the adiabatic index $\gamma>1$ . Although \eqref{eq:EOS} is widely used, 
there are situations where it is more appropriate to use other EOS. A general EOS can be
expressed as
\begin{equation}\label{eq:gEOS}
p = p(\rho,e),
\end{equation}
which is assumed to satisfy the following condition 
\begin{equation}\label{eq:assumpEOS}
\mbox{if}\quad \rho > 0,\quad \mbox{then}\quad e>0~\Leftrightarrow~p(\rho,e) > 0.
\end{equation}
Such condition is reasonable, holds for the ideal EOS \eqref{eq:EOS} and was also 
used 
in \cite{Zhang2011}.  


Since 
\eqref{eq:MHD} 
involves strong nonlinearity,
its analytic treatment is very difficult.  Numerical simulation
is 
a primary 
approach to 
explore the physical mechanisms in MHD.
In the past few decades, the numerical study of MHD has attracted
much attention, and
various 
numerical schemes
have been developed for \eqref{eq:MHD}. 
Besides the standard difficulty in solving nonlinear hyperbolic conservation laws, an additional numerical challenge for the MHD system comes from the divergence-free condition \eqref{eq:2D:BxBy0}. Although \eqref{eq:2D:BxBy0}
holds for the exact solution as long as it does initially, it cannot be easily preserved by a numerical scheme (for $d \ge 2$). Numerical evidence and some analysis in the literature indicate that negligence in dealing with the condition \eqref{eq:2D:BxBy0} can lead to numerical instabilities or nonphysical features in the computed solutions, see e.g.,  \cite{Brackbill1980,Evans1988,BalsaraSpicer1999,Toth2000,Dedner2002,Li2005}.
Up to now, many numerical techniques have been proposed to control the divergence error of numerical magnetic field. They include but are not limited to:
the  
eight-wave methods \cite{Powell1994,Chandrashekar2016},
the projection method \cite{Brackbill1980},
the hyperbolic divergence cleaning methods \cite{Dedner2002},
the locally divergence-free methods \cite{Li2005,Yakovlev2013}, 
the constrained transport method \cite{Evans1988} and its many variants 
\cite{Ryu1998,BalsaraSpicer1999,Londrillo2000,Balsara2004,Londrillo2004,Torrilhon2004,Torrilhon2005,Rossmanith2006,Artebrant2008,Li2008,Balsara2009,Li2011,Li2012,Christlieb2014}. 
The readers are also referred to an early survey in \cite{Toth2000}.

Another numerical challenge in the simulation of MHD is preserving the
positivity of density $\rho$ and pressure $p$.
In physics, these two quantities are non-negative. Numerically their
positivity is very critical, but not always satisfied by numerical solutions.
In fact, once the negative density or pressure is obtained in the simulations,
the discrete problem will become ill-posed due to the loss of hyperbolicity, causing the break-down of the simulation codes.
However, most of the existing MHD schemes are generally not positivity-preserving (PP), and thus may suffer from a large risk of failure when simulating MHD problems with strong discontinuity, low density, low pressure or low plasma-beta.
Several efforts have been made to reduce such risk. Balsara and Spicer \cite{BalsaraSpicer1999a} proposed a strategy to maintain positive pressure by switching the Riemann solvers for different wave situations. Janhunen \cite{Janhunen2000} designed a new 1D Riemann solver for the modified MHD system, and claimed its PP property by numerical experiments. Waagan \cite{Waagan2009} designed a positive linear reconstruction for second-order MUSCL-Hancock scheme, and conducted some 1D analysis based on the presumed PP  property of the first-order scheme. 
From a relaxation system, Bouchut et al. \cite{Bouchut2007,Bouchut2010} derived a multiwave approximate Riemann solver 
for 1D ideal MHD,  
and deduced sufficient conditions for the solver to satisfy discrete entropy inequalities and the PP property. 
Recent years have witnessed some significant advances in developing bound-preserving high-order schemes for hyperbolic systems (e.g.,  \cite{zhang2010,zhang2010b,zhang2011b,Hu2013,Xu2014,Liang2014,WuTang2015,moe2017positivity,WuTang2017ApJS,Xu2017}). 
High-order limiting techniques were well 
developed in \cite{Balsara2012,cheng} for the finite volume or DG methods of MHD, to enforce the admissibility\footnote{Throughout this paper, the admissibility of a solution or state $\bf U$ means  
	that the density and pressure corresponding to the conservative vector $\bf U$ are both positive, see 
	Definition \ref{def:G}.} of
the reconstructed or DG polynomial solutions at certain nodal points. These techniques are based on a presumed proposition that the cell-averaged solutions computed by those schemes are always admissible.
Such proposition has not yet been rigorously proved for those methods,    
although it could be deduced for the 1D schemes in \cite{cheng} 
under some 
assumptions (see Remark \ref{rem:chengassum}). 
With the presumed PP property of the Lax-Friedrichs (LF) scheme, Christlieb et al. \cite{Christlieb,Christlieb2016} developed PP 
high-order finite difference methods for \eqref{eq:MHD} by extending the parametrized flux limiters \cite{Xu2014,Xiong2016}. 

%
It was demonstrated numerically that 
the above PP treatments could enhance the robustness of MHD codes. 
However,
as mentioned in \cite{Christlieb},  
{\em there was no rigorous proof 
	to genuinely and
	completely show the PP property of those or any other schemes for \eqref{eq:MHD} in the multi-dimensional cases. 
	Even for the simplest first-order schemes, such as
	the LF scheme, the PP property is still unclear in theory. Moreover, it 
	is also 
	unanswered theoretically 
	whether the divergence-free condition \eqref{eq:2D:BxBy0} is 
	connected with the PP property of 
	schemes for \eqref{eq:MHD}. Therefore, it is significant to explore provably PP schemes for \eqref{eq:MHD} and 
	develop related theories for rigorous PP analysis.}


The aim of this paper is to carry out a
rigorous PP analysis 
of conservative finite volume and DG schemes with the LF flux 
for one- and multi-dimensional ideal MHD system \eqref{eq:MHD}. 
Such analysis is extremely nontrivial and technical.
The challenges mainly come from 
the intrinsic complexity of the system \eqref{eq:MHD}--\eqref{eq:2D:BxBy0}, 
as well as the unclear relation between the PP property and the divergence-free condition on the magnetic field. 
Fortunately, we find an important novel starting point of the analysis, based on an equivalent form of
the admissible state set. 
This form helps us to  
successfully derive the generalized LF splitting properties, which couple 
a discrete divergence-free (DDF) condition for the magnetic
field with
the convex combination of some LF splitting terms. 
These properties imply     
a theoretical connection between the PP property and the proposed DDF condition. 
As the generalized LF splitting properties involve a number of strongly coupled states,
their discovery and proofs are extremely technical.  
With the aid of these properties, we present the 
rigorous PP analysis for finite volume and DG schemes on uniform Cartesian meshes. 
Meanwhile, our analysis also reveals 
that the DDF condition is necessary and crucial for achieving the PP property. 
This finding is consistent with 
the existing numerical evidences 
that violating the divergence-free condition 
may more easily cause negative pressure 
(see e.g., \cite{Brackbill1980,Balsara2004,Rossmanith2006,Balsara2012}),   
as well as our previous work on the relativistic MHD \cite{WuTangM3AS}. 
Without considering the relativistic effect, the system \eqref{eq:MHD} yields unboundedness of velocities and poses difficulties essentially
different from the relativistic case. 
It is also 
worth mentioning that, as it will be shown, 
the 1D LF scheme is not always PP 
for piecewise
constant $B_1$, making some existing techniques \cite{zhang2010} for PP analysis inapplicable in the multi-dimensional ideal MHD case. 
Contrary to the usual expectation, we also find that the 1D LF scheme with a standard numerical viscosity parameter is not always PP, no matter how small the CFL number is. 
A proper 
viscosity parameter should be estimated, introducing 
additional difficulties 
in the analysis. 
Note that, 
for the incompressible
flow system in the vorticity-stream function formulation, 
there is also a  
divergence-free condition (but) on 
fluid velocity, i.e., the incompressibility condition,  
which is crucial 
in designing 
schemes that satisfies the 
maximum principle of vorticity, 
see e.g. \cite{zhang2010}.
An important difference in our MHD case is that 
our divergence-free quantity 
(the magnetic field)  
is also nonlinearly related to defining 
the concerned positive quantity --- the internal energy or pressure, see \eqref{eq:DefG}.

The paper is organized as follows. Section \ref{sec:eqDef} gives several important properties of the admissible states for the PP analysis. 
%
%
Sections \ref{sec:1Dpcp} and \ref{sec:2Dpcp} respectively study  
1D and 2D PP schemes.  
Numerical verifications are provided in Section \ref{sec:examples}, 
 and the 3D extension  
are given in Appendix \ref{sec:3D}. 
Section \ref{sec:con} concludes the paper with several remarks. 


\section{Admissible states}\label{sec:eqDef}

Under the condition \eqref{eq:assumpEOS}, it is natural to
define the set of 
admissible states $\bf U$ of the ideal MHD
as follows.

\begin{definition}\label{def:G}
	The set of admissible states of the ideal MHD 
	is defined by
	\begin{equation}\label{eq:DefG}
	{\mathcal G} = \left\{   {\bf U} = (\rho,{\bf m},{\bf B},E)^\top ~\Big|~ \rho > 0,~
	{\mathcal E}(  {\bf U}  ) := E- \frac12 \left( \frac{|{\bf m}|^2}{\rho} + |{\bf B}|^2 \right) > 0 \right\},
	\end{equation}
	where ${\mathcal E} ({\bf U}) =  \rho e$ denotes the internal energy.
\end{definition}


Given that the initial data are admissible, a 
scheme is defined to be PP if the  numerical solutions always stay in the set $\mathcal G$. 
One can see from \eqref{eq:DefG} that it is difficult to 
numerically preserve the positivity of ${\mathcal E}$, 
whose computation nonlinearly involves  
all the conservative variables. 
In most of numerical methods, the conservative
quantities are themselves evolved according to their own conservation laws,
which are seemingly unrelated to and 
numerically do not necessarily guarantee the positivity of the
computed ${\mathcal E}$. In theory, it is indeed a challenge to make a priori judgment
on whether a scheme is always PP under all circumstances or not.

\subsection{Basic properties}
To overcome the difficulties arising from the nonlinearity of 
the function ${\mathcal E} ({\bf U})$, 
we propose the following equivalent definition of ${\mathcal G}$.

\begin{lemma}[Equivalent definition]\label{theo:eqDefG}
	The admissible state set ${\mathcal G}$ is equivalent to
	\begin{equation}\label{eq:newDefG}
	{\mathcal G}_* = \left\{   {\bf U} = (\rho,{\bf m},{\bf B},E)^\top ~\Big|~ \rho > 0,~~~
	{\bf U} \cdot {\bf n}^* + \frac{|{\bf B}^*|^2}{2} > 0,~\forall~{\bf v}^*, {\bf B}^* \in {\mathbb {R}}^3 \right\},
	\end{equation}
	where
	$${\bf n}^* = \bigg( \frac{|{\bf v}^*|^2}2,~- {\bf v}^*,~-{\bf B}^*,~1 \bigg)^\top.$$
\end{lemma}

\begin{proof}
	If ${\bf U}   \in {\mathcal G}$, then $\rho>0$, and for any ${\bf v}^*, {\bf B}^* \in {\mathbb {R}}^3$, 
	\begin{align*}
		{\bf U} \cdot {\bf n}^* + \frac{|{\bf B}^*|^2}{2}
		=  \frac{\rho}{2}  \big| \rho^{-1} {\bf m} - {\bf v}^*  \big|^2 + \frac{| {\bf B} - {\bf B}^* |^2}2  +  {\mathcal E}(  {\bf U}  ) \ge {\mathcal E}(  {\bf U}  ) > 0, 
	\end{align*}
	that is ${\bf U} \in {\mathcal G}_*$. Hence $ {\mathcal G} \subset {\mathcal G}_* $. On the other hand, if
	${\bf U}   \in {\mathcal G}_*$, then $\rho>0$, and taking
	${\bf v}^* = \rho^{-1} {\bf m}$ and ${\bf B}^*={\bf B}$ gives
	$
	0< {\bf U} \cdot {\bf n}^* + {|{\bf B}^*|^2}/{2} = {\mathcal E}(  {\bf U}  ).
	$
	This means ${\bf U} \in {\mathcal G}$. Therefore $ {\mathcal G}_* \subset {\mathcal G}$. In conclusion, $ {\mathcal G} = {\mathcal G}_*$.
\end{proof}

{\em The two constraints in \eqref{eq:newDefG}
	are both linear with respect to $\bf U$, making it more effective to analytically
	verify the PP property of numerical schemes for ideal MHD.} 

The convexity of admissible state set is very useful in bound-preserving analysis, because 
it can help reduce the complexity of analysis 
if the schemes can be rewritten into certain
convex combinations, see e.g., \cite{zhang2010b,Zhang2011,Wu2017}. For the ideal MHD, the convexity of 
$\mathcal G_*$ or $\mathcal G$ can be easily shown by definition. 

\begin{lemma}[Convexity]\label{theo:MHD:convex}
	The set ${\mathcal G}_*$ is convex. 
	Moreover, $\lambda {\bf U}_1 + (1-\lambda) {\bf U}_0 \in {\mathcal G}_*$
	for any ${\bf U}_1 \in {\mathcal G}_*, {\bf U}_0 \in \overline{\mathcal G}_*$ and $\lambda \in (0,1]$, where $\overline{\mathcal G}_*$ is the closure of ${\mathcal G}_*$.
\end{lemma}

\begin{proof}
	The first component of $ {\bf U}_\lambda :=  \lambda {\bf U}_1 + (1-\lambda) {\bf U}_0  $ equals 
	$\lambda \rho_1 + (1-\lambda) \rho_0 >0$.
	For $\forall~ {\bf v}^*, {\bf B}^* \in {\mathbb {R}}^3$,
	${\bf U}_\lambda \cdot {\bf n}^* + \frac{|{\bf B}^*|^2}{2}
	=  \lambda \big( {\bf U}_1 \cdot {\bf n}^* + \frac{|{\bf B}^*|^2}{2} \big) + (1-\lambda) \big( {\bf U}_0 \cdot {\bf n}^* + \frac{|{\bf B}^*|^2}{2} \big)  > 0.$ 
	This shows ${\bf U}_\lambda \in {\mathcal G}_*$. 
	The proof is completed by the definition of convexity. 
\end{proof}

We also have the following orthogonal invariance, which can be verified directly.

\begin{lemma}[Orthogonal invariance] \label{lem:MHD:zhengjiao}
	Let ${\bf T} :={\rm diag}\{1,{\bf T}_3,{\bf T}_3,1\}$, where
	${\bf T}_3$ is any orthogonal matrix  of size 3.
	If ${\bf U} \in{\mathcal G}$, then
	${\bf T} {\bf U} \in{\mathcal G}$.
\end{lemma}


We refer to the following property \eqref{eq:LFproperty} as the {\em LF splitting property}, 
\begin{equation}\label{eq:LFproperty}
{\bf U} \pm \frac{ {\bf F}_i ({\bf U}) }{\alpha} \in {\mathcal G}, \quad \forall~{\bf U} \in {\mathcal G},~\forall~\alpha \ge \chi {\mathscr{R}}_i ({\bf U}) ,
\end{equation}
where $\chi\ge1 $ is some constant, and  
${\mathscr{R}}_i ({\bf U})$ is the spectral radius of the Jacobian matrix in $x_i$-direction, $i=1,2,3$. 
For the ideal MHD system with the EOS \eqref{eq:gEOS}, one has \cite{Powell1994}
$$
{\mathscr{R}}_i ({\bf U}) = |v_i| + \mathcal{C}_i,
$$
with 
$$
{\mathcal C}_i := \frac{1}{\sqrt{2}} \left[ \mathcal{C}_s^2 + \frac{ |{\bf B}|^2}{\rho} + \sqrt{ \left( \mathcal{C}_s^2 + \frac{ |{\bf B}|^2}{\rho} \right)^2 - 4 \frac{ \mathcal{C}_s^2 B_i^2}{\rho}  }  \right]^\frac12,
$$
where $\mathcal{C}_s=\sqrt{\gamma p/\rho}$ is the sound speed. 

If true, the LF splitting property would be very useful in 
analyzing the PP property of the schemes with the LF flux, see its roles in \cite{zhang2010b,WuTang2015,Wu2017} 
for the equations of hydrodynamics. 
Unfortunately, for the ideal MHD, 
\eqref{eq:LFproperty} is untrue in general,  
as evidenced numerically in \cite{cheng} for ideal gases. 
In fact, one can disprove \eqref{eq:LFproperty}, 
see the proof of the following proposition in Appendix \ref{sec:proof}.

\begin{proposition}\label{lemma:2.8}	 
	The LF splitting property \eqref{eq:LFproperty} does not hold in general.
\end{proposition}

\subsection{Generalized LF splitting properties} \label{sec:GLFs}

Since \eqref{eq:LFproperty} does not hold, we would like to seek some alternative properties which are weaker than \eqref{eq:LFproperty}. 
By considering the convex combination of some LF splitting terms, 
we discover the {\em generalized LF splitting properties} under 
some ``discrete divergence-free'' condition for the magnetic field. 
{\em As one of the most highlighted points of this paper, the discovery and proofs of such properties are very nontrivial and extremely technical.}

\subsubsection{A constructive inequality}

We first construct an inequality, which will play a pivotal role 
in establishing the generalized LF splitting properties.

\begin{lemma}\label{theo:MHD:LLFsplit}
	If ${\bf U}, \tilde{\bf U} \in {\mathcal G}$, then the inequality
	\begin{equation}\label{eq:MHD:LLFsplit}
	\bigg( {\bf U} - \frac{ {\bf F}_i({\bf U})}{\alpha}
	+
	\tilde{\bf U} + \frac{ {\bf F}_i(\tilde{\bf U})}{\alpha}
	\bigg) \cdot {\bf n}^* + |{\bf B}^*|^2
	+  \frac{  B_i - \tilde B_i }{\alpha} ({\bf v}^* \cdot {\bf B}^*)  > 0 ,
	\end{equation}
	holds for any ${\bf v}^*, {\bf B}^* \in {\mathbb{R}}^3$ and $|\alpha|>\alpha_{i} ({\bf U},\tilde{\bf U}) $, where $i\in\{1,2,3\}$, and 
	\begin{gather}\label{eq:alpha_i}
		\alpha_{i} ({\bf U},\tilde{\bf U}) = 
		\min_{\sigma \in \mathbb{R}} \alpha_{i} ( {\bf U}, \tilde {\bf U};\sigma ),
		\\ \nonumber
		\alpha_{i} ( {\bf U}, \tilde {\bf U};\sigma ) 
		=
		\max\big\{ |v_i|+ {\mathscr{C}}_i,|\tilde v_i| +  \tilde {\mathscr{C}}_i ,| \sigma v_i + (1-\sigma) \tilde v_i | + \max\{ {\mathscr{C}}_i , \tilde {\mathscr{C}}_i \}  \big\}
		+ f ( {\bf U}, \tilde {\bf U};  \sigma ),
	\end{gather}
	with 
	\begin{align*}
		&		f( {\bf U}, \tilde {\bf U}; \sigma) = \frac{ |\tilde{\bf B}-{\bf B}| }{\sqrt{2}} \sqrt{  \frac{\sigma^2}{\rho} + \frac{ (1-\sigma)^2 }{\tilde \rho}  },
		\\	
		&	
		{\mathscr{C}}_i = \frac{1}{\sqrt{2}}  \left[ {\mathscr{C}}_s^2 + \frac{ |{\bf B}|^2}{\rho} + \sqrt{ \left( {\mathscr{C}}_s^2 + \frac{ |{\bf B}|^2}{\rho} \right)^2 - 4 \frac{ {\mathscr{C}}_s^2 B_i^2}{\rho}  } \right]^\frac12,
	\end{align*}
	and ${\mathscr{C}}_s=\frac{p}{\rho \sqrt{2e}}$.
\end{lemma}

\begin{proof}
	{\tt (i)}. We first prove \eqref{eq:MHD:LLFsplit} for $i=1$. 
	Let define
	\begin{align*}
		& \Pi_u = \big( {\bf U}
		+
		\tilde{\bf U}
		\big) \cdot {\bf n}^* + |{\bf B}^*|^2, \quad  
		\Pi_f= \big(  {\bf F}_1({\bf U})
		- {\bf F}_1(\tilde{\bf U})
		\big) \cdot {\bf n}^* - \big( B_1 - \tilde B_1 \big) ({\bf v}^* \cdot {\bf B}^*) .
	\end{align*}
	Then it only needs to show
	\begin{equation}\label{eq:needproof1}
	\frac{|\Pi_f|}{\Pi_u} \le \alpha_1 ({\bf U},\tilde{\bf U}),
	\end{equation}
	by noting that
	\begin{equation}\label{eq:Piu}
	\Pi_u =  | {\bm \theta}|^2 > 0 ,
	\end{equation}
	where the nonzero vector ${\bm \theta} \in {\mathbb{R}}^{14}$ is defined as
	$$
	{\bm \theta}= \frac1{\sqrt{2}} \Big(
	\sqrt{\rho} ( {\bf v} - {\bf v}^* ),~
	{\bf B}-{\bf B}^*,~
	\sqrt{2 \rho e},~
	\sqrt{\tilde \rho} ( \tilde {\bf v} - {\bf v}^* ),~
	\tilde{\bf B}-{\bf B}^*,~
	\sqrt{2 \tilde \rho \tilde e}
	\Big)^\top.
	$$
	The proof of \eqref{eq:needproof1} is divided into the following two steps.

	{\tt Step 1}. Reformulate $\Pi_f$ into a quadratic form in the variables $\theta_j, 1\le j \le 14$. 
	We require that the coefficients of the quadratic form do not depends on ${\bf v}^*$ and ${\bf B}^*$.
	This is very nontrivial and becomes the key step of the proof. 
	We first arrange  $\Pi_f$  by a technical decomposition
	\begin{equation}\label{eq:3parts}
	\Pi_f = \Pi_1 + \Pi_2 + \Pi_3 + (\Pi_4-\tilde \Pi_4) ,
	\end{equation}
	where
	\begin{align*}
		& \Pi_j = \frac12 v_1^* \big( B_j^2-\tilde B_j^2 \big) -  v_1^* B_j^* ( B_j-\tilde B_j),\quad j=1, 2, 3, \\
		& \Pi_4 = \frac{\rho v_1 }2 |{\bf v}-{\bf v}^*|^2    +  v_1 \rho e  +  p ( v_1 - v_1^* )
		+ \sum_{j=2}^3 ( B_j (v_1-v_1^*)
		- B_1( v_j -v_j^* ) ) ( B_j - B_j^* ),\\
		& \tilde \Pi_4 = \frac{\tilde \rho \tilde v_1 }2 |\tilde{\bf v}-{\bf v}^*|^2   + \tilde v_1 \tilde \rho \tilde e  +  \tilde p ( \tilde v_1 - v_1^* )
		+ \sum_{j=2}^3 ( \tilde B_j ( \tilde v_1-v_1^*)
		- \tilde B_1( \tilde v_j -v_j^* ) ) ( \tilde B_j - B_j^* ).
	\end{align*}
	One can immediately rewrite $\Pi_4$ and $\tilde \Pi_4$ as
	\begin{align*}
		&\Pi_4  = v_1 \Big( \sum_{j=1}^3 \theta_j^2 +  \theta_7^2 \Big)
		+ 2{\mathscr{C}}_s 
		\theta_1 \theta_7 +  \frac{ 2 B_2} { \sqrt{\rho} } \theta_1 \theta_5 +  \frac{ 2 B_3} { \sqrt{\rho} } \theta_1 \theta_6
		-  \frac{2B_1}{\sqrt{\rho}} (\theta_2 \theta_5 + \theta_3 \theta_6),
		\\
		&
		\tilde \Pi_4 = \tilde v_1 \Big( \sum_{j=8}^{10} \theta_j^2 +  \theta_{14}^2 \Big)
		+  
		2 \tilde {\mathscr{C}}_s
		\theta_8 \theta_{14}
		+  \frac{ 2 \tilde B_2} { \sqrt{\tilde\rho} } \theta_8 \theta_{12} +  \frac{ 2 \tilde B_3} { \sqrt{\tilde \rho} } \theta_8 \theta_{13}
		-  \frac{2\tilde B_1}{\sqrt{\tilde \rho}} (\theta_9 \theta_{12} + \theta_{10} \theta_{13}).
	\end{align*}
	After a careful investigation, we find that $\Pi_j$, $j=1,2,3$, can be reformulated as
	\begin{equation*}
		\begin{split}
			\Pi_j & = \sigma_j \frac{\tilde B_j - B_j} { \sqrt{\rho} } ( \theta_1 \theta_{j+3} +
			\theta_1 \theta_{j+10} )
			+ (1-\sigma_j) \frac{\tilde B_j - B_j} { \sqrt{\tilde \rho} } (\theta_{8} \theta_{j+3} + \theta_{8} \theta_{j+10} )
			\\
			& \quad + \big( \sigma_j v_1 + (1-\sigma_j) \tilde v_1 \big) \theta_{j+3}^2 - \big( \sigma_j v_1 + (1-\sigma_j) \tilde v_1 \big) \theta_{j+10}^2 ,
		\end{split}
	\end{equation*}
	where $\sigma_1$, $\sigma_2$ and $\sigma_3$ can be taken as any real numbers.
	In summary, we have reformulated $\Pi_f$ into a quadratic form in the variables $\theta_j, 1\le j \le 14$.

	{\tt Step 2}. Estimate the upper bound of $\frac{|\Pi_f|}{\Pi_u}$. 
	There are several approaches to estimate the bound, resulting in 
	different formulas. One sharp upper bound is the spectral radius of the symmetric matrix associated with the above quadratic form, but 
	cannot be formulated explicitly and computed easily in practice. 
	An explicit sharp upper bound 
	is $\alpha_{1} ({\bf U},\tilde{\bf U})$ in \eqref{eq:alpha_i}. It is estimated as follows. We first notice that
	\begin{align*}
		\Pi_4  & = v_1 \Big( \sum_{j=1}^3 \theta_j^2 +  \theta_7^2 \Big)
		+ {\bm \vartheta}_6^\top {\bf A}_6 {\bm \vartheta}_6 ,
	\end{align*}
	where ${\bm \vartheta}_6 = (\theta_1,\theta_2,\theta_3,\theta_5,\theta_6,\theta_7)^\top$, and
	$$
	{\bf A}_6 =  \begin{pmatrix}
	0 & 0 & 0 & B_2 \rho^{-\frac12} & B_3 \rho^{-\frac12} & {\mathscr{C}}_s  \\
	0 & 0 & 0 & -B_1 \rho^{-\frac12} & 0 & 0 \\
	0 & 0 & 0 & 0 & -B_1 \rho^{-\frac12} & 0 \\
	B_2 \rho^{-\frac12} & -B_1 \rho^{-\frac12} & 0 & 0 & 0 & 0 \\
	B_3 \rho^{-\frac12} & 0 & -B_1 \rho^{-\frac12} & 0 & 0 & 0 \\
	{\mathscr{C}}_s  & 0 & 0 & 0 & 0 & 0
	\end{pmatrix}.
	$$
	The spectral radius 
	of ${\bf A}_6$ is $ {\mathscr{C}}_1$. This gives the following estimate 
	\begin{equation} \label{eq:PI4}
	\begin{aligned} 
	|\Pi_4|  & \le |v_1| \bigg( \sum_{j=1}^3
	\theta_{j}^2  +  \theta_7^2 \bigg)
	+ | {\bm \vartheta}_6^\top {\bf A}_6 {\bm \vartheta}_6 |
	\le |v_1| \bigg(  \sum_{j=1}^3
	\theta_{j}^2  +  \theta_7^2 \bigg) + {\mathscr{C}}_1 |{\bm \vartheta}_6 |^2
	\\
	& = ( |v_1| + {\mathscr{C}}_1 ) \bigg( \sum_{j=1}^3
	\theta_{j}^2  +  \theta_7^2 \bigg) + {\mathscr{C}}_1 \big(  \theta_5^2 +  \theta_6^2 \big). 
	\end{aligned}
	\end{equation}
	Similarly, we have
	\begin{align}
		|\tilde \Pi_4|  \le
		( |\tilde v_1| + \tilde {\mathscr{C}}_1 ) \bigg( \sum_{j=8}^{10}
		\theta_{j}^2  +  \theta_{14}^2 \bigg) + \tilde {\mathscr{C}}_1 \big(  \theta_{12}^2 +  \theta_{13}^2 \big).
		\label{eq:tPI4}
	\end{align}
	Let then focus on the first three terms at the right hand of \eqref{eq:3parts} and rewrite their summation as
	\begin{align}\label{eq:proofwkleq}
		\Pi_1 + \Pi_2 + \Pi_3 =  {\bm \vartheta}_8^\top {\bf A}_8 {\bm \vartheta}_8 +  \sum_{j=1}^3 \big( \sigma_j v_1 + (1-\sigma_j) \tilde v_1 \big)
		\big( \theta_{j+3}^2 -  \theta_{j+10}^2 \big)
		,
	\end{align}
	where ${\bm \vartheta}_8=(\theta_1,\theta_4,\theta_5,\theta_6,\theta_8,\theta_{11},\theta_{12},\theta_{13})^\top$, and
	$$
	{\bf A}_8 = \frac12
	\begin{pmatrix}
	~0~ &  ~{\bm \psi}~ & ~0~ & ~{\bm \psi}~ \\
	~{\bm \psi}^\top~ & ~{\bf O}~ &  ~\tilde{\bm \psi}^\top~  & ~{\bf O}~  \\
	~0~ & ~\tilde{\bm \psi}~ & ~0~ & ~\tilde {\bm \psi}~   \\
	~{\bm \psi}^\top~ & ~{\bf O}~ & ~\tilde {\bm \psi}^\top~ & ~{\bf O}~
	\end{pmatrix},
	$$
	with 
	${\bf O}$ denoting the $3\times 3$ null matrix, and
	\begin{align*}
		& {\bm \psi}=\rho^{-\frac12} \left( \sigma_1(\tilde B_1- B_1), \sigma_2(\tilde B_2-B_2), \sigma_3 (\tilde B_3-B_3) \right),
		\\
		& \tilde {\bm \psi}=\tilde \rho^{-\frac12} \left( (1-\sigma_1)(\tilde B_1- B_1), (1-\sigma_2)(\tilde B_2-B_2), (1-\sigma_3) (\tilde B_3-B_3) \right).
	\end{align*}
	Some algebraic manipulations show that the spectral radius of ${\bf A}_8$ is
	$$
	\varrho ({\bf A}_8)=\frac12 \left[ |{\bm \psi}|^2 + | \tilde {\bm \psi}|^2
	+ \sqrt{ ( |{\bm \psi}|^2 - | \tilde {\bm \psi}|^2   )^2 + 4 ( {\bm \psi} \cdot \tilde {\bm \psi} )^2 } \right]^\frac12.
	$$
	It then follows from \eqref{eq:proofwkleq} that, for $\forall \sigma_1,\sigma_2,\sigma_3\in \mathbb{R}$, 
	\begin{align*}
		|\Pi_1 + \Pi_2 + \Pi_3| & \le  \varrho ({\bf A}_8) | {\bm \vartheta}_8|^2 +  \sum_{j=1}^3 \big| \sigma_j v_1 + (1-\sigma_j) \tilde v_1 \big|
		\big| \theta_{j+3}^2 -  \theta_{j+10}^2 \big|.
	\end{align*}
	For simplicity, we set $\sigma_1=\sigma_2=\sigma_3=\sigma$, then 
	$\varrho ({\bf A}_8)= f({\bf U}, \tilde {\bf U};\sigma),$ 
	and 
	\begin{align}\nonumber
		|\Pi_1 + \Pi_2 + \Pi_3| & \le  f({\bf U}, \tilde {\bf U};\sigma)  | {\bm \vartheta}_8|^2 +
		| \sigma v_1 + (1-\sigma) \tilde v_1 | 
		\sum_{j=1}^3
		\big| \theta_{j+3}^2 -  \theta_{j+10}^2 \big|
		\\
		& \le  f({\bf U}, \tilde {\bf U};\sigma)  | {\bm \theta}|^2 +
		{| \sigma v_1 + (1-\sigma)\tilde v_1 | }  
		\sum_{j=1}^3
		\big( \theta_{j+3}^2 +  \theta_{j+10}^2 \big). \label{eq:PI123}
	\end{align}
	Combining \eqref{eq:3parts}--
	\eqref{eq:tPI4} and \eqref{eq:PI123}, we have 
	\begin{align*}
		\begin{split}
			|\Pi_f| & \le ( |v_1| + {\mathscr{C}}_1 ) \bigg( \sum_{j=1}^3
			\theta_{j}^2  +  \theta_7^2 \bigg)
			+ ( |\tilde v_1| + \tilde {\mathscr{C}}_1 ) \bigg( \sum_{j=8}^{10}
			\theta_{j}^2  +  \theta_{14}^2 \bigg)   +  f({\bf U}, \tilde {\bf U};\sigma)   | {\bm \theta}|^2
			\\
			& \quad    + {\mathscr{C}}_1 \big(  \theta_5^2 +  \theta_6^2 \big)  + \tilde {\mathscr{C}}_1 \big(  \theta_{12}^2 +  \theta_{13}^2 \big) +
			| \sigma v_1 + (1-\sigma)\tilde v_1 |
			\sum_{j=1}^3
			\big( \theta_{j+3}^2 +  \theta_{j+10}^2 \big)
		\end{split}
		\\
		\begin{split}
			& \le ( |v_1| + {\mathscr{C}}_1 ) \bigg( \sum_{j=1}^3
			\theta_{j}^2  +  \theta_7^2 \bigg)
			+ ( |\tilde v_1| + \tilde {\mathscr{C}}_1 ) \bigg( \sum_{j=8}^{10}
			\theta_{j}^2  +  \theta_{14}^2 \bigg)   +  f({\bf U}, \tilde {\bf U};\sigma)   | {\bm \theta}|^2
			\\
			& \quad   +
			\Big( | \sigma v_1 + (1-\sigma)\tilde v_1 |  + {\mathscr{C}}_1 \Big) \sum_{j=4}^6
			\theta_{j}^2
			+\Big( | \sigma v_1 + (1-\sigma)\tilde v_1 | + \tilde {\mathscr{C}}_1 \Big) \sum_{j=11}^{13}
			\theta_{j}^2
		\end{split}
		\\
		\begin{split}
			& \le \alpha_1 ( {\bf U},\tilde {\bf U};\sigma )~| {\bm \theta}|^2 = \alpha_1 ( {\bf U},\tilde {\bf U};\sigma )~\Pi_u,
		\end{split}
	\end{align*}
	for all $\sigma \in \mathbb{R}$. Hence 
	$$|\Pi_f| \le \Pi_u \min_{\sigma \in \mathbb{R} } \alpha_1 ( {\bf U},\tilde {\bf U};\sigma ) = \Pi_u \alpha_1 ( {\bf U},\tilde {\bf U} ),$$
	that is, the inequality \eqref{eq:needproof1} holds. The proof for the case of $i=1$ is completed.

	{\tt (ii)}. We then verify the inequality
	\eqref{eq:MHD:LLFsplit} for the cases $i=2$ and $3$, by using the inequality \eqref{eq:MHD:LLFsplit} for the case $i=1$ as well as the
	orthogonal invariance in Lemma \ref{lem:MHD:zhengjiao}.
	For the case of $i=2$, we introduce an orthogonal matrix ${\bf T} = {\rm diag} \{ 1, {\bf T}_3, {\bf T}_3, 1\}$ with 
	${\bf T}_3 := ({\bf e}_2^\top, {\bf e}_1^\top, {\bf e}_3^\top)$, where ${\bf e}_\ell $ is the $\ell$-th row of the unit matrix of size 3.   
	We then have ${\bf T} {\bf U}, {\bf T} \tilde {\bf U} \in {\mathcal G}$ by Lemma \ref{lem:MHD:zhengjiao}. 
	Let ${\mathcal H}_i({\bf U},\tilde{ {\bf U}},{\bf v}^*,{\bf B}^*,\alpha)$ denote the left-hand side term of \eqref{eq:MHD:LLFsplit}. 
	Using \eqref{eq:MHD:LLFsplit} with $i=1$ for ${\bf T} {\bf U}, {\bf T} \tilde {\bf U}, {\bf  v}^* {\bf T}_3 , {\bf B}^* {\bf T}_3$, we have 
	\begin{equation}\label{eq:MHD:LLFsplit111}
	{\mathcal H}_1( {\bf T} {\bf U}, {\bf T} {\tilde {\bf U}}  , {\bf  v}^* {\bf T}_3 , {\bf B}^* {\bf T}_3 ,\alpha)  > 0 ,
	\end{equation}
	for any $\alpha>\alpha_1 ( {\bf T} {\bf U}, {\bf T} {\tilde {\bf U}}  ) = \alpha_2 ({\bf U},\tilde{\bf U}) $.
	Utilizing ${\bf F}_1 ({\bf T}{ \bf U} ) = {\bf T} {\bf F}_2 ({\bf U})$ and the orthogonality of $\bf T$ and ${\bf T}_3$, we find that
	\begin{align} \nonumber
		{\mathcal H}_1( {\bf T} {\bf U}, {\bf T} {\tilde {\bf U}} , {\bf  v}^* {\bf T}_3 , {\bf B}^* {\bf T}_3  ,\alpha) 
		= {\mathcal H}_2({\bf U},\tilde{ \bf U},{\bf v}^*,{\bf B}^*,\alpha).
	\end{align}
	Thus \eqref{eq:MHD:LLFsplit111} implies  
	\eqref{eq:MHD:LLFsplit} 
	for $i=2$.
	%
	%
	%
	%
	%
	%
	Similar arguments for $i=3$. The proof is completed.
\end{proof}


\begin{remark}
	In practice, it is not easy to determine the minimum value in 
	\eqref{eq:alpha_i}. Since $\alpha_i({\bf U},\tilde{\bf U})$ only plays the role of a lower bound, 
	one can replace it with $\alpha_i({\bf U},\tilde{\bf U};\sigma)$ for a  special 
	$\sigma$. For example, taking $\sigma=\frac{\rho}{\rho+\tilde \rho}$ 
	minimizes $f({\bf U},\tilde{\bf U};\sigma)$ and gives 
	{\small
		\begin{equation*}
			\alpha_{i} \bigg({\bf U},\tilde{\bf U};\frac{\rho}{\rho+\tilde \rho} \bigg) =
			\max\bigg\{ |v_i|+ {\mathscr{C}}_i, |\tilde v_i| +  \tilde {\mathscr{C}}_i , \frac{|\rho v_i + \tilde \rho \tilde v_i |}{\rho + \tilde \rho} + \max\{ {\mathscr{C}}_i , \tilde {\mathscr{C}}_i \}  \bigg\}
			+ \frac{ |{\bf B}-\tilde{\bf B}| }{ \sqrt{2 (\rho + \tilde \rho)}  }.
	\end{equation*}} 
	Taking $\sigma=\frac{ \sqrt{\rho}}{\sqrt{\rho}+\sqrt{\tilde \rho}}$ gives
	{\small  
		\begin{equation*}
			\alpha_{i} \bigg({\bf U},\tilde{\bf U};\frac{\sqrt{\rho}}{\sqrt{\rho}+\sqrt{\tilde \rho}} \bigg) =
			\max\bigg\{ |v_i|+ {\mathscr{C}}_i, |\tilde v_i| +  \tilde {\mathscr{C}}_i , \frac{| \sqrt{\rho} v_i + \sqrt{\tilde \rho} \tilde v_i |} {\sqrt{\rho}+\sqrt{\tilde \rho}} + \max\{ {\mathscr{C}}_i , \tilde {\mathscr{C}}_i \}  \bigg\}
			+ \frac{ |{\bf B}-\tilde{\bf B}| }{ \sqrt{\rho} + \sqrt{\tilde \rho}  }.
		\end{equation*}
	}
\end{remark}

Let  $a_i := \max\{ {\mathscr{R}}_i ({\bf U}), {\mathscr{R}}_i(\tilde{\bf U}) \}$. For the gamma-law EOS,  the following proposition shows that 
$\alpha_{i} ({\bf U},\tilde{\bf U}) < 2a_i$ and 
$\alpha_{i} ({\bf U},\tilde{\bf U}) < a_i + {\mathcal O}( |{\bf U} -\tilde {\bf U}| )$, $i=1,2,3$. 	When ${\bf U}=\tilde {\bf U}$ with zero magnetic field, 
$\alpha_{i} ( {\bf U}, \tilde{\bf U} )=|v_i| + \frac{p}{\rho \sqrt{2e}}$, 
which is consistent with the bound in 
the LF splitting property for the Euler equations with 
a general EOS \cite{Zhang2011}. 	 

\begin{proposition}\label{prop:alpha}
	For any admissible states ${\bf U},\tilde{\bf U}$ of an ideal gas, 
	it holds  
	\begin{gather} \label{eq:aaaaWKL}
		\alpha_{i} ({\bf U},\tilde{\bf U}) 
		<2a_i,
		\\ \label{eq:bbbbWKL}
		\alpha_{i} ({\bf U},\tilde{\bf U}) 
		< a_i + \min \big\{ \big| | v_i| - |\tilde v_i| \big|, \big|  {\mathscr{C}}_i - \tilde {\mathscr{C}}_i  \big|  \big\} 
		+ \frac{ |{\bf B}-\tilde{\bf B}| }{ \sqrt{2 (\rho + \tilde \rho)}  }.
	\end{gather}
\end{proposition}

\begin{proof}
	The inequality \eqref{eq:aaaaWKL} can be shown as follows
	{\small \begin{align*}
			\alpha_{i} ({\bf U},\tilde{\bf U})  & \le 
			\alpha_{i} \bigg({\bf U},\tilde{\bf U};\frac{\sqrt{\rho}}{\sqrt{\rho}+\sqrt{\tilde \rho}} \bigg)  
			\\
			& \le 
			\max\Big\{ |v_i|+ {\mathscr{C}}_i,~|\tilde v_i| + 
			\tilde {\mathscr{C}}_i    \Big\}
			+ \frac{| \sqrt{\rho} v_i + \sqrt{\tilde \rho} \tilde v_i |} {\sqrt{\rho}+\sqrt{\tilde \rho}}
			+ \frac{ |{\bf B}-\tilde{\bf B}| }{ \sqrt{\rho} + \sqrt{\tilde \rho}  }
			\\
			& <  
			a_i
			+ \frac{\sqrt{\rho} } {\sqrt{\rho}+\sqrt{\tilde \rho}}
			\bigg( |v_i| + \frac{|\bf B|}{\sqrt{\rho}} \bigg)
			+ \frac{\sqrt{\tilde \rho} } {\sqrt{\rho}+\sqrt{\tilde \rho}}
			\bigg( |\tilde v_i| + \frac{|\tilde {\bf B}|}{\sqrt{\tilde \rho}} \bigg)
			\\
			& \le 
			a_i + \max \bigg\{ | v_i| + \frac{|\bf B|}{\sqrt{\rho}}, 
			|\tilde v_i| + \frac{|\tilde {\bf B}|}{\sqrt{\tilde \rho}}  \bigg\}
			\\
			& 
			\le a_i + \max \bigg\{ | v_i| + \frac{{\mathcal C}_i}{\sqrt{2}}, 
			|\tilde v_i| + \frac{|\tilde {\mathcal C}_i|}{\sqrt{2}}  \bigg\}
			< 2 a_i,
	\end{align*}}
	where we have used ${\mathscr{C}}_i < {\mathcal{C}}_i$ because of ${\mathscr{C}}_s = \sqrt{ \frac{(\gamma-1)p}{2\rho } }< {\mathcal{C}}_s$. 
	We then turn to prove \eqref{eq:bbbbWKL}. 
	Using the triangle inequality, one can easily show that  
	\begin{align*}
		&| v_i| + \tilde {\mathscr{C}}_i  
		\le \min \big\{ \big| | v_i| - |\tilde v_i| \big|, \big|  {\mathscr{C}}_i - \tilde {\mathscr{C}}_i  \big|  \big\}  + 
		\max\big\{ |v_i| +  {\mathscr{C}}_i, |\tilde v_i| + \tilde {\mathscr{C}}_i \big\},
		\\
		&|\tilde v_i| +  {\mathscr{C}}_i  
		\le \min \big\{ \big| | v_i| - |\tilde v_i| \big|, \big|  {\mathscr{C}}_i - \tilde {\mathscr{C}}_i  \big|  \big\}  + 
		\max \big\{ |v_i| +  {\mathscr{C}}_i, |\tilde v_i| + \tilde {\mathscr{C}}_i \big\}.
	\end{align*}
	Therefore,   
	\begin{align*}
		&	\max\bigg\{ |v_i|+ {\mathscr{C}}_i,|\tilde v_i| +  \tilde {\mathscr{C}}_i ,\frac{|\rho v_i + \tilde \rho \tilde v_i |}{\rho + \tilde \rho} + \max\{ {\mathscr{C}}_i , \tilde {\mathscr{C}}_i \}  \bigg\}
		\\
		& \quad \le \max \big\{ 
		|v_i|+ {\mathscr{C}}_i,|\tilde v_i| +  \tilde {\mathscr{C}}_i ,
		|\tilde v_i|+ {\mathscr{C}}_i,| v_i| +  \tilde {\mathscr{C}}_i 
		\big\}	
		\\
		& \quad \le \max\big\{ |v_i|+ {\mathscr{C}}_i,|\tilde v_i| +  \tilde {\mathscr{C}}_i  \big\} + \min \big\{ \big| | v_i| - |\tilde v_i| \big|, \big|  {\mathscr{C}}_i - \tilde {\mathscr{C}}_i  \big|  \big\} 
		\\
		& \quad 
		< a_i + \min \big\{ \big| | v_i| - |\tilde v_i| \big|, \big|  {\mathscr{C}}_i - \tilde {\mathscr{C}}_i  \big|  \big\}. 
	\end{align*}
	Then using $\alpha_{i} ({\bf U},\tilde{\bf U} ) \le \alpha_{i} \big({\bf U},\tilde{\bf U};\frac{\rho}{\rho+\tilde \rho} \big)$ completes the proof. 
\end{proof}


\begin{remark}
	It is worth emphasizing the importance of the last term at the left-hand side of \eqref{eq:MHD:LLFsplit}.
	This term is extremely technical, necessary and crucial in deriving
	the generalized LF splitting properties. 
	Including this term becomes one of the breakthrough points in 
	this paper.   
	The value of this term is not always positive or negative. However, without this term, the inequality \eqref{eq:MHD:LLFsplit} does not hold, even if $\alpha_i$ is replaced with $\chi\alpha_i$ for any constant $\chi \ge 1$.  
	More importantly, this term can be canceled out dexterously
	under the ``discrete divergence-free'' condition \eqref{eq:descrite1DDIV} or \eqref{eq:descrite2DDIV},
	see the proofs of generalized LF splitting properties in 
	the following theorems.
\end{remark}

Let us figure out 
some facts and observations. 
Note that the inequality \eqref{theo:MHD:LLFsplit} in Lemma \ref{theo:MHD:LLFsplit} 
involves two states (${\bf U}$ and $\tilde {\bf U}$). 
In the relativistic MHD case (Lemma 2.9 in \cite{WuTangM3AS}), %
we derive the 
generalized LF splitting properties by an inequality, which is similar to \eqref{theo:MHD:LLFsplit} but involves only one state. 
It seems natural to conjecture a similar ``one-state'' inequality for the ideal MHD case in the following form
\begin{equation}\label{eq:MHD:LLFsplit:one-state}
\bigg( {\bf U} + \frac{ {\bf F}_i({\bf U})}{\alpha}
\bigg) \cdot {\bf n}^* + \frac{|{\bf B}^*|^2}{2} 
+  \frac{  1 }{\alpha} \bigg( v_i^* \frac{|{\bf B}^*|^2}{2} -   B_i({\bf v}^* \cdot {\bf B}^*) \bigg) > 0, \quad 
\forall {\bf U} \in {\mathcal G},~~ 
\forall {\bf v}^*,{\bf B}^* \in {\mathbb R}^3,
\end{equation}
for any $ |\alpha| > \widehat \alpha_i ({\bf U})$, 
where 
the lower bound $\widehat \alpha_i ({\bf U})$ is expected to be independent of ${\bf v}^*$ and ${\bf B}^*$. 
For special relativistic MHD, the lower bound can be taken as {\em the speed of light} \cite{WuTangM3AS}, 
which is a constant, and brings us much convenience because any velocities 
(e.g., $|{\bf v}|$ and $|{\bf v}^*|$) are uniformly smaller than such a constant  according to the theory of special relativity.  
However, unfortunately for the ideal MHD, it is {\em impossible} to 
establish \eqref{eq:MHD:LLFsplit:one-state} 
for any $ |\alpha| > \widehat \alpha_i ({\bf U})$ with a desired bound $ \widehat \alpha_i ({\bf U})$ only dependent on ${\bf U}$. 
This is because the non-relativistic velocities are generally unbounded. As  $v_i^* {\rm sign} (-\alpha)$ and $|{\bf B}^*|$ approach $+\infty$, 
the negative cubic term ${v_i^* |{\bf B}^*|^2}/{2\alpha}$ in \eqref{eq:MHD:LLFsplit:one-state} 
dominates the sign and cannot be controlled by any other terms at the left-hand side \eqref{eq:MHD:LLFsplit:one-state}. 
 Hence, the construction of 
	generalized LF splitting properties in the ideal MHD case 
	has difficulties  essentially different from the special MHD case.  
If not requiring $\widehat \alpha_i$ to be independent of ${ v}_i^*$, we have the following proposition with the proof displayed in Appendix \ref{sec:proofwkl111}.

\begin{proposition}\label{lem:fornewGLF}
	The inequality \eqref{eq:MHD:LLFsplit:one-state} 
	holds for any $ |\alpha| > \widehat \alpha_i ({\bf U},{v}_i^*)$ 
	and any $i \in \{1,2,3\}$, where 
	$$
	\widehat \alpha_i ({\bf U},{v}_i^*) = 
	\max \big\{ |v_i|, |v_i^*|  \big\} 
	+ {\mathscr{C}}_i.
	$$	
\end{proposition}

\subsubsection{Derivation of generalized LF splitting properties} \label{sec:gLxF}


We first present the 1D generalized LF splitting property.

\begin{theorem}[1D generalized LF splitting]\label{theo:MHD:LLFsplit1D}
	If $ \hat{\bf U}= (\hat\rho, \hat{\bf m}, \hat{\bf B}, \hat E)^{\top} $ and $ \check{\bf U}=(\check \rho, \check{\bf m}, \check{\bf B}, \check{E})^{\top} $ both belong to $\mathcal G$, and satisfy
	1D 
	``discrete divergence-free'' condition
	\begin{equation}\label{eq:descrite1DDIV}
	\hat{B_1} - \check {B_1}=0,
	\end{equation}
	then for any  $\alpha > \alpha_1 (\hat{\bf U},\check{\bf U}) $  it holds
	\begin{equation}\label{eq:MHD:LLFsplit1D}
	\overline{\bf U}:=\frac{1}{2} \bigg( \hat{\bf U} - \frac{ {\bf F}_1(\hat{\bf U})}{\alpha}
	+
	\check{\bf U} + \frac{ {\bf F}_1(\check{\bf U})}{\alpha} \bigg)
	\in {\mathcal G}.
	\end{equation}
\end{theorem}
\begin{proof}
	The first component of $\overline{\bf U}$ equals $\frac12 \big( \hat \rho \big(1-\frac{\hat v_1}{\alpha}\big) + \check\rho \big(1+\frac{\check v_1}{\alpha}\big) \big) > 0$.
	For any ${\bf v}^*,{\bf B}^*\in \mathbb{R}^3$, utilizing Lemma \ref{theo:MHD:LLFsplit} and the condition \eqref{eq:descrite1DDIV} gives
	{\small \begin{align*}
			\overline{\bf U} \cdot {\bf n}^* + \frac{|{\bf B}^*|^2}{2}
			= \frac12 \bigg( \hat{\bf U} - \frac{ {\bf F}_1(\hat{\bf U})}{\alpha}
			+
			\check{\bf U} + \frac{ {\bf F}_1(\check{\bf U})}{\alpha}
			\bigg) \cdot {\bf n}^* + \frac{|{\bf B}^*|^2}{2}
			>   \frac{   \check B_1 - \hat B_1 }{2\alpha} ({\bf v}^* \cdot {\bf B}^*) = 0.
	\end{align*}}
	This implies $\overline{\bf U}\in{\mathcal G}_* = {\mathcal G}$.
\end{proof}

\begin{remark}\label{rem:exa}
	As indicated by Proposition \ref{prop:alpha}, 
	the bound $\alpha_1 (\hat{\bf U},\check{\bf U}) $ for $\alpha$  
	can be very close to  $a_1 = \max\{ {\mathscr{R}}_1 (\hat{\bf U}), {\mathscr{R}}_1(\check{\bf U}) \}$, which 
	is the numerical viscosity coefficient in the standard local LF scheme. 
	Nevertheless, \eqref{eq:MHD:LLFsplit1D} does not hold for $\alpha = a_1$ in general.  
	A counterexample can be given by 
	considering 
	the following admissible states of ideal gas with $\gamma = 1.4$ and $\hat B_1 = \check B_1$, 
	\begin{equation}\label{eq:contourex}
	\begin{cases}
	\hat{\bf U}= ( 0.2,0,0.2,0,10,5,0,62.625 )^\top, 
	\\
	\check{\bf U}= ( 0.32,0,-0.32,0,10,10,0,100.16025 )^\top. 
	\end{cases}
	\end{equation}
	For \eqref{eq:contourex} and $\alpha= a_1 $, one can verify that 
	$\overline{\bf U}$ in \eqref{eq:MHD:LLFsplit1D} 
	satisfies ${\mathcal E}(\overline{\bf U}) 
	<-0.05$ and 
	$\overline{\bf U} \notin {\mathcal G}$. 
\end{remark}
\begin{remark}\label{rem:chengassum}
	The proof of Lemma 2.1 in \cite{cheng} implies that
	\begin{equation}\label{eq:statebycheng}
	{\bf U}_\lambda := \hat{\bf U} - \lambda \big( {\bf F}_1 ( \hat{\bf U} ) + a_1 \hat {\bf U}  - {\bf F}_1 ( \check{\bf U} ) - a_1 \check{\bf U} \big) \in {\mathcal G}, \quad \forall \lambda \in \big(0, 1/{(2a_1)}\big], 
	\end{equation}
	holds for all admissible states $\hat{\bf U}, \check {\bf U}$ with  $\hat B_1= \check B_1$. 
	On the contrary, for the special admissible states $\hat{\bf U},{\check {\bf U}}$ in \eqref{eq:contourex}, 
	Remark \ref{rem:exa} yields that \eqref{eq:statebycheng} does not always hold when $\lambda $ is close to $\frac{1}{2a_1}$,
	because
	$
	\mathop {\lim }\limits_{\lambda  \to {1}/{(2 a_1)}} {\mathcal E} ( {\bf U}_\lambda  ) 
	={\mathcal E}( \overline {\bf U}) < 0 .
	$ 
	This deserves further explanation, as 
	the derivation of \eqref{eq:statebycheng} in \cite{cheng} is 
	not mathematically rigorous but based on two assumptions.
	One assumption is very reasonable (but unproven), stating that 
	the exact solution ${\bf U}(x_1,t)$ to the 1D Riemann problem (RP)
	\begin{equation}\label{eq:RP}
	\begin{cases}
	\frac{\partial {\bf U}}{\partial t} + \frac{\partial {\bf F}_1 ( {\bf U} )}{\partial x_1} = {\bf 0} ,\\
	{\bf U}(x_1,0) = \begin{cases}
	\hat {\bf U},  &x_1 < 0,\\
	\check {\bf U}, &x_1 > 0,
	\end{cases}
	\end{cases}
	\end{equation}
	is always admissible if $\hat{\bf U},~\check{\bf U}\in {\mathcal G}$ with $\hat B_1= \check B_1$. 
	Another ``assumption'' (not mentioned but implicitly used in \cite{cheng}) is that $a_1=\| {\mathscr{R}}_1 ( {\bf U}(\cdot,0) )\|_{\infty}$ is an upper bound of the maximum wave speed in the above RP. 
	In fact, $a_1$ may not always be such a bound when the 
	fast shocks exist in the RP solution, as indicated in \cite{Guermond2016} for the gas dynamics system (with zero magnetic field). 
	Hence, the latter assumption may affect some 1D analysis in \cite{cheng}, 
	see our finding in Theorem \ref{eq:1DnotPP}. 
	It is also worth emphasizing that the 1D analysis in \cite{cheng} could work in general	if 
	$\| {\mathscr{R}}_1 ( {\bf U}(\cdot,0) )\|_{\infty}$ is replaced with a rigorous upper bound of the maximum wave speed in the RP. 
\end{remark}

\begin{remark}
	Proposition \ref{lem:fornewGLF} can also be used to
	derive generalized LF splitting properties, 
	see Appendix \ref{sec:gLFnewproof} for 
	the 1D  
	case. 
\end{remark}

We then present the multi-dimensional generalized LF splitting properties.

\begin{theorem}[2D generalized LF splitting]\label{theo:MHD:LLFsplit2D}
	If $\bar{\bf U}^i$, $\tilde{\bf U}^{i}$, $\hat{\bf U}^i$,
	$\check{\bf U}^i
	\in {\mathcal G}$ for $i=1,\cdots,{\tt Q}$ satisfy the 2D  ``discrete divergence-free'' condition
	\begin{equation}\label{eq:descrite2DDIV}
	\frac{{\sum\limits_{i=1}^{\tt Q} {{\omega _i}({\bar B_1}^i - {\tilde B_1}^i)} }}{{\Delta x}}
	+ \frac{{\sum\limits_{i=1}^{\tt Q} {{\omega _i}({\hat B_2}^i - {\check B_2}^i)} }}{{\Delta y}}
	=0,
	\end{equation}
	where $\Delta x,\Delta y >0$, and the sum of the positive numbers
	$\left\{\omega _i \right\}_{i=1}^{\tt Q}$ equals one,
	then for any $\alpha_{1}^{\tt LF}$ and $\alpha_{2}^{\tt LF}$ satisfying
	$ \alpha_{1}^{\tt LF} > \max_{1\le i\le {\tt Q}}  \alpha_1 ( \bar { \bf U }^i,  \tilde{
		\bf U}^i  ) $, $\alpha_{2}^{\tt LF} > \max_{1\le i\le {\tt Q}}
	\alpha_2 (  \hat { \bf U }^i , \check{\bf U}^i ),
	$
	it holds
	\begin{equation} \label{eq:MHD:LLFsplit2D}
	\begin{split}
	\overline{\bf U}:= 
	\frac{1}{ 2\left( \frac{\alpha_{1}^{\tt LF} }{\Delta x} + \frac{\alpha_{2}^{\tt LF}}{\Delta y} \right)}
	\sum\limits_{i=1}^{\tt Q} { \omega _i  }
	\bigg[
	&\frac{\alpha_{1}^{\tt LF}}{\Delta x} \bigg(
	\bar { \bf U }^i -  \frac{ {\bf F}_1 ( \bar { \bf U}^i)  }{ \alpha_{1}^{\tt LF} }
	+\tilde{
		\bf U}^i +   \frac{ {\bf F}_1 ( \tilde { \bf U}^i)  }{ \alpha_{1}^{\tt LF} }
	\bigg)
	\\
	+&
	\frac{\alpha_{2}^{\tt LF}}{\Delta y} \bigg(
	\hat { \bf U }^i -  \frac { {\bf F}_2 ( \hat { \bf U}^i)  } { \alpha_{2}^{\tt LF} }
	+\check{
		\bf U}^i +  \frac { {\bf F}_2 ( \check { \bf U}^i)  } { \alpha_{2}^{\tt LF} }
	\bigg)
	\bigg]
	\in {\mathcal G}.
	\end{split}
	\end{equation}
\end{theorem}

\begin{proof}
	The first component of $\overline{\bf U}$  equals  
	\begin{align*}
		&   
		\frac1{ 2\left( \frac{\alpha_{1}^{\tt LF} }{\Delta x} + \frac{\alpha_{2}^{\tt LF}}{\Delta y} \right)}
		\sum\limits_{i=1}^{\tt Q} {{\omega _i}}
		\bigg(
		\frac{ \bar { \rho }^i ( \alpha_{1}^{\tt LF} -   { {\bar v_1}^i  } )
			+\tilde{
				\rho }^i ( \alpha_{1}^{\tt LF} +   {{\tilde v_1}^i} )   }{\Delta x}
		+
		\frac{ \hat { \rho }^i ( \alpha_{2}^{\tt LF} -   { { \hat v_2}^i  } )
			+\check{
				\rho}^i ( \alpha_{2}^{\tt LF} +  {\check v_2}^i) }{\Delta y}
		\bigg),
	\end{align*}
	which is positive. 
	For any ${\bf v}^*,{\bf B}^* \in \mathbb{R}^3$, using Lemma \ref{theo:MHD:LLFsplit} and the condition
	\eqref{eq:descrite2DDIV} gives
	{\small \begin{align*}
			&  \bigg(  \overline{\bf U} \cdot {\bf n}^* + \frac{|{\bf B}^*|^2}{2} \bigg) \times  2\left( \frac{\alpha_{1}^{\tt LF} }{\Delta x} + \frac{\alpha_{2}^{\tt LF}}{\Delta y} \right)
			\\
			& \quad =
			\sum\limits_{i=1}^{\tt Q} { \omega _i  }
			\Bigg\{
			\frac{\alpha_{1}^{\tt LF}}{\Delta x} \bigg[\bigg(
			\bar { \bf U }^i -  \frac{ {\bf F}_1 ( \bar { \bf U}^i)  }{ \alpha_{1}^{\tt LF} }
			+\tilde{
				\bf U}^i +   \frac{ {\bf F}_1 ( \tilde { \bf U}^i)  }{ \alpha_{1}^{\tt LF} }
			\bigg) \cdot {\bf n}^* + |{\bf B}^*|^2 \bigg]
			\\
			& \quad \qquad \quad \ +
			\frac{\alpha_{2}^{\tt LF}}{\Delta y} \bigg[ \bigg(
			\hat { \bf U }^i -  \frac { {\bf F}_2 ( \hat { \bf U}^i)  } { \alpha_{2}^{\tt LF} }
			+\check{
				\bf U}^i +  \frac { {\bf F}_2 ( \check { \bf U}^i)  } { \alpha_{2}^{\tt LF} }
			\bigg) \cdot {\bf n}^* + |{\bf B}^*|^2 \bigg]
			\Bigg\}
			\\[2mm]
			& \quad \overset{\eqref{eq:MHD:LLFsplit}}{>}
			\sum\limits_{i=1}^{\tt Q} { \omega _i  }
			\Bigg\{
			\frac{\alpha_{1}^{\tt LF}}{\Delta x} \bigg[  - \frac{ \bar B_1^i - \tilde B_1^i }{ \alpha_{1}^{\tt LF} } ({\bf v}^* \cdot {\bf B}^*)  \bigg]
			+
			\frac{\alpha_{2}^{\tt LF}}{\Delta y} \bigg[  - \frac{  \hat B_2^i - \check B_2^i }{ \alpha_{2}^{\tt LF} } ({\bf v}^* \cdot {\bf B}^*)  \bigg]
			\Bigg\}
			\\[2mm]
			& \quad = - ({\bf v}^* \cdot {\bf B}^*)
			\sum\limits_{i=1}^{\tt Q}   {{\omega _i}} \left( \frac{{{\bar B_1}^i - {\tilde B_1}^i} }{{\Delta x}} + \frac{ {\hat B_2}^i - {\check B_2}^i} {{\Delta y}} \right) \quad \overset{\eqref{eq:descrite2DDIV}}{=} 0. \nonumber
	\end{align*}}
	It follows that $\overline{\bf U} \cdot {\bf n}^* + \frac{|{\bf B}^*|^2}{2}>0$. Thus $\overline{\bf U} \in {\mathcal G}_* =  {\mathcal G}$.
\end{proof}




\begin{theorem}[3D generalized LF splitting]\label{theo:MHD:LLFsplit3D}
	If $\bar{\bf U}^i$, $\tilde{\bf U}^{i}$, $\hat{\bf U}^i$,
	$\check{\bf U}^i$, $\acute{\bf U}^i$, $\grave{\bf U}^i \in {\mathcal G}$
	for $i=1,\cdots, {\tt Q}$,
	and they satisfy the 3D ``discrete divergence-free'' condition
	\begin{equation*}
		\frac{{\sum\limits_{i=1}^{\tt Q} {{\omega _i}({\bar B_1}^i - {\tilde B_1}^i)} }}{{\Delta x}}
		+ \frac{{\sum\limits_{i=1}^{\tt Q} {{\omega _i}({\hat B_2}^i - {\check B_2}^i)} }}{{\Delta y}}
		+ \frac{{\sum\limits_{i=1}^{\tt Q} {{\omega _i}({\acute B_3}^i - {\grave B_3}^i)} }}{{\Delta { z}}}
		=0,
	\end{equation*}
	with $\Delta x, \Delta y, \Delta { z} >0$, and the sum of the positive numbers
	$\left\{\omega _i \right\}_{i=1}^{\tt Q}$ equals one,
	then for any $\alpha_{1}^{\tt LF}$, $\alpha_{2}^{\tt LF}$ and $\alpha_{3}^{\tt LF}$ satisfying
	$$ \alpha_{1}^{\tt LF} > \max_{1\le i\le {\tt Q}}  \alpha_1 ( \bar { \bf U }^i,  \tilde{
		\bf U}^i  )  ,\quad \alpha_{2}^{\tt LF} > \max_{1\le i\le {\tt Q}}
	\alpha_2 (  \hat { \bf U }^i , \check{\bf U}^i ),\quad \alpha_{3}^{\tt LF} > \max_{1\le i\le {\tt Q}}
	\alpha_3 (  \acute { \bf U }^i , \grave{\bf U}^i ),
	$$
	it holds $\overline{\bf U} \in {\mathcal G}$, where 
	\begin{equation*}
		\begin{split}
			\overline{\bf U} &:= \frac{1}{ 2\left( \frac{\alpha_{1}^{\tt LF} }{\Delta x} + \frac{\alpha_{2}^{\tt LF}}{\Delta y} + \frac{\alpha_{3}^{\tt LF}}{\Delta { z}} \right)}
			\sum\limits_{i=1}^{\tt Q} { \omega _i  }
			\bigg[
			\frac{\alpha_{1}^{\tt LF}}{\Delta x} \bigg(
			\bar { \bf U }^i -  \frac{ {\bf F}_1 ( \bar { \bf U}^i)  }{ \alpha_{1}^{\tt LF} }
			+\tilde{
				\bf U}^i +   \frac{ {\bf F}_1 ( \tilde { \bf U}^i)  }{ \alpha_{1}^{\tt LF} }
			\bigg)
			\\
			& +
			\frac{\alpha_{2}^{\tt LF}}{\Delta y} \bigg(
			\hat { \bf U }^i -  \frac { {\bf F}_2 ( \hat { \bf U}^i)  } { \alpha_{2}^{\tt LF} }
			+\check{
				\bf U}^i +  \frac { {\bf F}_2 ( \check { \bf U}^i)  } { \alpha_{2}^{\tt LF} }
			\bigg)
			+
			\frac{\alpha_{3}^{\tt LF}}{\Delta { z}} \bigg(
			\acute { \bf U }^i -  \frac { {\bf F}_3 ( \acute { \bf U}^i)  } { \alpha_{3}^{\tt LF} }
			+\grave{
				\bf U}^i +  \frac { {\bf F}_3 ( \grave { \bf U}^i)  } { \alpha_{3}^{\tt LF} }
			\bigg)
			\bigg].
		\end{split}
	\end{equation*}
\end{theorem}
\begin{proof}
	The proof is similar to that of Theorem \ref{theo:MHD:LLFsplit2D} and omitted here.
\end{proof}

\begin{remark}
	In the above
	generalized LF splitting properties,
	the convex combination $\overline{\bf U}$ depends on
	a number of strongly coupled states, making it extremely  difficult to check the admissibility of $\overline{\bf U}$. 
	Such difficulty is subtly overcame by using the inequality \eqref{eq:MHD:LLFsplit} under  the ``discrete divergence-free'' condition, which is
	an approximation to \eqref{eq:2D:BxBy0}. For example,
	the 2D ``discrete divergence-free'' condition \eqref{eq:descrite2DDIV}  can be derived
	by using some quadrature rule for the integrals at the left side of
	\begin{equation} \label{eq:div000}
	\begin{split}
	& \frac{1}{{\Delta x}} \left( \frac{1}{{\Delta y}}\int_{{\tt y}_0 }^{{\tt y}_0  + \Delta y}
	{ \big( B_1 ({\tt x}_0  + \Delta x,{\tt y}) - B_1 ({\tt x}_0 ,{\tt y}) \big) d{\tt y}} \right)  
	\\
	&
	+ \frac{1}{{\Delta y}} \left( \frac{1}{{\Delta x}}\int_{{\tt x}_0 }^{{\tt x}_0  + \Delta x}
	{  \big( B_2 ({\tt x},{\tt y}_0  + \Delta y)-B_2 ({\tt x},{\tt y}_0 )\big)   d{\tt x}} \right)   \\
	& = \frac{1}{{\Delta x\Delta y}}\int_{I} {\left( {\frac{{\partial B_1 }}{{\partial {\tt x}}} + \frac{{\partial B_2 }}{{\partial {\tt y}}}} \right)d{\tt x}d{\tt y}}=0,
	\end{split}
	\end{equation}
	where $({\tt x},{\tt y})=(x_1,x_2)$ and $I = [{\tt x}_0 ,{\tt x}_0  + \Delta x] \times [{\tt y}_0 ,{\tt y}_0  + \Delta y] $.
	It is worth emphasizing that, 
	like the necessity of the last term at the left-hand side of   
	\eqref{eq:MHD:LLFsplit}, 
	the proposed DDF condition 
	is necessary and crucial for the generalized 
	LF splitting properties. 
	Without this condition, those properties do not hold in general, even if $\alpha_i$ is replaced with $\chi \alpha_i$ or 
	$\chi a_i$ for any constant $\chi \ge 1$, see the proof of Theorem \ref{theo:counterEx}. 
\end{remark}

The above generalized LF splitting properties 
are important tools in 
analyzing PP schemes on uniform Cartesian  meshes
if the numerical flux is taken as the LF flux 
\begin{equation}\label{eq:LFflux}
\hat {\bf F}_\ell ( {\bf U}^- , {\bf U}^+ ) = \frac{1}{2} \Big( {\bf F}_\ell (  {\bf U}^- ) + {\bf F}_\ell ( {\bf U}^+ ) - 
\alpha_{\ell ,n}^{\tt LF} (  {\bf U}^+ -  {\bf U}^- ) \Big),\quad \ell=1,\cdots,d.
\end{equation}
Here $\{\alpha_{\ell,n}^{\tt LF}\}$ denote the numerical viscosity parameters  specified at the $n$-th discretized time level.
The extension of the above results on non-uniform or unstructured meshes will be presented in 
a separate paper.  

\section{One-dimensional positivity-preserving schemes}\label{sec:1Dpcp}

This section applies the above theories to 
study the provably PP schemes with the LF flux \eqref{eq:LFflux} for 
the 
system \eqref{eq:MHD} in one dimension. 
In 1D, the divergence-free
condition \eqref{eq:2D:BxBy0} and the fifth equation in \eqref{eq:MHD} yield that $B_1(x_1,t)\equiv {\rm constant}$ (denoted by ${\tt B}_{\tt const}$) for all $x_1$ and $t \ge 0$.

To avoid confusing subscripts, we will use the symbol $\tt x$ to represent the variable $x_1$ in \eqref{eq:MHD}.
Assume that the spatial domain is divided into uniform cells $\{ I_j=({\tt x}_{j-\frac{1}{2}},{\tt x}_{j+\frac{1}{2}}) \}$,
with a constant spatial step-size $\Delta x$.
And the time interval is divided into the mesh $\{t_0=0, t_{n+1}=t_n+\Delta t_{n}, n\geq 0\}$
with the time step-size $\Delta t_{n}$ determined by some CFL condition. 
Let 
$\bar {\bf U}_j^n $ denote the numerical cell-averaged approximation of the exact solution ${\bf U}({\tt x},t)$ over $I_j$ at $t=t_n$. 
Assume the discrete initial data $\bar {\bf U}_j^0\in {\mathcal G}$. 
A scheme is defined to be PP if its numerical solution $\bar {\bf U}_j^n$ always stays at ${\mathcal G}$.

\subsection{First-order scheme}

The 1D first-order LF scheme reads
\begin{equation}\label{eq:1DMHD:LFscheme}
\bar {\bf U}_j^{n+1} = \bar {\bf U}_j^n - \frac{\Delta t_n}{\Delta { x}} \Big( \hat {\bf F}_1 ( \bar {\bf U}_j^n ,\bar {\bf U}_{j+1}^n) - \hat {\bf F}_1 ( \bar {\bf U}_{j-1}^n, \bar {\bf U}_j^n )  \Big) ,
\end{equation}
where the numerical flux $\hat {\bf F}_1 (\cdot,\cdot) $ is defined by \eqref{eq:LFflux}.

A surprising discovery is that 
the LF scheme \eqref{eq:1DMHD:LFscheme} with  
a standard parameter $\alpha_{1,n}^{\tt LF}= \max_j {\mathscr{R}}_1 ( \bar {\bf U}_{j}^n ) $ (although works well in most cases)   
is not always PP regardless of how small the CFL number is. However, if the parameter $\alpha_{1,n}^{\tt LF}$ in \eqref{eq:LFflux} satisfies 
\begin{equation}\label{eq:Lxa1}
\alpha_{1,n}^{\tt LF} > \max_{j} \alpha_1 ( \bar {\bf U}_{j+1}^n, \bar {\bf U}_{j-1}^n ),
\end{equation} 
then we can rigorously prove that  
the scheme \eqref{eq:1DMHD:LFscheme} is PP when the CFL number is less than one.  
These results are shown the following two theorems. 
We remark that the lower bound given in \eqref{eq:Lxa1} is acceptable in comparison with the standard parameter $\max_{j} {\mathscr{R}}_1 ({\bf U}_{j}^n)$, because 
one can derive from Proposition \ref{prop:alpha} that 
$$
\max_{j} \alpha_1 ( \bar {\bf U}_{j+1}^n, \bar {\bf U}_{j-1}^n )
< 2 \max_{j} {\mathscr{R}}_1 ({\bf U}_{j}^n), 
$$
and for smooth problems, $
\max_{j} \alpha_1 ( \bar {\bf U}_{j+1}^n, \bar {\bf U}_{j-1}^n ) < \max_{j} {\mathscr{R}}_1 ({\bf U}_{j}^n)
+ {\mathcal O}( \Delta x ).
$

\begin{theorem}\label{eq:1DnotPP}
	Assume that $\bar {\bf U}_j^0 \in{\mathcal G}$ and $\bar B_{1,j}^0 = {\tt B}_{\tt const}$ for all $j$. 
	Let the parameter $\alpha_{1,n}^{\tt LF}=\max_j {\mathscr{R}}_1 ( \bar {\bf U}_{j}^n ) $, and 
	$$\Delta t_n ={\tt C}  \frac{ \Delta x}{ \alpha_{1,n}^{\tt LF} }, $$
	where ${\tt C} $ is the CFL number. For any constant ${\tt C} >0$,  
	the scheme \eqref{eq:1DMHD:LFscheme} is not PP. 
\end{theorem}

\begin{proof}
	We prove it by contradiction. 
	Assume that there exists a CFL number ${\tt C}>0$, such that the scheme \eqref{eq:1DMHD:LFscheme} is PP.
	We consider the ideal gases with $\gamma=1.4$, and the following (admissible) data 
	\begin{equation}\label{eq:exampleUUU}
	\bar{\bf U}_k^n = 
	\begin{cases}
	\big( \frac{8}{25},~0,~-\frac{8}{25},~0,~{\tt B}_{\tt const},~10,~0,~\frac{2504}{25} + \frac{5 {\tt p}}{2} \big)^\top, & k\le j-1,
	\\[2mm]
	\big( \frac12,~\frac32,~-2,~0,~{\tt B}_{\tt const},~8,~0,~\frac{353}{4} 
	+  \frac{5 {\tt p}}{2} \big)^\top, & k=j,
	\\[2mm]
	\big( \frac15,~0,~\frac15,~0,~{\tt B}_{\tt const},~5,~0,~\frac{313}{5} + \frac{5 {\tt p}}{2} \big)^\top, & k \ge j+1,
	\end{cases}
	\end{equation}  
	where 
	${\tt B}_{\tt const}=10$ and ${\tt p}>0$. 
	For any ${\tt p} \in \big(0,\frac1{800}\big)$, we have 
	$$
	\alpha_{1,n}^{\tt LF}  = \max_k {\mathscr{R}}_1 ( \bar {\bf U}_{k}^n ) 
	= \frac{\sqrt{5}}{4} \Big( 7{\tt p} + 10^3 + \sqrt{ 49 {\tt p}^2 + 10^6 } \Big)^{\frac12},
	$$ 
	and the state $\bar {\bf U}^{n+1}_j$ computed by \eqref{eq:1DMHD:LFscheme} depends on $\tt p$, specifically, 
	\begin{equation*}
		\begin{aligned}
			\bar {\bf U}^{n+1}_j  & =  \Bigg( 
			\frac12 - \frac{6 {\tt C}}{25},~
			\frac{ 3 ( 1 - {\tt C} ) }{2} + \frac{  75 {\tt C}    }{ 4 \alpha_{1,n}^{\tt LF}  },~
			{\tt C} \Big( \frac{97}{50} - \frac{25} { \alpha_{1,n}^{\tt LF} } \Big)-2,~0,~10,\\
			&  \qquad 
			8+{\tt C} \Big( \frac{  10 }{ \alpha_{1,n}^{\tt LF} } - \frac12 \Big),~0,~
			\frac{ 5 {\tt p} }{2} + \frac{ 353 }{4}
			- \frac{ 687 {\tt C} }{100} + 
			\frac{75 {\tt C}}{ \alpha_{1,n}^{\tt LF} }
			\Bigg)^\top =: {\bf U} ({\tt p}). 
		\end{aligned}
	\end{equation*}
	By assumption, we have 
	${\bf U} ({\tt p}) \in {\mathcal G}$. 
	This yields $0<{\tt C}<\frac{25}{12}$,  
	and ${\mathcal E} ( {\bf U} ({\tt p}) )>0$ for any ${\tt p} \in (0,\frac{1}{800})$.
	The continuity of ${\mathcal E} ( {\bf U})$ with respect to $\bf U$ on $\mathbb{R}^+\times \mathbb{R}^7$ 
	implies that
	$$
	0 \le
	\mathop {\lim }\limits_{{\tt p} \to  0^+ } {\mathcal E} ( {\bf U} ({\tt p}) ) =  {\mathcal E} \Big( \mathop {\lim }\limits_{{\tt p} \to  0^+ } {\bf U} ({\tt p}) \Big)
	= \frac{3 {\tt C} }{400} \times   \frac{  8 { \tt C }^2 + 75 { \tt C} - 200  }{ 25 - 12 {\tt C} } < 0,
	$$
	which is a contradiction. Thus the assumption is incorrect, and the proof is completed. 
\end{proof}


\begin{theorem}\label{theo:1DMHD:LFscheme}
	Assume that $\bar {\bf U}_j^0 \in{\mathcal G}$ and $\bar B_{1,j}^0 = {\tt B}_{\tt const}$ for all $j$, and
	the parameter $\alpha_{1,n}^{\tt LF}$ satisfies \eqref{eq:Lxa1}. 
	Then the state $\bar {\bf U}_j^n$, computed by the scheme \eqref{eq:1DMHD:LFscheme} under the CFL condition
	\begin{equation}\label{eq:CFL:LF}
	0< \alpha_{1,n}^{\tt LF} \Delta t_n / \Delta x \le 1,
	\end{equation}
	belongs to ${\mathcal G}$ and satisfies $\bar B_{1,j}^n = {\tt B}_{\tt const} $ for all $j$
	and $n\in {\mathbb{N}}$.
\end{theorem}

\begin{proof}
	Here the induction argument is used for the time level number $n$.
	It is obvious that the conclusion holds for $n=0$ under the hypothesis on the initial data.
	We now assume that $\bar {\bf U}_j^n\in {\mathcal G}$ with $\bar B_{1,j}^n = {\tt B}_{\tt const} $ for all $j$,
	and check whether the conclusion holds for $n+1$. 
	For the numerical flux in \eqref{eq:LFflux}, the fifth equation in \eqref{eq:1DMHD:LFscheme} gives
	\begin{equation*}
		\bar B_{1,j}^{n+1} = \bar B_{1,j}^{n} - \frac{\lambda}{2}  \big(  2 \bar B_{1,j}^{n} - \bar B_{1,j+1}^{n} - \bar B_{1,j-1}^{n} \big)
		= {\tt B}_{\tt const},
	\end{equation*}
	for all $j$, where $\lambda = \alpha_{1,n}^{\tt LF}  \Delta t_n / \Delta x \in (0,1]$ due to \eqref{eq:CFL:LF}. 
	We rewrite the scheme \eqref{eq:1DMHD:LFscheme} as 
	\begin{equation*}
		\bar {\bf U}_j^{n+1} = (1-\lambda) \bar {\bf U}_j^n  + \lambda {\bf \Xi},
	\end{equation*}
	with 
	$$
	{\bf \Xi}:=\frac{1}{2 }\bigg( \bar {\bf U}_{j+1}^n - \frac{  {\bf F}_1( \bar {\bf U}_{j+1}^n) }{ \alpha_{1,n}^{\tt LF} } +
	\bar {\bf U}_{j-1}^n + \frac{ {\bf F}_1 ( \bar {\bf U}_{j-1}^n )} {\alpha_{1,n}^{\tt LF}}\bigg).
	$$
	Under the induction hypothesis $\bar {\bf U}_{j-1}^n,\bar {\bf U}_{j+1}^n \in {\mathcal G}$ and $\bar B_{1,j-1}^{n}=\bar B_{1,j+1}^{n}$,
	we conclude that ${\bf \Xi} \in {\mathcal G}$ by the generalized LF splitting property in Theorem \ref{theo:MHD:LLFsplit1D}.
	The convexity of $\mathcal G$ further yields $\bar {\bf U}_j^{n+1} \in {\mathcal G}$.
	The proof is completed. 
\end{proof}

\begin{remark}
	If the condition \eqref{eq:CFL:LF} is enhanced to $0< \alpha_{1,n}^{\tt LF} \Delta t_n / \Delta x < 1$, then Theorem \ref{theo:1DMHD:LFscheme} holds for all  $\alpha_{1,n}^{\tt LF} \ge \max_{j} \alpha_1 ( \bar {\bf U}_{j+1}^n, \bar {\bf U}_{j-1}^n )$, by Lemma \ref{theo:MHD:convex}. 
	It is similar for the following Theorems \ref{thm:PP:1DMHD}, \ref{theo:2DMHD:LFscheme}, 
	\ref{theo:FullPP:LFscheme}, \ref{thm:PP:2DMHD}, and will not be repeated. 
\end{remark}



\subsection{High-order schemes}\label{sec:High1D}

We now study the provably PP high-order schemes for 1D MHD equations \eqref{eq:MHD}. With the provenly PP LF scheme \eqref{eq:1DMHD:LFscheme} as 
building block, any high-order finite difference schemes can be 
modified to be PP by a limiter \cite{Christlieb}. 
The following PP analysis is focused on finite volume and DG schemes. The considered 1D DG schemes are similar to those 
in \cite{cheng} but with a different viscosity parameter 
in the LF flux so that the PP property can be rigorously proved in our case. 

For the moment, we use the forward Euler method for time discretization, while high-order time discretization will be discussed later.
We consider the high-order finite volume schemes as well as the scheme satisfied by
the cell averages of a discontinuous Galerkin (DG) method, 
which have the following form
\begin{equation}\label{eq:1DMHD:cellaverage}
\bar {\bf U}_j^{n+1} = \bar {\bf U}_j^{n} - \frac{\Delta t_n}{\Delta x}
\Big(  \hat {\bf F}_1 ( {\bf U}_{j+ \frac{1}{2}}^-, {\bf U}_{j+ \frac{1}{2}}^+ )
- \hat {\bf F}_1 ( {\bf U}_{j- \frac{1}{2}}^-, {\bf U}_{j-\frac{1}{2}}^+)
\Big) ,
\end{equation}
where $\hat {\bf F}_1 ( \cdot,\cdot )$ is taken the LF flux defined in \eqref{eq:LFflux}.
The quantities ${\bf U}_{j + \frac{1}{2}}^-$ and ${\bf U}_{j + \frac{1}{2}}^+$ are the high-order approximations 
of the point values ${\bf U}\big( {\tt x}_{j + \frac{1}{2}} ,t_n \big)$ within the cells $I_j$ and $I_{j+1}$, respectively, 
computed by
\begin{equation}\label{eq:DG1Dvalues}
{\bf U}_{j + \frac{1}{2}}^- = {\bf U}_j^n \big( {\tt x}_{j + \frac{1}{2}}-0 \big), \quad {\bf U}_{j + \frac{1}{2}}^+ = {\bf U}_{j+1}^n \big( {\tt x}_{j + \frac{1}{2}}+0 \big),
\end{equation}
where the polynomial function ${\bf U}_j^n({\tt x})$ is with the cell-averaged value of $\bar {\bf U}_j^n$,
approximates ${\bf U}( {\tt x},t_n)$ within the cell $I_j$, and is either reconstructed in the finite volume methods from $\{\bar {\bf U}_j^n\}$ or directly evolved in the DG methods
with degree ${\tt K} \ge 1$.
The evolution equations for the high-order ``moments'' of ${\bf U}_j^n({\tt x})$ in the DG methods are omitted because we are only concerned with the PP property of the schemes here.

Generally the high-order scheme \eqref{eq:1DMHD:cellaverage} is not PP.
As proved in the following theorem,
the scheme \eqref{eq:1DMHD:cellaverage} becomes PP if 
${\bf U}_{j + \frac{1}{2}}^\pm$ are computed by \eqref{eq:DG1Dvalues} with ${\bf U}^n_j({\tt x})$ satisfying
\begin{align}\label{eq:1DDG:con1}
	&
	B_{1,j+\frac12}^{\pm} = {\tt B}_{\tt const}  ,\quad \forall j,  
	\\ \label{eq:1DDG:con2}
	& {\bf U}_j^n ( \hat {\tt x}_j^{(\mu)} ) \in {\mathcal G}, \quad \forall \mu \in \{1,2,\cdots,{\tt L}\}, ~\forall j,
\end{align}
and $\alpha_{1,n}^{\tt LF}$ satisfies \eqref{eq:LxaH}. 
Here $\{ \hat {\tt x}_j^{(\mu)} \}_{\mu=1}^{ {\tt L}}$ are the {\tt L}-point Gauss-Lobatto quadrature nodes in the interval $I_j$,
whose associated quadrature weights are denoted by $\{\hat \omega_\mu\}_{\mu=1} ^{\tt L}$ with $\sum_{\mu=1}^{\tt L} \hat\omega_\mu = 1$.
We require $2{\tt L}-3\ge {\tt K}$ such that the algebraic precision of corresponding quadrature  is at least $\tt K$, e.g., taking $\tt L$ as the integral part of  $\frac{{\tt K}+3}{2}$.

\begin{theorem} \label{thm:PP:1DMHD}
	If the polynomial vectors $\{{\bf U}^n_j({\tt x})\}$ satisfy \eqref{eq:1DDG:con1}--\eqref{eq:1DDG:con2}, and the parameter $\alpha_{1,n}^{\tt LF}$
	in \eqref{eq:LFflux} satisfies 
	\begin{equation}\label{eq:LxaH}
	\alpha_{1,n}^{\tt LF} > \max_{j} \alpha_1 ( {\bf U}_{j+\frac12}^\pm,  {\bf U}_{j-\frac12}^\pm ),
	\end{equation}
	then the high-order scheme \eqref{eq:1DMHD:cellaverage} is PP under the CFL condition
	\begin{equation}\label{eq:CFL:1DMHD}
	0< \alpha_{1,n}^{\tt LF}  \Delta t_n/ \Delta x \le \hat \omega_1.
	\end{equation}
\end{theorem}

\begin{proof}
	The exactness of the $\tt L$-point Gauss-Lobatto quadrature rule for the polynomials of degree $\tt K$ yields
	$$
	\bar{\bf U}_j^n = \frac{1}{\Delta x} \int_{I_j} {\bf U}_j^n ({\tt x}) d{\tt x} = \sum \limits_{\mu=1}^{\tt L} \hat \omega_\mu {\bf U}_j^n (\hat {\tt x}_j^{(\mu)} ).
	$$
	Noting $\hat \omega_1 = \hat \omega_{\tt L}$ and $\hat{\tt x}_j^{1,{\tt L}}={\tt x}_{j\mp\frac12}$, we can then rewrite the scheme \eqref{eq:1DMHD:cellaverage} into the convex combination form
	\begin{align} \label{eq:1DMHD:convexsplit}
		\bar{\bf U}_j^{n+1} =
		\sum \limits_{\mu=2}^{{\tt L}-1} \hat \omega_\mu {\bf U}_j^n ( \hat {\tt x}_j^{(\mu)} )  
		+  (\hat \omega_1-\lambda) \left( {\bf U}_{j-\frac{1}{2}}^+ +   {\bf U}_{j+\frac{1}{2}}^- \right)
		+ \lambda {\bf \Xi}_- + \lambda  {\bf \Xi}_+,
	\end{align}
	where $\lambda = \alpha_{1,n}^{\tt LF} \Delta t_n/\Delta x \in (0,\hat \omega_1 ]$, and
	\begin{align*}
		& {\bf \Xi}_\pm = \frac{1}{2} \left( {\bf U}_{j+\frac{1}{2}}^\pm - \frac{ {\bf F}_1 ( {\bf U}_{j+\frac{1}{2}}^\pm ) } {  \alpha_{1,n}^{\tt LF} }
		+ {\bf U}_{j-\frac{1}{2}}^\pm +  \frac{ {\bf F}_1 ( {\bf U}_{j-\frac{1}{2}}^\pm ) } {  \alpha_{1,n}^{\tt LF} }    \right).
	\end{align*}
	The condition \eqref{eq:1DDG:con1} and \eqref{eq:LxaH} yield ${\bf \Xi}_\pm \in {\mathcal G}$
	by the generalized LF splitting property in Theorem \ref{theo:MHD:LLFsplit1D}.
	We therefore have $\bar{\bf U}_j^{n+1} \in {\mathcal G} $ from \eqref{eq:1DMHD:convexsplit}
	by the convexity of $\mathcal G$.
\end{proof}

\begin{remark}
	The condition \eqref{eq:1DDG:con1} is easily ensured in practice,  since the exact solution $B_1(x_1,t)\equiv {\tt B}_{\tt const}$ and
	the flux for $B_1$ is zero. While the condition \eqref{eq:1DDG:con2} can be enforced by a simple scaling limiting procedure,
	which was well designed in \cite{cheng} by extending the 
	techniques in \cite{zhang2010,zhang2010b}.
	The details of the procedure are omitted here.
\end{remark}




%
%
%

The above analysis is focused on first-order time discretization. 
Actually it is also valid for the high-order explicit time discretization 
using strong stability preserving (SSP)  methods  \cite{Gottlieb2001,Gottlieb2005,Gottlieb2009}.
This is because of the convexity of 
$\mathcal G$, as well as the fact that a SSP method is certain convex combination of the forward Euler method.

\section{Two-dimensional positivity-preserving schemes}\label{sec:2Dpcp}

This section discusses positivity-preserving (PP) schemes for the MHD system \eqref{eq:MHD} in two dimensions ($d=2$).
The extension of our analysis to 3D case ($d=3$) is straightforward and displayed in Appendix \ref{sec:3D}. 
Our analysis will reveal 
that the PP property of conservative multi-dimensional MHD schemes is strongly connected with 
a discrete divergence-free condition on the numerical magnetic field. 

For convenience, 
the symbols $({\tt x},{\tt y})$ are used to denote the variables $(x_1,x_2)$ in \eqref{eq:MHD}. 
Assume that the 2D spatial domain is divided into a uniform rectangular mesh with cells $\big\{I_{ij}=({\tt x}_{i-\frac{1}{2}},{\tt x}_{i+\frac{1}{2}})\times
({\tt y}_{j-\frac{1}{2}},{\tt y}_{j+\frac{1}{2}}) \big\}$. 
The spatial step-sizes in ${\tt x},{\tt y}$ directions are denoted by
$\Delta x,\Delta y$ respectively. The time interval is also divided into the mesh $\{t_0=0, t_{n+1}=t_n+\Delta t_{n}, n\geq 0\}$
with the time step size $\Delta t_{n}$ determined by the CFL condition. We  use $\bar {\bf U}_{ij}^n $ 
to denote the numerical approximation to the cell-averaged value of the exact solution over $I_{ij}$ 
at time $t_n$.
We aim at seeking numerical schemes whose solution  
$\bar {\bf U}_{ij}^n$ 
is preserved in $\mathcal G$.

\subsection{First-order scheme} \label{sec:2D:FirstOrder}
The 2D first-order LF scheme reads
\begin{equation} \label{eq:2DMHD:LFscheme}
\begin{split}
\bar {\bf U}_{ij}^{n+1} = \bar {\bf U}_{ij}^n - \frac{\Delta t_n}{\Delta x} \Big( \hat {\bf F}_{1,i+\frac12,j} - \hat {\bf F}_{1,i-\frac12,j}\Big) 
- \frac{\Delta t_n}{\Delta y} \Big( \hat {\bf F}_{2,i,j+\frac12} - \hat {\bf F}_{2,i,j-\frac12}  \Big),
\end{split}
\end{equation}
where 
$\hat {\bf F}_{1,i+\frac12,j}=  \hat {\bf F}_1 ( \bar {\bf U}_{ij}^n ,\bar {\bf U}_{i+1,j}^n)$, 
$\hat {\bf F}_{2,i,j+\frac12} = \hat {\bf F}_2 ( \bar {\bf U}_{ij}^n ,\bar {\bf U}_{i,j+1}^n) $, and $\hat {\bf F}_\ell (\cdot,\cdot), \ell=1,2,$ are the LF fluxes in \eqref{eq:LFflux}.

As mentioned in \cite{Christlieb}, there was still no rigorous proof that the LF
scheme \eqref{eq:2DMHD:LFscheme} or any other first-order scheme is PP in the multi-dimensional cases.
For the ideal MHD with the EOS \eqref{eq:EOS},
it seems natural to conjecture \cite{cheng} that 
\begin{equation}\label{eq:conjectureCheng}
\mbox{
	given $\bar {\bf U}_{ij}^n \in {\mathcal G}~~\forall i,j$, then $\bar {\bf U}_{ij}^{n+1}$ computed from \eqref{eq:2DMHD:LFscheme} always
	belongs to $\mathcal G$,
}
\end{equation}
under certain CFL condition (e.g., the CFL number is less than 0.5).
If \eqref{eq:conjectureCheng} holds true, it would be important and very useful for developing PP high-order schemes \cite{cheng,Christlieb2016,Christlieb} for \eqref{eq:MHD}.
Unfortunately, the following theorem shows that 
\eqref{eq:conjectureCheng} does not always 
hold, 
no matter 
how small the specified CFL number is, 
and even if 
the parameter $\alpha_{\ell,n}^{\tt LF} $ is taken as  $\chi \max_{ij}{ {\mathscr{R}}_\ell (\bar{\bf U}^n_{ij})}$   
with any given constant $\chi \ge 1$. 
(Note that 
increasing numerical viscosity can usually enhance 
the robustness of a LF scheme and increase the possibility 
of achieving PP property, and   
$\alpha_{\ell,n}^{\tt LF} = \chi \max_{ij}{ {\mathscr{R}}_\ell (\bar{\bf U}^n_{ij})}$ corresponds 
to the $\chi$ times larger numerical viscosity in comparison 
with the standard one.)



\begin{theorem}\label{theo:counterEx}
	Let $\alpha_{\ell,n}^{\tt LF} = \chi \max_{ij}{ {\mathscr{R}}_\ell (\bar{\bf U}^n_{ij})}$ with the 
	constant $\chi \ge 1$, and 
	$$\Delta t_n = \frac{ {\tt C} }{ \alpha_{1,n}^{\tt LF} /\Delta x + \alpha_{2,n}^{\tt LF}  / \Delta y }, $$
	where ${\tt C}>0$ is the CFL number. 
	For any given constants $\chi$ and $\tt C$, 
	there always exists a set of admissible states 
	$\{  \bar {\bf U}_{ij}^{n},\forall i,j\}$ 
	such that the solution $\bar {\bf U}_{ij}^{n+1}$ 
	of \eqref{eq:2DMHD:LFscheme} does not belong to $\mathcal G$. 
	In other words, for any given $\chi$ and $\tt C$, 
	the admissibility of $\{  \bar {\bf U}_{ij}^{n}, \forall i,j \}$ does not always guarantee that $\bar {\bf U}_{ij}^{n+1} \in {\mathcal G}$, $\forall i,j$. 
\end{theorem}

\begin{proof}
	We prove it by contradiction. 
	Assume that there exists 
	a constant $\chi\ge 1$ and a CFL number ${\tt C}>0$, 
	such that $\bar {\bf U}_{ij}^n \in {\mathcal G},~\forall i,j$ 
	always ensure $\bar {\bf U}_{ij}^{n+1} \in {\mathcal G},~\forall i,j$. 
	Consider the ideal gases and a special set of admissible states 
	\begin{equation}\label{eq:counterEx2D}
	{\small 
		\bar{\bf U}_{k,m}^n = 
		\begin{cases}
		\left(1,~1,~0,~0,~1,~0,~0,~
		\frac{\tt p}{\gamma-1}+1\right)^\top, \qquad (k,m)=(i-1,j), 
		\\[2mm]
		\left(1,~1,~0,~0,~1+ \epsilon,~0,~0,~\frac{\tt p}{\gamma-1}+\frac{1+(1+\epsilon)^2}2 \right)^\top, \qquad (k,m)=(i+1,j),
		\\[2mm]
		\left(1, \frac{4 \chi + \epsilon }{4 \chi},~0,~0,~1+\frac{\epsilon}{2},~0,~0,~
		\frac{\tt p}{\gamma-1} + \frac{ (4 \chi + \epsilon)^2 }{32 \chi^2} 
		+ \frac{(\epsilon+2)^2}{8} 
		\right)^\top,~{\rm otherwise},
		\end{cases}}
	\end{equation}
	where ${\tt p} >0$ and $\epsilon >0$. For  $\forall {\tt p} \in \big(0,\frac{1}{\gamma}\big)$ and $\forall \epsilon \in \big(0,\frac{1}{\chi}\big)$, we have  
	$
	\alpha_{1,n}^{\tt LF}  
	= 
	= \chi (2+\epsilon)$,  
	$\alpha_{2,n}^{\tt LF}  
	= \chi \sqrt{\gamma {\tt p} + (1+\epsilon)^2}.
	$ 
	Hence 
	$
	\lambda_1 ({\tt p},\epsilon) := \alpha_{1,n}^{\tt LF} \frac{\Delta t_n}{\Delta x} 
	= \frac{ {\tt C } \Delta y ( 2 + \epsilon)  }{ \Delta y (2 + \epsilon) + \Delta x \sqrt{\gamma {\tt p} + (1+\epsilon)^2} }.
	$
	Substituting \eqref{eq:counterEx2D} into \eqref{eq:2DMHD:LFscheme} gives 
	\begin{equation*}
		\begin{aligned}
			\bar {\bf U}_{ij}^{n+1} &= \bigg(
			1, \frac{4 \chi + \epsilon }{4 \chi},~0,~0,~1+\frac{\epsilon}{2},~0,~0,~		 
			\lambda_1 ({\tt p},\epsilon) \times \frac{3+(1+\epsilon)^2}{4}
			\\
			& \qquad 
			+ \Big(1- \lambda_1 ({\tt p},\epsilon) \Big) 
			\Big(  \frac{ (4 \chi + \epsilon)^2 }{32 \chi^2} 
			+ \frac{(\epsilon+2)^2}{8}  \Big) + \frac{\tt p}{\gamma-1} 	
			\bigg)^\top =: {\bf U} ({\tt p},\epsilon) .
		\end{aligned}
	\end{equation*}
	By assumption we have   
	${\bf U} ({\tt p},\epsilon) \in {\mathcal G}$, and ${\mathcal E} ( {\bf U} ({\tt p},\epsilon) )>0$, for any $ {\tt p} \in \big(0,\frac{1}{\gamma}\big),~ \epsilon \in \big(0,\frac{1}{\chi}\big)$.
	The continuity of ${\mathcal E} ( {\bf U})$ with respect to $\bf U$ on $\mathbb{R}^+\times \mathbb{R}^7$ further implies that
	$$
	0 \le
	\mathop {\lim }\limits_{{\tt p} \to  0^+ } {\mathcal E} ( {\bf U} ({\tt p},\epsilon) ) =  {\mathcal E} \Big( \mathop {\lim }\limits_{{\tt p} \to  0^+ } {\bf U} ({\tt p},\epsilon) \Big)
	= - \frac{  \epsilon {\tt C} (2+\epsilon) ( 8 \chi + \epsilon - 4 \epsilon \chi^2 ) }{32 \chi^2 \big( 2+ \epsilon + (1+\epsilon) {\Delta x}/{\Delta y} \big)} < 0,
	$$
	which is a contradiction. Thus the assumption is incorrect, and the proof is completed.  
\end{proof}

\begin{remark}\label{rem:1DdisBnotPCP} 
	The proof of Theorem \ref{theo:counterEx} also implies that, for any specified CFL number,
	the 1D LF scheme \eqref{eq:1DMHD:LFscheme} is not always PP when $B_1$ is piecewise
	constant. 
\end{remark}


Inspired by Theorem \ref{theo:counterEx}, we conjecture that,  
to fully ensure the admissibility of $\bar {\bf U}_{ij}^{n+1}$, 
additional condition is required for the states $\{\bar {\bf U}_{i,j}^n, \bar {\bf U}_{i\pm1,j}^n, \bar {\bf U}_{i,j\pm1}^n\}$ except for their admissibility.
Such additional  necessary condition 
should be a divergence-free condition in discrete sense for 
$\{\bar{\bf B}_{ij}^n\}$,
whose importance for robust simulations has been widely realized. 
The following analysis confirms that a discrete divergence-free (DDF) condition does 
play an important role in achieving the PP property.

If the states $\{\bar {\bf U}_{i,j}^n\}$ are all admissible and
satisfy the following DDF condition
\begin{equation}\label{eq:DisDivB}
\mbox{\rm div} _{ij} \bar {\bf B}^n := \frac{ \left( \bar  B_1\right)_{i+1,j}^n - \left( \bar  B_1 \right)_{i-1,j}^n } {2\Delta x} + \frac{ \left( \bar  B_2 \right)_{i,j+1}^n - \left( \bar B_2 \right)_{i,j-1}^n } {2\Delta y} = 0,
\end{equation}
then we can rigorously prove that the scheme \eqref{eq:2DMHD:LFscheme} preserves $\bar {\bf U}_{ij}^{n+1}\in {\mathcal G}$, by using the generalized LF splitting property in Theorem \ref{theo:MHD:LLFsplit2D}. 

\begin{theorem} \label{theo:2DMHD:LFscheme}
	If for all $i$ and $j$, $\bar {\bf U}_{ij}^n \in {\mathcal G}$ and satisfies the DDF  condition \eqref{eq:DisDivB},
	then the solution $ \bar {\bf U}_{ij}^{n+1}$ of \eqref{eq:2DMHD:LFscheme} always belongs to ${\mathcal G}$ under the CFL condition
	\begin{equation}\label{eq:CFL:LF2D}
	0< \frac{ \alpha_{1,n}^{\tt LF} \Delta t_n}{\Delta x} + \frac{ \alpha_{2,n}^{\tt LF} \Delta t_n}{\Delta y}  \le  1,
	\end{equation}
	where the parameters $\{\alpha_{\ell,n}^{\tt LF}\}$ satisfy
	\begin{equation}\label{eq:Lxa12}
	\alpha_{1,n}^{\tt LF} > \max_{i,j} \alpha_1 ( \bar {\bf U}_{i+1,j}^n, \bar {\bf U}_{i-1,j}^n ),\quad
	\alpha_{2,n}^{\tt LF} > \max_{i,j} \alpha_2 ( \bar {\bf U}_{i,j+1}^n, \bar {\bf U}_{i,j-1}^n ).
	\end{equation}
\end{theorem}

\begin{proof}
	Substituting  \eqref{eq:LFflux} into \eqref{eq:2DMHD:LFscheme} gives
	\begin{equation*}
		\bar {\bf U}_{ij}^{n+1} = \lambda {\bf \Xi} +  (1-\lambda) \bar {\bf U}_{ij}^n ,
	\end{equation*}
	where $\lambda := \Delta t_n \left( \frac{\alpha_{1,n}^{\tt LF}}{\Delta x} + \frac{\alpha_{2,n}^{\tt LF}}{\Delta y} \right) \in (0,1]$ by \eqref{eq:CFL:LF2D}, and 
	\begin{equation*}
		\begin{split}
			{\bf \Xi} & :=
			\frac{1}{ 2\left(   \frac{\alpha_{1,n}^{\tt LF}}{\Delta x} + \frac{\alpha_{2,n}^{\tt LF}}{\Delta y}   \right) }
			\Bigg[
			\frac{\alpha_{1,n}^{\tt LF}}{\Delta x}
			\left( \bar {\bf U}_{i+1,j}^n - \frac{ {\bf F}_1( \bar {\bf U}_{i+1,j}^n)}{ \alpha_{1,n}^{\tt LF} } +
			\bar {\bf U}_{i-1,j}^n + \frac{ {\bf F}_1( \bar {\bf U}_{i-1,j}^n) }{ \alpha_{1,n}^{\tt LF} } \right) \\
			& \quad + \frac{\alpha_{2,n}^{\tt LF}}{\Delta y}  \left( \bar {\bf U}_{i,j+1}^n - \frac{ {\bf F}_2( \bar {\bf U}_{i,j+1}^n)}{ \alpha_{2,n}^{\tt LF} } +
			\bar {\bf U}_{i,j-1}^n + \frac{ {\bf F}_2( \bar {\bf U}_{i,j-1}^n) }{ \alpha_{2,n}^{\tt LF} } \right) \Bigg].
		\end{split}
	\end{equation*}
	Using the condition \eqref{eq:DisDivB} and
	Theorem \ref{theo:MHD:LLFsplit2D} gives ${\bf \Xi} \in {\mathcal G}$.
	The convexity of $\mathcal G$ further yields $\bar {\bf U}_{ij}^{n+1} \in {\mathcal G}$. The proof is completed.
\end{proof}

\begin{remark}
	The data in \eqref{eq:counterEx2D} satisfies 
	$
	{\rm div}_{ij} \bar {\bf B}^{n}  = \frac{\epsilon}{2\Delta x} > 0, 
	$
	which can be very small when $0<\epsilon \ll 1$. 
	Therefore, from the proof of Theorem \ref{theo:counterEx}, 
	we conclude that 
	violating the condition \eqref{eq:DisDivB} slightly could lead to 
	inadmissible solution of the scheme \eqref{eq:2DMHD:LFscheme}, 
	if the pressure is sufficiently low. This demonstrates the importance 
	of \eqref{eq:DisDivB}. 
\end{remark}

We now discuss whether the LF scheme \eqref{eq:2DMHD:LFscheme} preserves the DDF condition \eqref{eq:DisDivB}.

\begin{theorem} \label{theo:2DDivB:LFscheme}
	For the LF scheme \eqref{eq:2DMHD:LFscheme},
	the divergence error 
	$$ \varepsilon_{\infty}^n := \max_{ij} \left| {\rm div}_{ij} \bar {\bf B}^{n} \right| ,$$ 
	does not grow with $n$
	under the condition \eqref{eq:CFL:LF2D}.
	Moreover, 
	$\{\bar {\bf U}_{ij}^n\}$ 
	satisfy \eqref{eq:DisDivB}
	for all $i,j$ and $n \in \mathbb{N}$, if \eqref{eq:DisDivB} holds for the discrete initial data $\{\bar {\bf U}_{ij}^0\}$.
\end{theorem}

\begin{proof}
	Using the linearity of the operator $\mbox{div}_{ij}$, one can deduce from \eqref{eq:2DMHD:LFscheme} that  
	\begin{equation*} 
		\begin{split}
			\mbox{div}_{ij} \bar {\bf B}^{n+1} 
			=& ( 1 - \lambda  )  \mbox{div}_{ij} \bar {\bf B}^{n}
			+ \frac{ \lambda_1 }{2} (
			\mbox{div}_{i+1,j} \bar {\bf B}^{n}  +  \mbox{div}_{i-1,j} \bar {\bf B}^{n}
			)
			\\
			&
			+ \frac{\lambda_2 }{2} (
			\mbox{div}_{i,j+1} \bar {\bf B}^{n}     +  \mbox{div}_{i,j-1} \bar {\bf B}^{n}
			),
		\end{split}
	\end{equation*}
	where $\lambda_1 = \frac{ \alpha_{1,n}^{\tt LF} \Delta t_n } { \Delta x}, \lambda_2 = \frac{ \alpha_{2,n}^{\tt LF} \Delta t_n }{\Delta y}, \lambda = \lambda_1+ \lambda_2 \in (0,1] $. 
	It follows that 
	\begin{equation}\label{eq:notincease}
	\varepsilon_{\infty}^{n+1} \le  ( 1 - \lambda  ) \varepsilon_{\infty}^{n}  + \lambda_1  \varepsilon_{\infty}^{n}  +   \lambda_2 \varepsilon_{\infty}^{n}  = \varepsilon_{\infty}^{n}.
	\end{equation}
	This means $\varepsilon_{\infty}^{n} $
	does not grow with $n$. If $\varepsilon_{\infty}^{0} =0 $ for the discrete initial data $\{\bar {\bf U}_{ij}^0\}$, then $\varepsilon_{\infty}^{n} =0$ by  \eqref{eq:notincease}, i.e., the condition \eqref{eq:DisDivB} is satisfied for all $i,j$ and $n \in \mathbb{N}$.
\end{proof}

Finally, we obtain the first provably PP scheme for the 2D MHD system \eqref{eq:MHD}, as stated in the following theorem.

\begin{theorem} \label{theo:FullPP:LFscheme}
	Assume that the discrete initial data $\{\bar {\bf U}_{ij}^0\}$ are admissible and satisfy \eqref{eq:DisDivB}, which can be met by, e.g.,  the following second-order approximation
	\begin{align*}
		&\Big( \bar \rho_{ij}^0, \bar {\bf m}_{ij}^0, \left(\bar B_3\right)_{ij}^0, \overline {(\rho e)}_{ij}^0 \Big)
		= \frac{1}{\Delta x \Delta y} \iint_{I_{ij}} \big( \rho,{\bf m }, B_3, \rho e \big) ({\tt x},{\tt y},0) d{\tt x} d{\tt y},
		\\
		& \left( \bar B_1 \right)_{ij}^0 = \frac{1}{2 \Delta y} \int_{ {\tt y}_{j-1} }^{ {\tt y}_{j+1 } }  B_1( {\tt x}_i,{\tt y},0) d{\tt y},
		~
		\left( \bar B_2 \right)_{ij}^0 = \frac{1}{2 \Delta x} \int_{ {\tt x}_{i-1} }^{ {\tt x}_{i+1 } }  B_2({\tt x},{\tt y}_j,0) d {\tt x},
		\\
		&~\bar E_{ij}^0 = \overline {(\rho e )}_{ij}^0 + \frac12 \left( \frac{ |\bar{\bf m}_{ij}^0|^2}{\bar \rho_{ij}^0} + |\bar{\bf B}_{ij}^0|^2 \right).
	\end{align*}
	If the parameters $\{\alpha_{\ell,n}^{\tt LF}\}$ satisfy \eqref{eq:Lxa12}, then under the CFL condition \eqref{eq:CFL:LF2D}, 
	the LF scheme \eqref{eq:2DMHD:LFscheme} always preserve both $\bar {\bf U}_{ij}^{n+1} \in {\mathcal G}$ and \eqref{eq:DisDivB} for all $i$, $j$ and $n \in \mathbb{N}$.
\end{theorem}

\begin{proof}
	This is a direct consequence of Theorems \ref{theo:2DMHD:LFscheme} and \ref{theo:2DDivB:LFscheme}.
\end{proof}


\subsection{High-order schemes}\label{sec:High2D}

This subsection discusses the provably PP high-order finite volume or DG schemes for the 2D MHD equations \eqref{eq:MHD}.
We will focus on the first-order forward Euler method for time discretization, and our analysis also works for  
high-order explicit time discretization using the SSP methods \cite{Gottlieb2001,Gottlieb2005,Gottlieb2009}.

Towards achieving high-order [$({\tt K}+1)$-th order] spatial accuracy, the approximate solution polynomials ${\bf U}_{ij}^n ({\tt x},{\tt y})$ of degree $\tt K$
are also built usually, as approximation to the exact solution ${\bf U}({\tt x},{\tt y},t_n)$ within $I_{ij}$. Such polynomial vector ${\bf U}_{ij}^n ({\tt x},{\tt y})$
is, either reconstructed in the finite volume methods
from the cell averages $\{\bar {\bf U}_{ij}^n\}$ or evolved in the DG methods. Moreover, the cell average of ${\bf U}_{ij}^n({\tt x},{\tt y})$ over $I_{ij}$ is $\bar {\bf U}_{ij}^{n}$.

Let $\{  {\tt x}_i^{(\mu)} \}_{\mu=1}^{\tt Q}$ and $\{  {\tt y}_j^{(\mu)} \}_{\mu=1}^{\tt Q}$ denote the $\tt Q$-point Gauss quadrature nodes in the intervals $[ {\tt x}_{i-\frac12}, {\tt x}_{i+\frac12} ]$ and $[ {\tt y}_{j-\frac12}, {\tt y}_{j+\frac12} ]$, respectively, and $\{\omega_\mu\}_{\mu=1}^{\tt Q}$ be the associated weights satisfying
$\sum_{\mu=1}^{\tt Q} \omega_\mu = 1$.
With this quadrature rule for approximating the integrals of numerical fluxes on cell interfaces,
a finite volume scheme or discrete equation for the cell average in the DG method (see e.g., \cite{zhang2010b}) can be written as 
\begin{equation}\label{eq:2DMHD:cellaverage}
\begin{split}
\bar {\bf U}_{ij}^{n+1} & = \bar {\bf U}_{ij}^{n}
- \frac{\Delta t_n}{\Delta x} \sum\limits_{\mu =1}^{\tt Q}  \omega_\mu \left(
\hat {\bf F}_1( {\bf U}^{-,\mu}_{i+\frac{1}{2},j}, {\bf U}^{+,\mu}_{i+\frac{1}{2},j} ) -
\hat {\bf F}_1( {\bf U}^{-,\mu}_{i-\frac{1}{2},j}, {\bf U}^{+,\mu}_{i-\frac{1}{2},j})
\right)  \\
& \quad -  \frac{ \Delta t_n }{\Delta y} \sum\limits_{\mu =1}^{\tt Q} \omega_\mu \left(
\hat {\bf F}_2( {\bf U}^{\mu,-}_{i,j+\frac{1}{2}} , {\bf U}^{\mu,+}_{i,j+\frac{1}{2}}   ) -
\hat {\bf F}_2( {\bf U}^{\mu,-}_{i,j-\frac{1}{2}} , {\bf U}^{\mu,+}_{i,j-\frac{1}{2}}  )
\right),
\end{split}
\end{equation}
where $\hat {\bf F}_1$ and $\hat {\bf F}_2$ are the LF fluxes in \eqref{eq:LFflux}, and
the limiting values 
are given by
\begin{align*}
	&{\bf U}^{-,\mu}_{i+\frac{1}{2},j} = {\bf U}_{ij}^n ({\tt x}_{i+\frac12},{\tt y}_j^{(\mu)}),\qquad
	{\bf U}^{+,\mu}_{i-\frac{1}{2},j} = {\bf U}_{ij}^n ({\tt x}_{i-\frac12},{\tt y}_j^{(\mu)}),
	\\
	&{\bf U}^{\mu,-}_{i,j+\frac{1}{2}} = {\bf U}_{ij}^n ({\tt x}_i^{(\mu)},{\tt y}_{j+\frac12}),\qquad
	{\bf U}^{\mu,+}_{i,j-\frac{1}{2}} = {\bf U}_{ij}^n ({\tt x}_i^{(\mu)},{\tt y}_{j-\frac12}).
\end{align*}
For the accuracy requirement,  $\tt Q$ should satisfy:
${\tt Q} \ge {\tt K}+1$ for a $\mathbb{P}^{\tt K}$-based DG method, 
or ${\tt Q} \ge ({\tt K}+1)/2$ for a $({\tt K}+1)$-th order finite volume scheme.


We denote 
$$
\overline{(B_1)}_{i+\frac{1}{2},j}^{\mu} := \frac12 \left( (B_1)_{i+\frac{1}{2},j}^{-,\mu} + (B_1)_{i+\frac{1}{2},j}^{+,\mu} \right), 
\ \ 
\overline{ ( B_2) }_{i,j+\frac{1}{2}}^{\mu}  := \frac12 \left( ( B_2)_{i,j+\frac{1}{2}}^{\mu,-} + ( B_2)_{i,j+\frac{1}{2}}^{\mu,+} \right),
$$
and    
define the discrete divergences of the numerical magnetic field ${\bf B}^n( {\tt x},{\tt y})$ as 
\begin{align*}
	&
	{\rm div} _{ij} {\bf B}^n := \frac{\sum\limits_{\mu=1}^{\tt Q} \omega_\mu \left(   \overline{(B_1)}_{i+\frac{1}{2},j}^{\mu} 
		- \overline{(B_1)}_{i-\frac{1}{2},j}^{\mu}  \right)}{\Delta x}   + \frac{\sum \limits_{\mu=1}^{\tt Q} \omega_\mu \left(  \overline{ ( B_2)}_{i,j+\frac{1}{2}}^{\mu} 
		-  \overline{ ( B_2)}_{i,j-\frac{1}{2}}^{\mu}    \right)}{\Delta y} , 
\end{align*}
which is an approximation to the left side of \eqref{eq:div000} with   
$({\tt x}_0,{ \tt y}_0)$ taken as $({\tt x}_{i-\frac12},{\tt y}_{j-\frac12})$. 
Let $\{ \hat {\tt x}_i^{(\nu) }\}_{\nu=1} ^ {\tt L}$ and $\{ \hat {\tt y}_j^{(\nu)} \}_{\nu=1} ^{\tt L}$ be the $\tt L$-point Gauss-Lobatto quadrature nodes in the intervals
$[{\tt x}_{i-\frac{1}{2}},{\tt x}_{i+\frac{1}{2}}]$ and $[{\tt y}_{j-\frac{1}{2}},{\tt y}_{j+\frac{1}{2}} ]$ respectively, and
$ \{\hat \omega_\nu\}_{\nu=1} ^ {\tt L}$ be  associated weights satisfying $\sum_{\nu=1}^{\tt L} \hat\omega_\nu = 1$, where  ${\tt L}\ge \frac{{\tt K}+3}2$ such that the
associated quadrature has algebraic precision of at least degree ${\tt K}$.
Then we have the following sufficient
conditions for that
the high-order scheme \eqref{eq:2DMHD:cellaverage} is PP.

\begin{theorem} \label{thm:PP:2DMHD}
	If the polynomial vectors $\{{\bf U}_{ij}^n({\tt x},{\tt y})\}$ satisfy:
	\begin{align}\label{eq:DivB:cst12}
		&
		\mbox{\rm div} _{ij} {\bf B}^n  = 0, \quad \forall~i,j,
		\\
		&{\bf U}_{ij}^n ( \hat {\tt x}_i^{(\nu)},{\tt y}_j^{(\mu)} ),~{\bf U}_{ij}^n ( {\tt x}_i^{(\mu)},  \hat {\tt y}_j^{(\nu)} ) \in {\mathcal G},\quad \forall~i,j,\mu,\nu,
		\label{eq:2Dadmissiblity}
	\end{align}
	then the scheme \eqref{eq:2DMHD:cellaverage} always preserves $\bar{\bf U}_{ij}^{n+1} \in {\mathcal G}$ under the CFL condition
	\begin{equation}\label{eq:CFL:2DMHD}
	0< \frac{\alpha_{1,n}^{\tt LF} \Delta t_n}{\Delta x} + \frac{ \alpha_{2,n}^{\tt LF} \Delta t_n}{\Delta y}  \le \hat \omega_1,
	\end{equation}
	where the parameters $\{\alpha_{\ell,n}^{\tt LF}\}$ satisfy
	\begin{equation}\label{eq:2DhighLFpara}
	\alpha_{1,n}^{\tt LF} >
	\max_{ i,j,\mu}  \alpha_1 \big(  { \bf U }_{i+\frac12,j}^{\pm,\mu} ,  { \bf U }_{i-\frac12,j}^{\pm,\mu}  \big)
	,\quad \alpha_{2,n}^{\tt LF} >  \max_{ i,j,\mu}  \alpha_2 \big(  { \bf U }_{i,j+\frac12}^{\mu,\pm} ,  { \bf U }_{i,j-\frac12}^{\mu,\pm}  \big).
	\end{equation}
	
\end{theorem}

\begin{proof}
	Using the exactness of the Gauss-Lobatto quadrature rule with $\tt L$ nodes and the Gauss quadrature rule with $\tt Q$ nodes for the polynomials of degree $\tt K$,
	one can derive (cf. \cite{zhang2010b} for more details) that
	\begin{equation} \label{eq:U2Dsplit}
	\begin{split}
	\bar{\bf U}_{ij}^n
	&= \frac{\lambda_1}{\lambda}  \sum \limits_{\nu = 2}^{{\tt L}-1} \sum \limits_{\mu = 1}^{\tt Q}  \hat \omega_\nu \omega_\mu  {\bf U}_{ij}^n\big(\hat {\tt x}_i^{(\nu)},{\tt y}_j^{(\mu)}\big) + \frac{\lambda_2}{\lambda} \sum \limits_{\nu = 2}^{{\tt L}-1} \sum \limits_{\mu = 1}^{\tt Q}  \hat \omega_\nu \omega_\mu  {\bf U}_{ij}^n\big( {\tt x}_i^{(\mu)},\hat {\tt y}_j^{(\nu)} \big) \\
	&\quad + \frac{\lambda_1 \hat \omega_1}{\lambda} \sum \limits_{\mu = 1}^{\tt Q}  \omega_\mu \left( {\bf U}_{i-\frac{1}{2},j}^{+,\mu} +
	{\bf U}_{i+\frac{1}{2},j}^{-,\mu} \right)
	+ \frac{\lambda_2 \hat \omega_1}{\lambda} \sum \limits_{\mu = 1}^{\tt Q}  \omega_\mu \left(  {\bf U}_{i,j-\frac{1}{2}}^{\mu,+} +
	{\bf U}_{i,j+\frac{1}{2}}^{\mu,-} \right) ,
	\end{split}
	\end{equation}
	where $\hat \omega_1 = \hat \omega_{\tt L}$ is used, and $\lambda_1 = \frac{ \alpha_{1,n}^{\tt LF} \Delta t_n } { \Delta x}, \lambda_2 = \frac{ \alpha_{2,n}^{\tt LF} \Delta t_n }{\Delta y}, \lambda = \lambda_1+ \lambda_2 \in (0,\hat \omega_1] $ by \eqref{eq:CFL:2DMHD}. 
	After substituting \eqref{eq:LFflux} and \eqref{eq:U2Dsplit} into \eqref{eq:2DMHD:cellaverage}, we rewrite the scheme  \eqref{eq:2DMHD:cellaverage} by technical arrangement into the following convex combination form
	\begin{align}	\label{eq:2DMHD:split:proof}
		&\bar{\bf U}_{ij}^{n+1}
		= 
		\sum \limits_{\nu = 2}^{{\tt L}-1} \hat \omega_\nu {\bf \Xi}_\nu
		+ 
		2( \hat \omega_1 - \lambda )  {\bf \Xi}_{\tt L}  
		+  2	\lambda {\bf \Xi}_1  ,
	\end{align}
	where ${\bf \Xi}_1 = \frac12 \left( {\bf \Xi}_- +  {\bf \Xi}_+ \right)$, and 
	\begin{align*}
		&{\bf \Xi}_\nu
		=  
		\frac{\lambda_1}{\lambda}  \sum \limits_{\mu = 1}^{\tt Q} \omega_\mu  {\bf U}_{ij}^n \big(\hat {\tt x}_i^{(\nu)},{\tt y}_j^\mu\big) + \frac{\lambda_2}{\lambda}
		\sum \limits_{\mu = 1}^{\tt Q} \omega_\mu  {\bf U}_{ij}^n \big( {\tt x}_i^{(\mu)},\hat {\tt y}_j^{(\nu)} \big),\quad 2 \le \nu \le {\tt L}-1,
		\\ \nonumber
		\begin{split}
			&{\bf \Xi}_{\tt L} = 
			\frac{1}{ 2\lambda }
			\sum\limits_{\mu=1}^{\tt Q} { \omega _\mu  }
			\bigg(
			\lambda_1 \left(
			{ \bf U }_{i+\frac12,j}^{-,\mu} 
			+ { \bf U }_{i-\frac12,j}^{+,\mu}  
			\right)
			+ \lambda_2 \left(
			{ \bf U }_{i,j+\frac12}^{\mu,-} 
			+ { \bf U }_{i,j-\frac12}^{\mu,+}  
			\right)
			\bigg),
		\end{split}
		\\ \nonumber
		\begin{split}
			&
			{\bf \Xi}_\pm = 
			\frac{1}{ 2\left( \frac{\alpha_{1,n}^{\tt LF} }{\Delta x} + \frac{\alpha_{2,n}^{\tt LF}}{\Delta y} \right)}
			\sum\limits_{\mu=1}^{\tt Q} { \omega _\mu  }
			\left[
			\frac{\alpha_{1,n}^{\tt LF}}{\Delta x} \left(
			{ \bf U }_{i+\frac12,j}^{\pm,\mu} -  \frac{ {\bf F}_1 ( { \bf U }_{i+\frac12,j}^{\pm,\mu}  )  }{ \alpha_{1,n}^{\tt LF} }
			+ { \bf U }_{i-\frac12,j}^{\pm,\mu}  +   \frac{ {\bf F}_1 ( { \bf U }_{i-\frac12,j}^{\pm,\mu}  )  }{ \alpha_{1,n}^{\tt LF} }
			\right)\right.
			\\
			&\qquad + \left.
			\frac{\alpha_{2,n}^{\tt LF}}{\Delta y} \left(
			{ \bf U }_{i,j+\frac12}^{\mu,\pm} -  \frac { {\bf F}_2 (  { \bf U }_{i,j+\frac12}^{\mu,\pm}  )  } { \alpha_{2,n}^{\tt LF} }
			+ { \bf U }_{i,j-\frac12}^{\mu,\pm}  +  \frac { {\bf F}_2 (  { \bf U }_{i,j-\frac12}^{\mu,\pm} )  } { \alpha_{2,n}^{\tt LF} }
			\right)
			\right].
		\end{split}
	\end{align*}
	The condition \eqref{eq:2Dadmissiblity} implies ${\bf \Xi}_\nu \in {\mathcal G},~2 \le \nu \le {\tt L}$, because $\mathcal G$ is convex. 
	In order to show the admissibility of ${\bf \Xi}_1$ by using 
	Theorem \ref{theo:MHD:LLFsplit2D}, 
	one has to verify the corresponding discrete divergence-free condition, 
	which is found to be \eqref{eq:DivB:cst12}. 
	Hence ${\bf \Xi}_1 \in{\mathcal G}$. 
	This means the form \eqref{eq:2DMHD:split:proof} is a convex combination of the admissible states $\{{\bf \Xi}_k,1\le k \le {\tt L}\}$.
	It follows from the convexity of $\mathcal G$ that $ \bar{\bf U}_{ij}^{n+1} \in {\mathcal G}$.
	The proof is completed. 
\end{proof}

\begin{remark} \label{rem:wu_aaa}
	For some other hyperbolic systems such as 
	the Euler \cite{zhang2010b} and shallow water \cite{Xing2010} 
	equations, the condition \eqref{eq:2Dadmissiblity}  
	is sufficient to ensure the positivity of 2D high-order   
	schemes. 
	However, contrary to 
	the usual 
	expectation (e.g., \cite{cheng}),   
	the condition \eqref{eq:2Dadmissiblity}  is 
	not sufficient in the ideal MHD case, even if ${\bf B}^n_{ij}({\tt x},{\tt y})$ is locally divergence-free. 
	This is indicated by Theorem \ref{theo:counterEx} and confirmed by the numerical experiments in the Section \ref{sec:examples}, 
	and demonstrates the necessity 
	of \eqref{eq:DivB:cst12} to some extent. 
\end{remark}

\begin{remark} \label{rem:notDDF}
	In practice, the 
	condition \eqref{eq:2Dadmissiblity} 
	can be easily met via a simple scaling limiting procedure 
	\cite{cheng}. 
	It is not easy to meet \eqref{eq:DivB:cst12} because it 
	depends on the limiting values of the magnetic field calculated from the four neighboring cells 
	of  $I_{ij}$.
	If ${\bf B}^n({\tt x},{\tt y})$ is globally 
	divergence-free,
	i.e., locally divergence-free in each cell with
	normal magnetic component continuous across the cell interfaces,
	then by Green's theorem, \eqref{eq:DivB:cst12} is naturally satisfied. 
	However, the PP limiting technique with local scaling 
	may 
	destroy the globally divergence-free property of ${\bf B}^n({\tt x},{\tt y})$. 
	Hence, it is nontrivial and still open to design a limiting procedure 
	for the polynomials  
	$\{{\bf U}_{ij}^n({\tt x},{\tt y})\}$ which can enforce the conditions \eqref{eq:2Dadmissiblity} and \eqref{eq:DivB:cst12} at the same time. 
	As a continuation of this work, Ref. \cite{WuShu2018} reports 
	our achievement in developing   
	multi-dimensional probably PP high-order schemes via the discretization of symmetrizable  ideal MHD equations. 	
\end{remark}

We now derive a lower bound of the internal energy when the 
proposed DDF condition \eqref{eq:DivB:cst12} is not satisfied, to show that 
negative internal energy may be more easily computed in the cases with large $|{\bf v}\cdot {\bf B}|$ and large discrete divergence error.

\begin{theorem} \label{thm:PP:2DMHD:further}
	Assume that the polynomial vectors $\{{\bf U}_{ij}^n({\tt x},{\tt y})\}$ satisfy 
	\eqref{eq:2Dadmissiblity}, and the parameters $\{\alpha_{\ell,n}^{\tt LF}\}$ satisfy 
	\eqref{eq:2DhighLFpara}. 
	Then under the CFL condition \eqref{eq:CFL:2DMHD}, 
	the solution $\bar{\bf U}_{ij}^{n+1}$ 
	of the scheme \eqref{eq:2DMHD:cellaverage} satisfies  
	that $\bar \rho_{ij}^{n+1} > 0$, 
	and 
	\begin{equation}\label{eq:2222}
	{\mathcal E} ( \bar{\bf U}_{ij}^{n+1} ) > -\Delta t_n \big(\bar{\bf v}_{ij}^{n+1} \cdot 
	\bar{\bf B}_{ij}^{n+1} \big) {\rm div}_{ij}{\bf B}^n, 
	\end{equation}
	where the lower bound dominates the negativity of ${\mathcal E} ( \bar{\bf U}_{ij}^{n+1} )$, and  
	$\bar{\bf v}_{ij}^{n+1} :=\bar{\bf m}_{ij}^{n+1}/\bar{\rho}^{n+1}_{ij}$. 
\end{theorem}


\begin{proof}
	It is seen from \eqref{eq:2DMHD:split:proof} that 	
	$\bar \rho_{ij}^{n+1}$ is a convex combination of 
	the first components of ${\bf \Xi}_\nu, 1 \le \nu \le {\tt L}$, which are 
	all positive. Thus $\bar \rho_{ij}^{n+1} > 0$. For any ${\bf v}^*,{\bf B}^* \in \mathbb{R}^3$, 
	$$
	\bigg(  {\bf \Xi}_1 \cdot {\bf n}^* + \frac{|{\bf B}^*|^2}{2} \bigg) \times  2\left( \frac{\alpha_{1,n}^{\tt LF} }{\Delta x} + \frac{\alpha_{2,n}^{\tt LF}}{\Delta y} \right) 
	>  - ({\bf v}^* \cdot {\bf B}^*) {\rm div}_{ij}{\bf B}^n,
	$$ 
	whose derivation is similar to that of Theorem \ref{theo:MHD:LLFsplit2D}. 
	Because ${\bf \Xi}_\nu \in {\mathcal G},~2 \le \nu \le {\tt L}$, we deduce  
	from \eqref{eq:2DMHD:split:proof} that 
	{\small	
		\begin{align*}  
			&\bar{\bf U}_{ij}^{n+1} \cdot {\bf n}^* + \frac{|{\bf B}^*|^2}{2}
			=  
			\sum \limits_{\nu = 2}^{{\tt L}-1} \hat \omega_\nu 
			\bigg( {\bf \Xi}_\nu \cdot {\bf n}^* + \frac{|{\bf B}^*|^2}{2}  \bigg)
			\\
			&	\quad + 
			2( \hat \omega_1 - \lambda ) \bigg( {\bf \Xi}_{\tt L}  \cdot {\bf n}^* + \frac{|{\bf B}^*|^2}{2} \bigg)
			+  2	\lambda \bigg( {\bf \Xi}_1 \cdot {\bf n}^* + \frac{|{\bf B}^*|^2}{2} \bigg)
			\\
			& > 2	\lambda \bigg( {\bf \Xi}_1 \cdot {\bf n}^* + \frac{|{\bf B}^*|^2}{2} \bigg)
			> -  \Delta t_n ({\bf v}^* \cdot {\bf B}^*) {\rm div}_{ij}{\bf B}^n.
	\end{align*}}
	Taking ${\bf v}^*= \bar{\bf v}^{n+1}_{ij} $ 
	and ${\bf B}^* = \bar{\bf B}^{n+1}_{ij}$ gives \eqref{eq:2222}.
\end{proof}

\section{Numerical experiments}\label{sec:examples}
Several numerical examples are provided in this section  
to further confirm the above PP analysis.

\subsection{1D case} We first give several 1D numerical examples. 


\subsubsection{Simple example}
This is a simple example that one can verify by hand with a calculator. 
It is used to numerically 
confirm the conclusion in Theorem \ref{eq:1DnotPP}, 
and show that the 1D Lax-Friedrichs (LF) scheme 
with the standard numerical viscosity parameter is not PP in general. 
We consider the 1D data in \eqref{eq:exampleUUU} with 
$\gamma =1.4$ and ${\tt p}=10^{-5}$, and then verify the 
pressure of $\bar {\bf U}_{j}^{n+1}$ computed by 
the LF scheme \eqref{eq:1DMHD:LFscheme} with the standard parameter $\alpha_{1,n}^{\tt LF}=\max_j {\mathscr{R}}_1 ( \bar {\bf U}_{j}^n ) $. 
The pressure $\bar p_{j}^{n+1}$ obtained by using different CFL numbers ${\tt C} 
\in \{ 0.001,0.002,\cdots, 1 \}$ are displayed in the left figure of Fig. \ref{fig:valcex1}. 
It is seen that the LF scheme \eqref{eq:1DMHD:LFscheme} with the standard parameter 
fails to guarantee the positivity of pressure, even though very small CFL number is used. However, when $\alpha_{1,n}^{\tt LF}$ 
satisfies the proposed condition \eqref{eq:Lxa1}, as expected by Theorem \ref{theo:1DMHD:LFscheme}, the positivity is always preserved for any CFL number less than one, see the right figure of Fig. \ref{fig:valcex1}. 

\begin{figure}[htbp]
	\centering
	\includegraphics[width=0.43\textwidth]{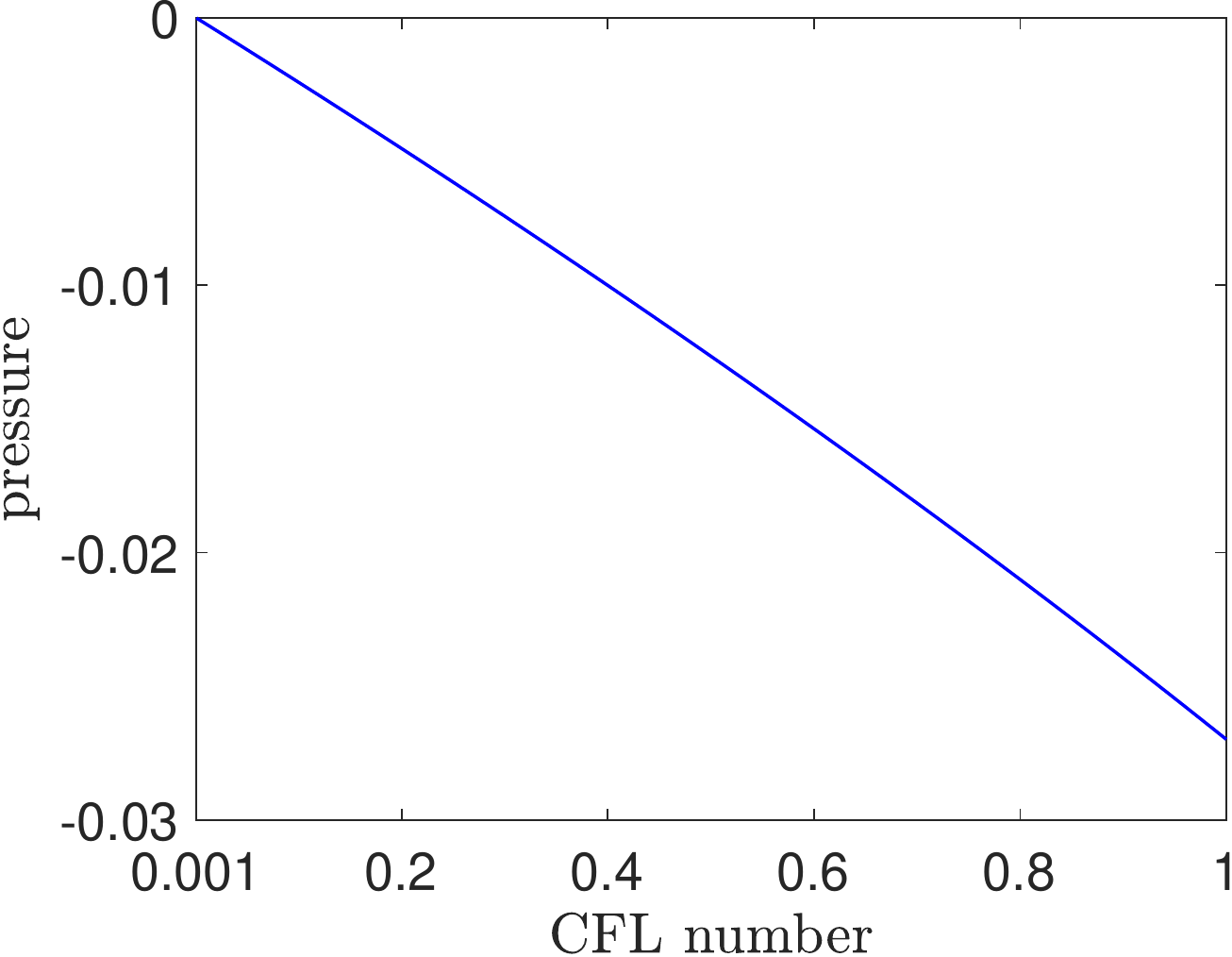}~~~~~~
	\includegraphics[width=0.43\textwidth]{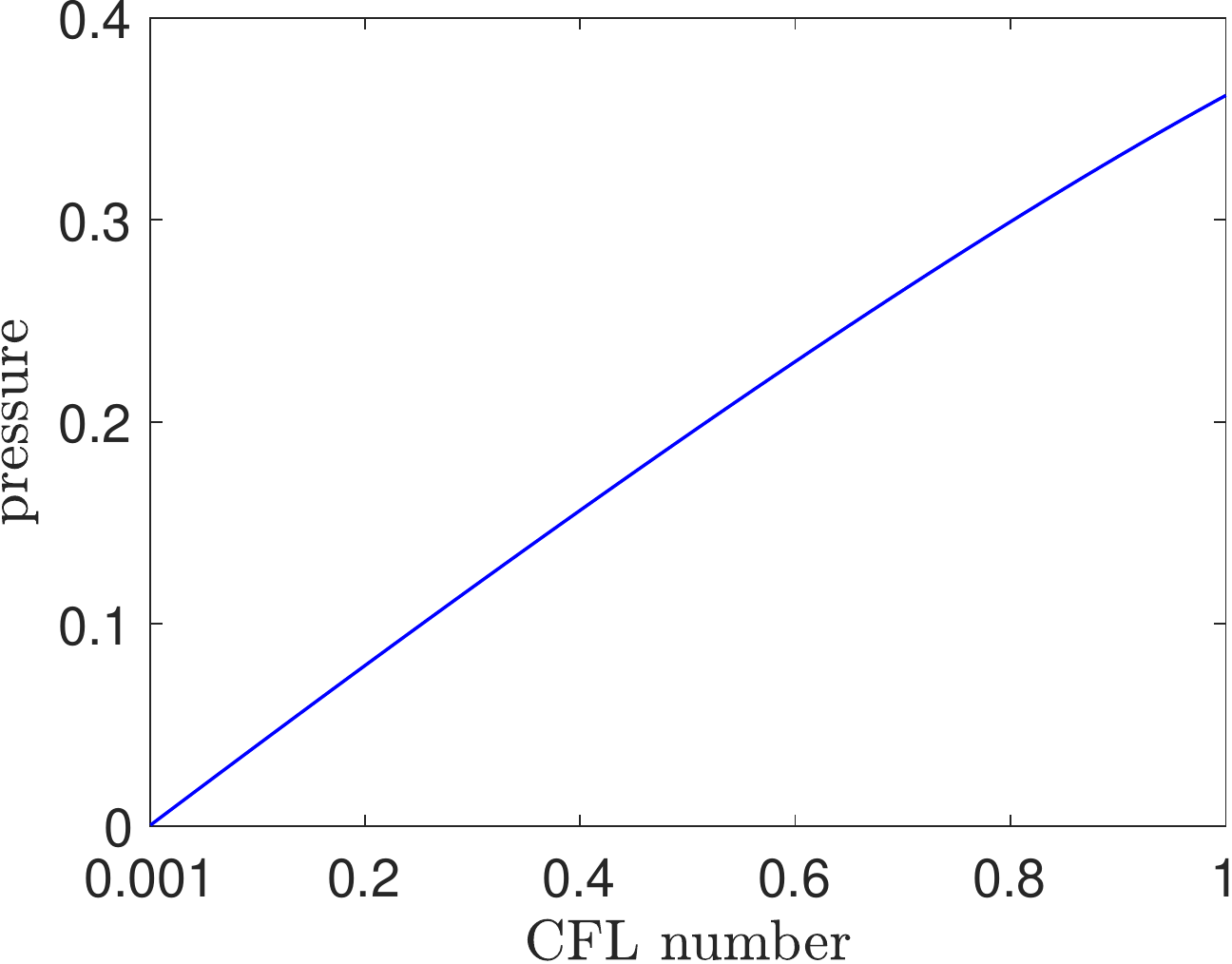}
	\caption{\small
		The pressure $\bar p_{j}^{n+1}$ obtained by the LF scheme \eqref{eq:1DMHD:LFscheme} with different parameter $\alpha_{1,n}^{\tt LF}$ and using different CFL numbers. Left: 
		$\alpha_{1,n}^{\tt LF}=\max_j {\mathscr{R}}_1 ( \bar {\bf U}_{j}^n ) $; right: 
		$\alpha_{1,n}^{\tt LF}=\max_{j} \alpha_1 \big( \bar {\bf U}_{j+1}^n, \bar {\bf U}_{j-1}^n;  
		\sqrt{ \bar{ \rho }_{j+1}^n } /{( \sqrt{ \bar{ \rho }_{j+1}^n } + 
			\sqrt{ \bar{ \rho }_{j-1}^n } )} \big)$.
	}\label{fig:valcex1}
\end{figure}

In the following, we conduct numerical experiments on several 1D MHD problems 
with low density, low pressure, strong discontinuity, and/or low plasma-beta 
$\beta := 2p/|{\bf B}|^2 $, 
to demonstrate the accuracy and robustness of the 1D provenly PP high-order methods. 
Without loss of generality, we take the third-order (${\mathbb P}^2$-based), 
discontinuous Galerkin (DG) method, together with the third-order explicit 
strong stability preserving (SSP) Runge-Kutta time discretization \cite{Gottlieb2009}, as our base scheme. 
The LF flux \eqref{eq:LFflux} is used with the 
numerical viscosity parameters satisfying the condition \eqref{eq:LxaH}. 
The PP limiter in \cite{cheng} is employed to enforce the condition 
\eqref{eq:1DDG:con2}. According to our analysis in Theorem \ref{thm:PP:1DMHD}, 
the resulting DG scheme is PP. 
Unless otherwise stated, all the computations are restricted to the ideal equation 
of state \eqref{eq:EOS}, 
and the CFL number is taken as 0.15.

\subsubsection{Accuracy test}

A smooth problem is tested to verify the accuracy of the third-order DG method. 
It is similar to the one simulated in \cite{zhang2010b} for testing the PP DG scheme for the Euler equations. The exact solution is given by 
\begin{equation*}
	(\rho,{\bf v},p,{\bf B}) ({\tt x},t)= 
	(1+0.99\sin({\tt x}-t),~1,~0,~0,~1,~0.1,~0,~0),\quad {\tt x}\in[0,2\pi],~t\ge 0,
\end{equation*}
which describes a MHD sine wave propagating with low density and $\gamma=1.4$. 
Table \ref{tab:1Dacc1} lists the numerical errors at $t=0.1$ 
in the numerical density and the corresponding convergence rates 
for the PP third-order DG method at different grid resolutions. 
The results show that the expected convergence order is achieved.

\begin{table}[htbp]\label{tab:1Dacc1}
	\centering
	\caption{\small 
		Numerical errors at $t=0.1$ in the density and corresponding convergence rates for
		the 1D PP third-order DG method at
		different grid resolutions.
	}
	\label{tab:order}
	\begin{tabular}{c|c|c|c|c|c|c}
		\hline
		~Mesh~&~$l^1$-error~& ~order~ &~$l^2$-error~&~order~&~$l^\infty$-error~&~order~ \\
		\hline
		$40$ &  2.1268e-4 & --          & 9.5354e-5 & -- & 5.9715e-5 & -- \\
		$80$ &  3.7004e-5 & 2.52   & 1.6502e-5&  2.53 & 1.0401e-5 & 2.52 \\
		$160$ &  5.1857e-6 & 2.84    & 2.3121e-6& 2.84 & 1.4582e-6 & 2.83\\
		$320$ & 6.6087e-7 & 2.97    & 2.9467e-7& 2.97 & 1.8587e-7 & 2.97 \\
		$640$ & 8.2817e-8 & 3.00 &  3.6926e-8& 3.00 & 2.3292e-8 & 3.00 \\
		$1280$ &  1.0358e-8 & 3.00  &4.6185e-9& 3.00 & 2.9133e-9 & 3.00\\
		\hline
	\end{tabular}
\end{table}

\subsubsection{Positivity-preserving tests}

Two extreme 1D Riemann problems are solved to verify the robustness and  
PP property of the PP third-order DG scheme.

The first is a 1D vacuum shock tube
problem \cite{Christlieb} with $\gamma=\frac53$ and the initial data given by 
\begin{equation*}
	(\rho,{\bf v},p,{\bf B}) ({\tt x},0)
	=\begin{cases}
		(10^{-12},~0,~0,~0,~10^{-12},~0,~0,~0), \quad & x<0,
		\\
		(1,~0,~0,~0,~0.5,~0,~1,~0), \quad & x>0.
	\end{cases}
\end{equation*}
We use this example to demonstrate that the PP DG scheme can handle extremely low density and pressure. 
The computational domain 
is taken as $[-0.5,0.5]$. 
Fig. \ref{fig:1DRP1} displays the density and pressure of the numerical solution on the mesh of $200$ cells as well as 
the highly resolved solution with $2000$ cells at $t=0.1$. 
In comparison with the results in \cite{Christlieb}, 
the low pressure and the low density are both captured correctly and well. 
The solutions of low
resolution and high resolution are in good agreement. 
The PP third-order DG method works very robustly during the whole simulation. 
If the PP limiter is not employed to enforce the condition 
\eqref{eq:1DDG:con2}, the method breaks down within a few time steps.

\begin{figure}[htbp]
	\centering
	\includegraphics[width=0.49\textwidth]{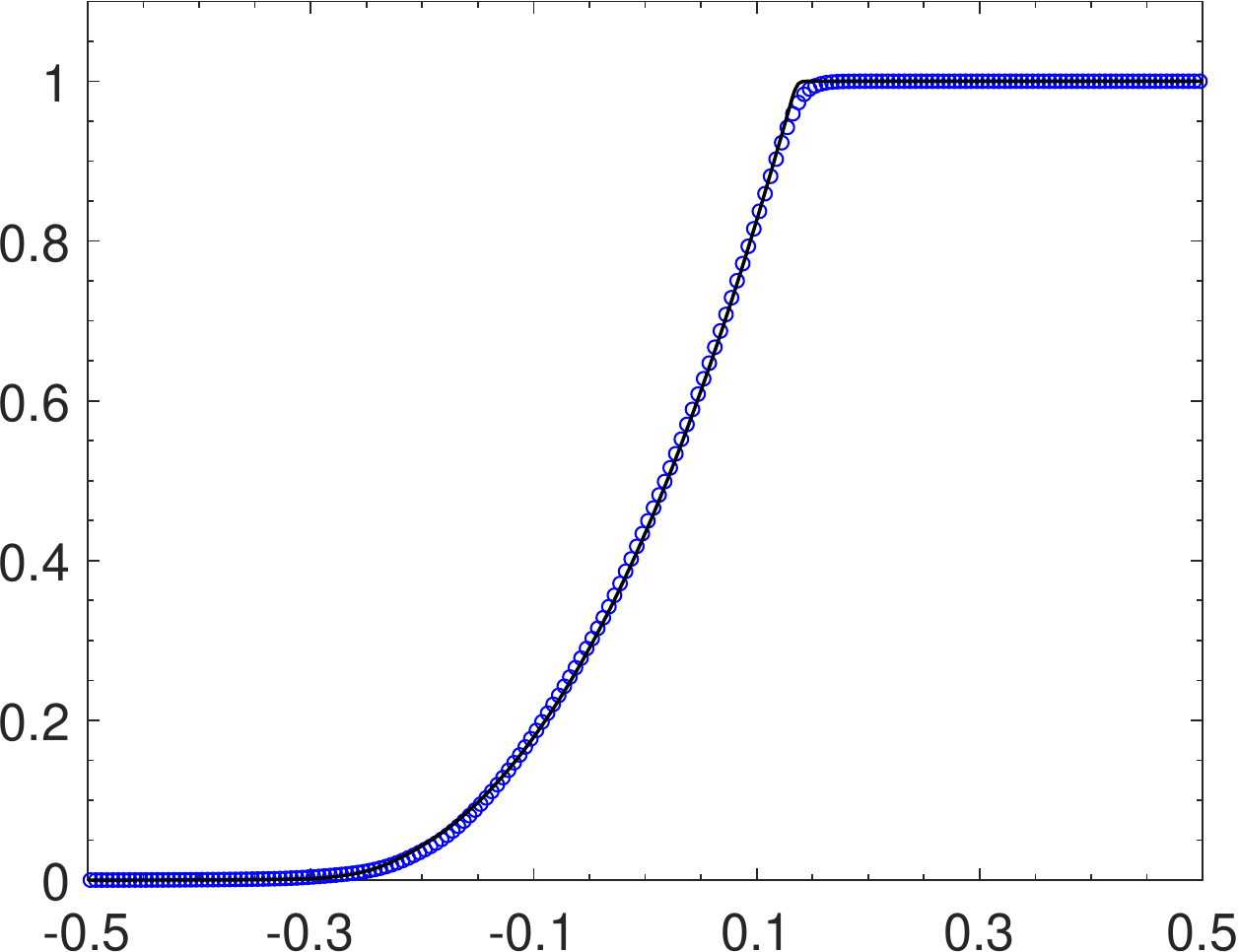}
	\includegraphics[width=0.49\textwidth]{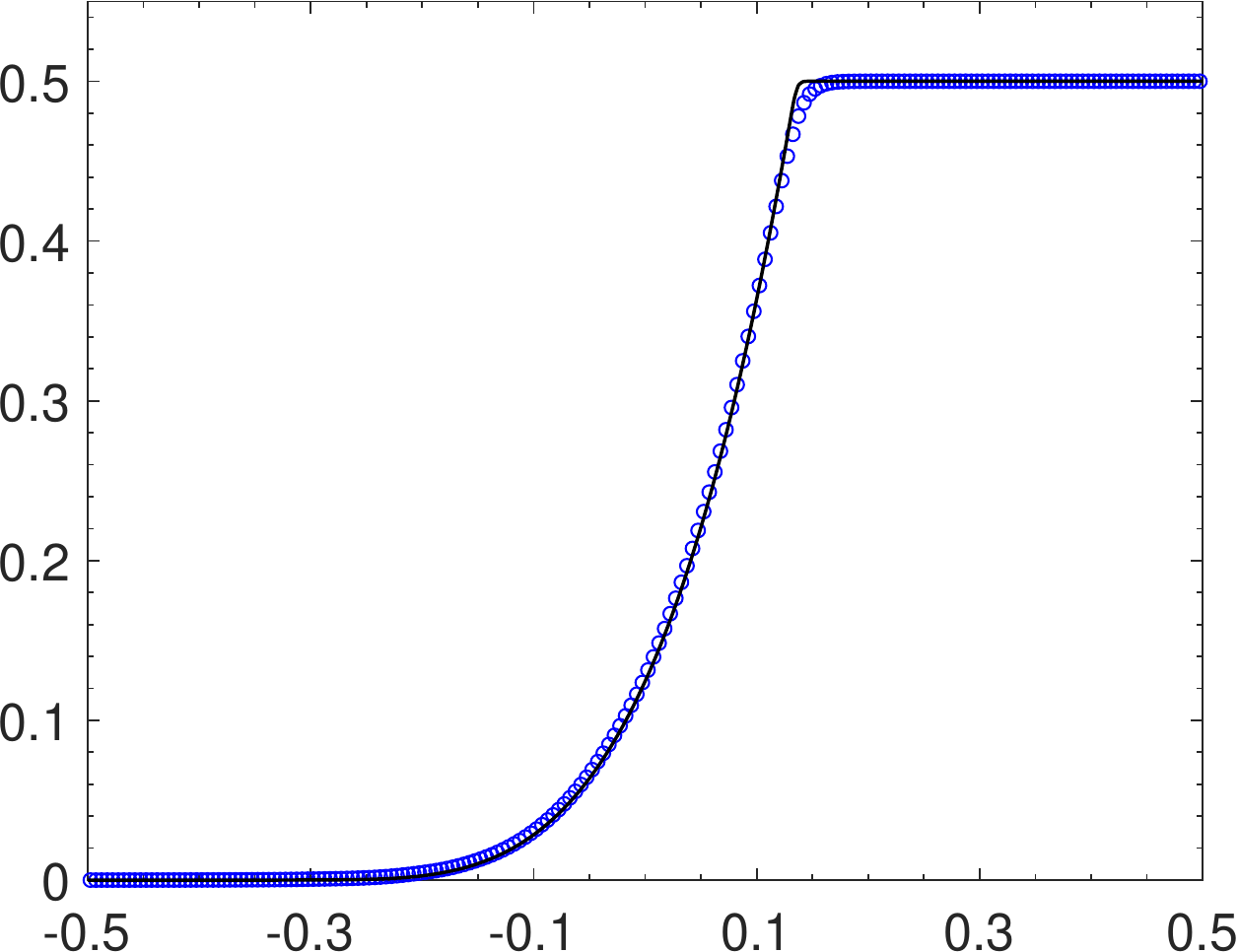}
	\caption{\small
		The density (left) and pressure (right) obtained by the PP third-order DG method 
		on the meshes of $200$ cells (symbols ``$\circ$'') and $2000$ cells (solid lines), respectively. 
	}\label{fig:1DRP1}
\end{figure}

The second Riemann problem is extended from the Leblanc problem \cite{zhang2010b} of gas dynamics by adding a strong magnetic field. The initial condition is  
\begin{equation*}
	(\rho,{\bf v},p,{\bf B}) ({\tt x},0)
	=\begin{cases}
		(2,~0,~0,~0,~10^{9},~0,~5000,~5000), \quad & x<0,
		\\
		(0.001,~0,~0,~0,~1,~0,~5000,~5000), \quad & x>0.
	\end{cases}
\end{equation*}
The adiabatic index $\gamma=1.4$, and the computational domain 
is $[-10,10]$. There exists a very large jump in the initial pressure, and the plasma-beta at the right state is extremely low 
($\beta = 4 \times 10^{-8}$). 
Successfully simulating this problem is a challenge.  
As the exact solution contains strong discontinuities, 
the WENO limiter \cite{Qiu2005} is implemented right before the PP limiting procedure 
with the aid of the local characteristic decomposition within the ``trouble'' cells adaptively detected by the indicator in \cite{Krivodonova}. 
To fully resolve the wave structure, a fine mesh is required for such test, see e.g., \cite{zhang2010b}. 
Fig. \ref{fig:1DRP2} gives the numerical results  
at $t=0.00003$ obtained by the PP third-order DG method using 
$3200$ cells and $10000$ cells, respectively. We observe 
that the strong discontinuities are well captured, and 
the low
resolution and high resolution are in good agreement. 
In this extreme test, it is also necessary to use the PP limiter to meet the 
condition \eqref{eq:1DDG:con2}, otherwise the DG method will break down quickly 
due to negative numerical pressure.

\begin{figure}[htbp]
	\centering
	\includegraphics[width=0.96\textwidth]{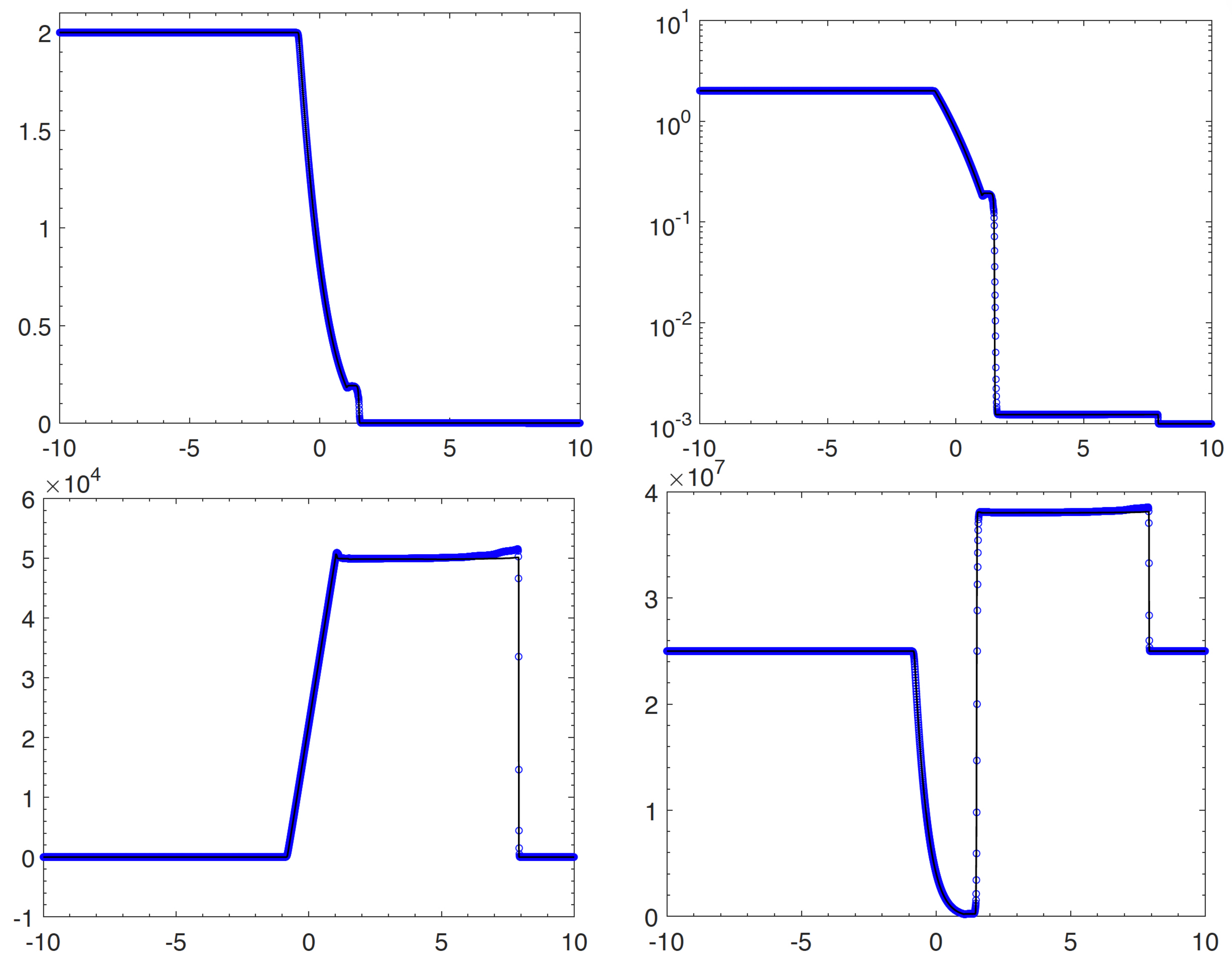}
	\caption{\small
		Numerical results at $t=0.00003$ obtained by the PP third-order DG method using $3200$ cells (symbols ``$\circ$'') and $10000$ cells (solid lines). 
		Top left: density; top right: log plot of density; 
		bottom left: velocity $v_1$; bottom right: magnetic pressure. 
	}\label{fig:1DRP2}
\end{figure}

\subsection{2D case}

We now present several 2D numerical examples to further 
confirm our theoretical analysis and the importance of the proposed 
discrete divergence-free (DDF) condition.

\subsubsection{Simple example}

This is a simple test which can be repeated easily by interested readers. 
We consider the 2D discrete data in \eqref{eq:counterEx2D} 
with $\gamma=1.4$, $\epsilon = 10^{-3}$ and ${\tt p}=10^{-8}$. The states in this data 
are admissible, and are slight perturbations of the constant state 
$(1,1,0,0,1,0,0,1+ 2.5 {\tt p} )^\top$ 
so that the proposed DDF condition \eqref{eq:DisDivB} is not satisfied. 
We then check the 
pressure $\bar p_{ij}^{n+1}$ computed by 
the 2D LF scheme \eqref{eq:2DMHD:LFscheme} with $\Delta x=\Delta y$ 
and different CFL numbers ${\tt C} \in \{0.01,0.02,\cdots,1\}$. 
The results are shown in Fig. \ref{fig:valcex2}, 
where two sets of parameters $\{\alpha_{\ell,n}^{\tt LF}\}$ are considered. 
It can be observed that, though the 
parameters $\{\alpha_{\ell,n}^{\tt LF}\}$ satisfy the condition \eqref{eq:Lxa12} 
or are even much larger, 
the admissibility of all the discrete states at the time level $n$ 
cannot ensure the positivity of numerical pressure at the next level. This  
confirms Theorem \ref{theo:counterEx} and the importance of DDF condition \eqref{eq:DisDivB}.  

\begin{figure}[htbp]
	\centering
	\includegraphics[width=0.43\textwidth]{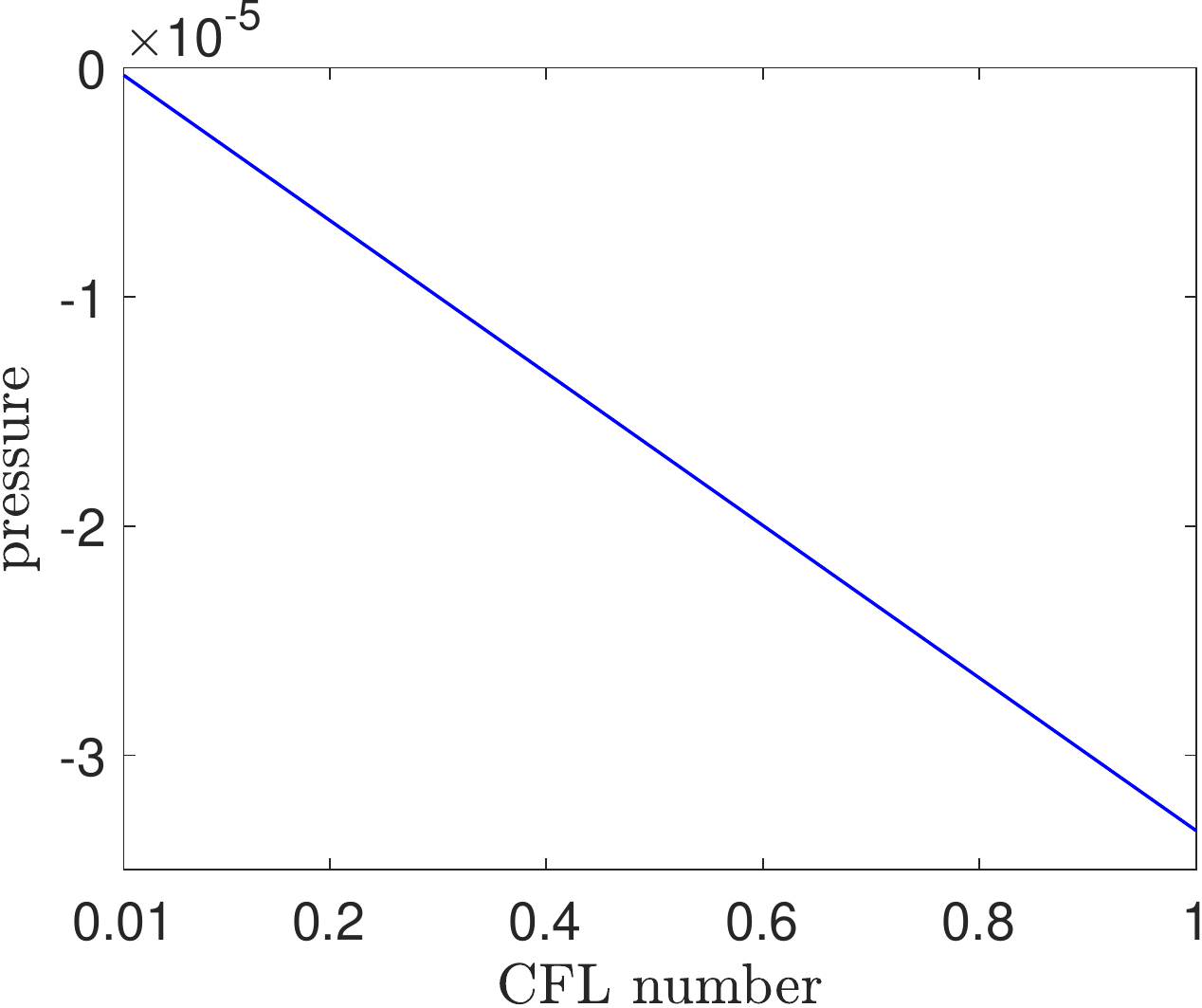}~~~~~
	\includegraphics[width=0.44\textwidth]{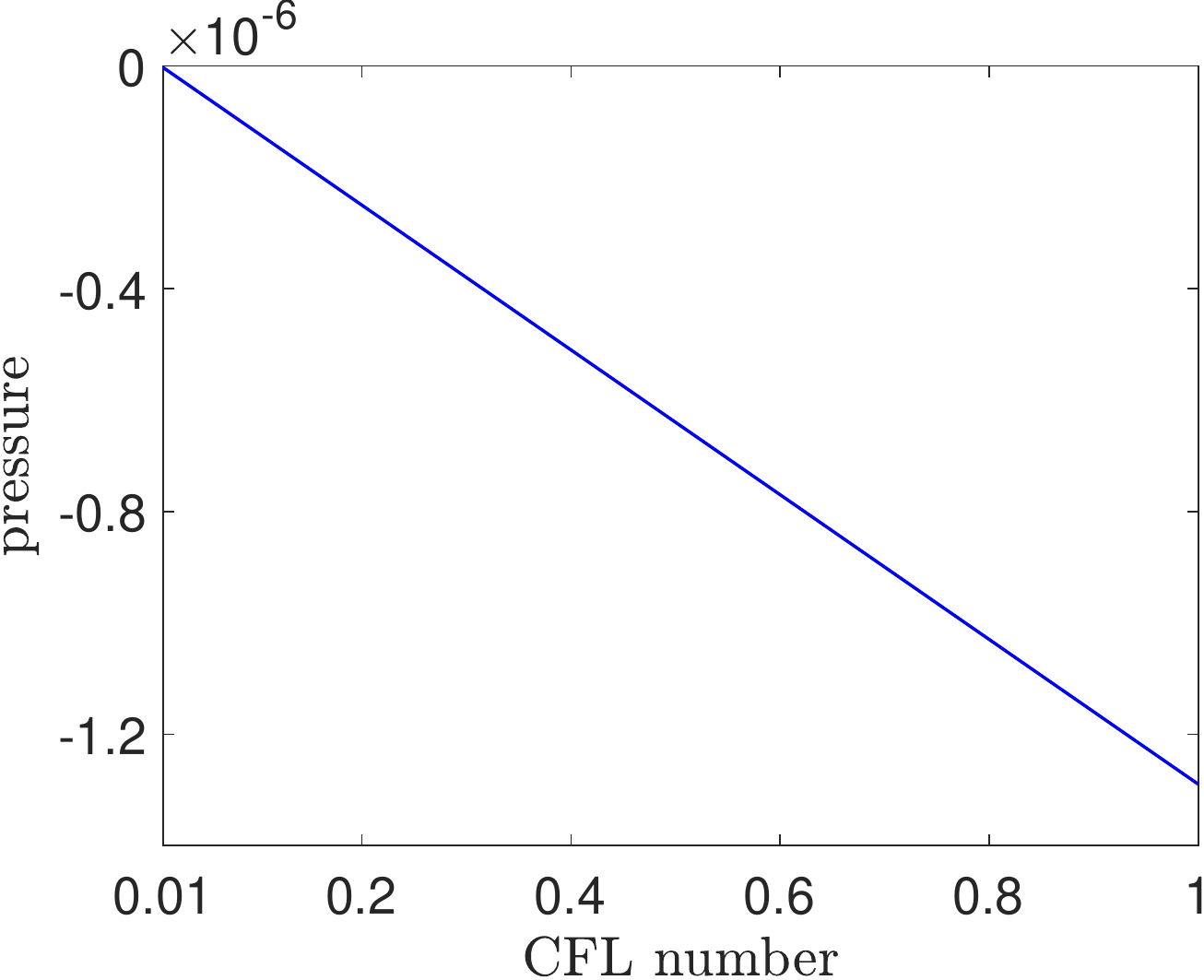}
	\caption{\small
		The pressure $\bar p_{ij}^{n+1}$ obtained by the 2D LF scheme 
		\eqref{eq:2DMHD:LFscheme} with larger numerical viscosity parameters and using different CFL numbers. Left: 
		$\alpha_{\ell,n}^{\tt LF} = 2 \max_{ij}{ {\mathscr{R}}_\ell (\bar{\bf U}^n_{ij})}$; right: 
		$\alpha_{\ell,n}^{\tt LF} = 50 \max_{ij}{ {\mathscr{R}}_\ell (\bar{\bf U}^n_{ij})}$.
	}\label{fig:valcex2}
\end{figure}

In the following, we consider several more practical examples to further verify  
our theoretical findings for 2D PP high-order schemes, and to
seek the numerical evidences for 
that only enforcing the condition \eqref{eq:2Dadmissiblity} 
is not sufficient to achieve PP high-order conservative scheme. 
To this end, we take 
the locally divergence-free DG methods \cite{Li2005}, together with the third-order SSP 
Runge-Kutta time discretization \cite{Gottlieb2009}, as the base schemes. 
We use the PP limiter in \cite{cheng} to enforce the condition \eqref{eq:2Dadmissiblity}. 
As we have discussed in Section \ref{sec:High2D}, 
the resulting high-order DG schemes do  
not always preserve the positivity of pressure under all circumstances, 
because the proposed DDF condition \eqref{eq:DivB:cst12} is not always satisfied 
(although the numerical magnetic field is locally divergence-free 
within each cell). 
It is worth mentioning that the locally divergence-free property and 
the PP limiter can enhance, to a certain extent, the robustness of high-order DG methods.  

Without loss of generality, the third-order (${\mathbb P}^2$-based) DG method is considered. Unless otherwise stated, all the computations are restricted to the ideal 
equation of state  \eqref{eq:EOS} with the adiabatic index $\gamma=\frac53$, 
and the CFL number is taken as 0.15.
For the problems involving discontinuity, before using the PP limiter, the WENO limiter \cite{Qiu2005} with locally divergence-free reconstruction (cf. \cite{ZhaoTang2017}) is also implemented with the aid of the local characteristic decomposition,  
to enhance the numerical stability of high-oder DG methods in resolving the strong discontinuities and their interactions. The WENO limiter is only used 
in the ``trouble'' cells adaptively detected by the indicator in \cite{Krivodonova}.

\subsubsection{Accuracy tests}

Two smooth problems are solved to test the accuracy of the ${\mathbb P}^2$-based 
DG method with the PP limiter. 
The first problem, similar to the one simulated in \cite{zhang2010b}, 
describes a MHD sine wave periodicly propagating within the domain $[0,2\pi]^2$ and $\gamma=1.4$. 
The exact solution is given by 
\begin{equation*}
	(\rho,{\bf v},p,{\bf B})({\tt x},{\tt y},t) 
	= \big( 1+0.99\sin({\tt x}+{\tt y}-2t),~1,~1,~0,~1,~0.1,~0.1,~0\big). 
\end{equation*}
The second problem is the vortex problem \cite{Christlieb}. 
The initial condition is a mean flow 
\begin{equation*}
	(\rho,{\bf v},p,{\bf B})({\tt x},{\tt y},0) 
	=(1,~1,~1,~0,~1,~0,~0,~0),
\end{equation*}
with vortex perturbations on $v_1, v_2, B_1, B_2,$ and $p$: 
\begin{gather*}
	(\delta v_1, \delta v_2)=\frac{\mu}{ \sqrt{2} \pi} {\rm e}^{0.5(1-r^2)} (-{\tt y},{\tt x}),
	\quad 
	(\delta B_1,\delta B_2) = \frac{\mu}{2\pi} {\rm e}^{0.5(1-r^2)} (-{\tt y},{\tt x}),
	\\
	\delta p = -\frac{ \mu^2 (1+r^2)  }{8 \pi^2} {\rm e}^{1-r^2},
\end{gather*}
where $r=\sqrt{{\tt x}^2+{\tt y}^2}$, and 
the vortex strength 
$\mu=5.389489439$ 
such that the lowest pressure in the center of the vortex is about $5.3 \times 10^{-12}$. 
The computational domain is $[-10,10]^2$ with periodic boundary conditions.  
Fig. \ref{fig:2Dacc} displays the 
numerical errors obtained by the third-order DG method with the PP limiter at different grid resolutions. 
The results show that the expected convergence order is achieved, and 
the PP limiter does not destroy the accuracy.

\begin{figure}[htbp]
	\centering
	\includegraphics[width=0.49\textwidth]{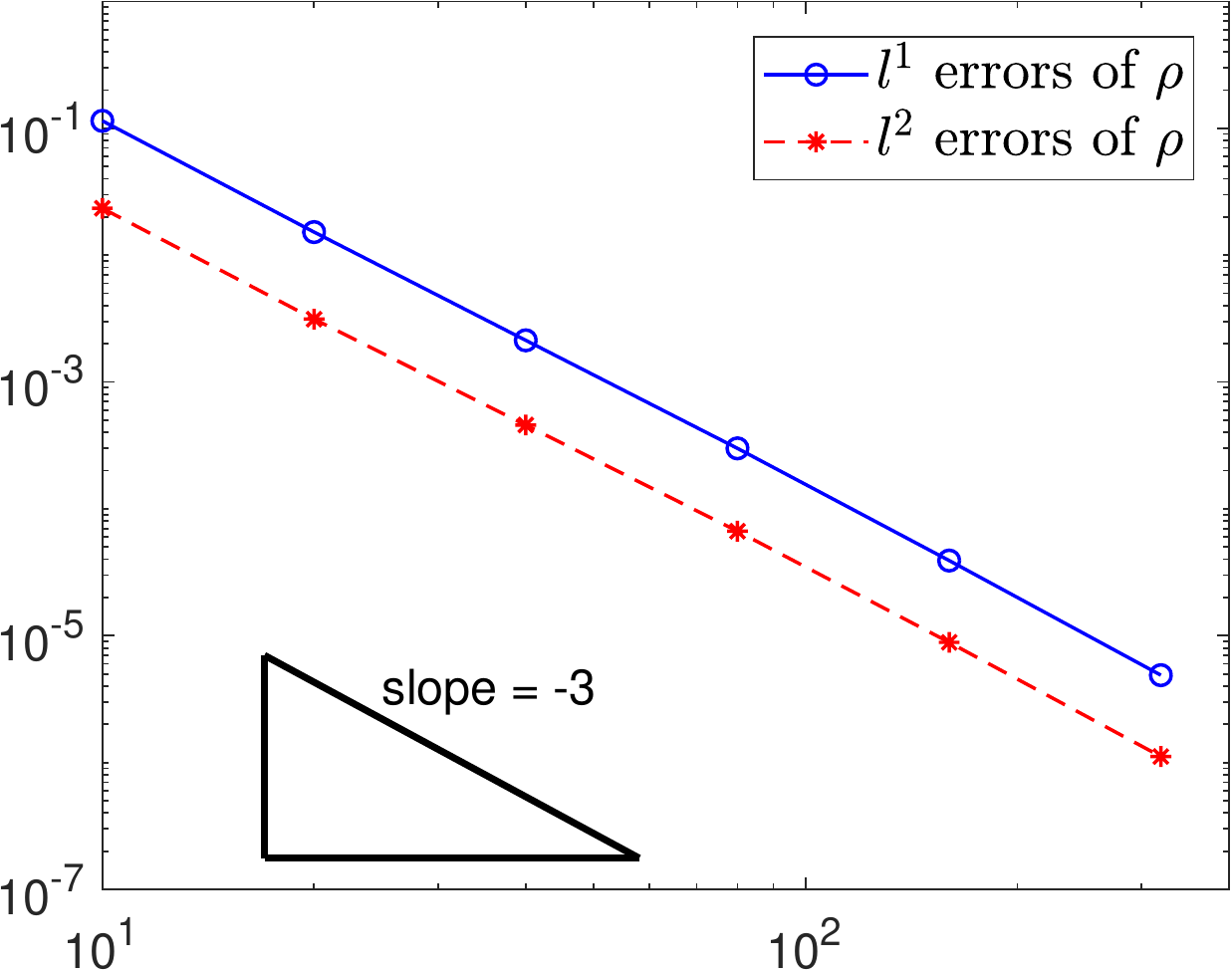}
	\includegraphics[width=0.49\textwidth]{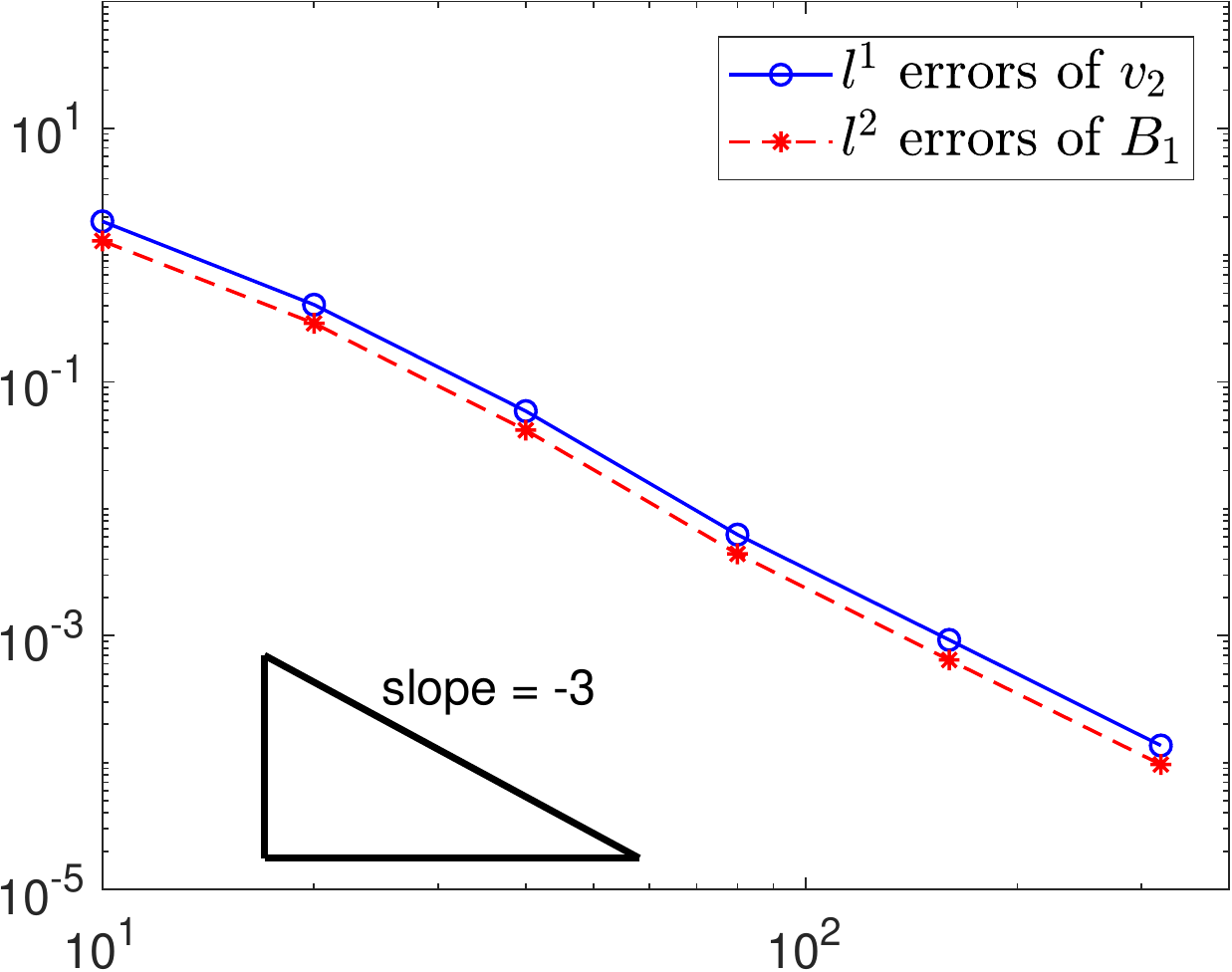}
	\caption{\small
		Numerical errors obtained by the third-order DG method at different grid resolutions with $N\times N$ cells. Left: the first 2D smooth problem at $t=0.1$; right: 
		the second 2D smooth problem at $t=0.05$. The horizontal axis denotes the value of $N$. 
	}\label{fig:2Dacc}
\end{figure}

\subsubsection{Benchmark tests}

The Orszag-Tang problem (see e.g., \cite{Li2005}) and the rotor problem \cite{BalsaraSpicer1999} are benchmark tests widely performed in the literature. Although not extreme, they are simulated by our third-order DG code  
to verify the high-resolution of the DG method as well as the correctness of our code. 
The contour plots of the density are shown 
in Fig. \ref{fig:standardtests} and agree well with those computed in \cite{BalsaraSpicer1999,Li2005}. We observe that the PP limiter does not get turned on, as the condition \eqref{eq:2Dadmissiblity} 
is automatically satisfied in these simulations.

\begin{figure}[htbp]
	\centering
	{\includegraphics[width=0.98\textwidth]{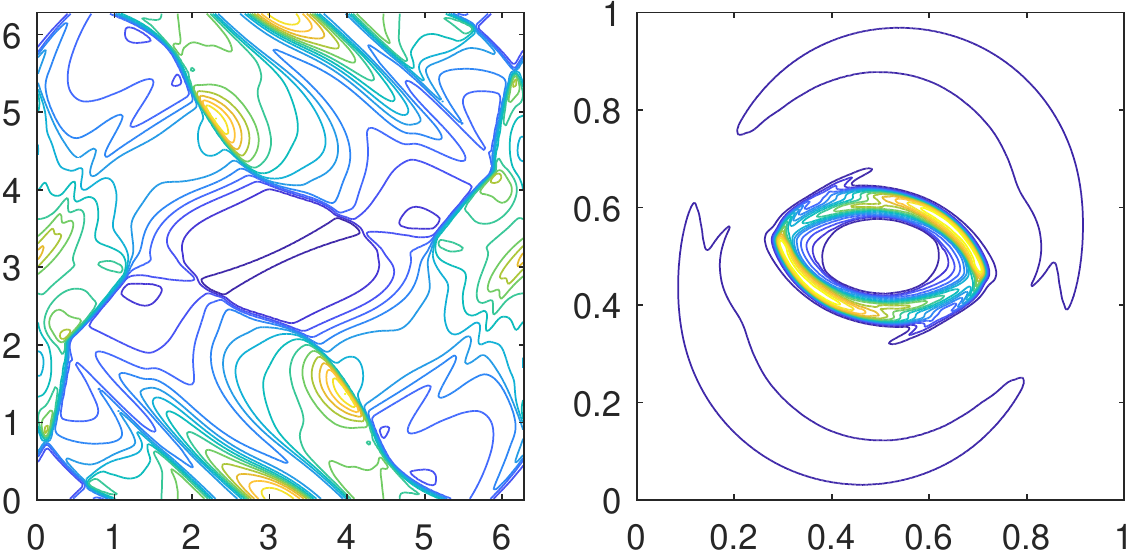}}
	\caption{\small Numerical solutions of the
		Orszag-Tang problem with $200\times 200$ cells at $t=2$ (left) and the rotor problem with $400\times 400$ at $t=0.295$ (right) computed by the third-order DG method.} 
	\label{fig:standardtests}
\end{figure}

\subsubsection{Positivity-preserving tests}

Two extreme problems are solved to demonstrate 
the importance of the proposed conditions \eqref{eq:DivB:cst12}--\eqref{eq:2Dadmissiblity} 
in Theorem \ref{thm:PP:2DMHD} and validate the estimate in Theorem \ref{thm:PP:2DMHD:further}. 

The first one is the blast problem \cite{BalsaraSpicer1999} to verify 
the importance of enforcing the condition \eqref{eq:2Dadmissiblity} 
for achieving PP high-order DG methods. 
This problem describes the propagation of a
circular strong fast magneto-sonic shock formulates and propagates into 
the  ambient
plasma with low plasma-beta. 
Initially, the computational domain $[-0.5,0.5]^2$ 
is filled with plasma at rest with unit density and adiabatic index $\gamma=1.4$. The explosion zone 
$(r<0.1)$ has a pressure of $10^3$, while the ambient medium $(r>0.1)$ has a lower pressure of $0.1$, where $r=\sqrt{{\tt x}^2+{\tt y}^2}$. The magnetic field is initialized in the $\tt x$-direction as $100/\sqrt{4\pi}$. 
Figure \ref{fig:BL1} shows the numerical results at $t=0.01$ computed by the third-order DG method with the PP limiter on the mesh of $320\times320$ 
uniform cells. 
We see that 
the results are highly in agreement with those displayed  
in \cite{BalsaraSpicer1999,Li2011,Christlieb}, 
and the density profile is well captured with much less oscillations than those shown in \cite{BalsaraSpicer1999,Christlieb}. 
It is noticed that 
the third-order DG method fails to preserve the positivity of pressure at time $t \approx 2.845\times 10^{-4}$ 
if the PP limiting procedure is not employed to enforce the condition \eqref{eq:2Dadmissiblity}.

\begin{figure}[htbp]
	\centering
	{\includegraphics[width=0.49\textwidth]{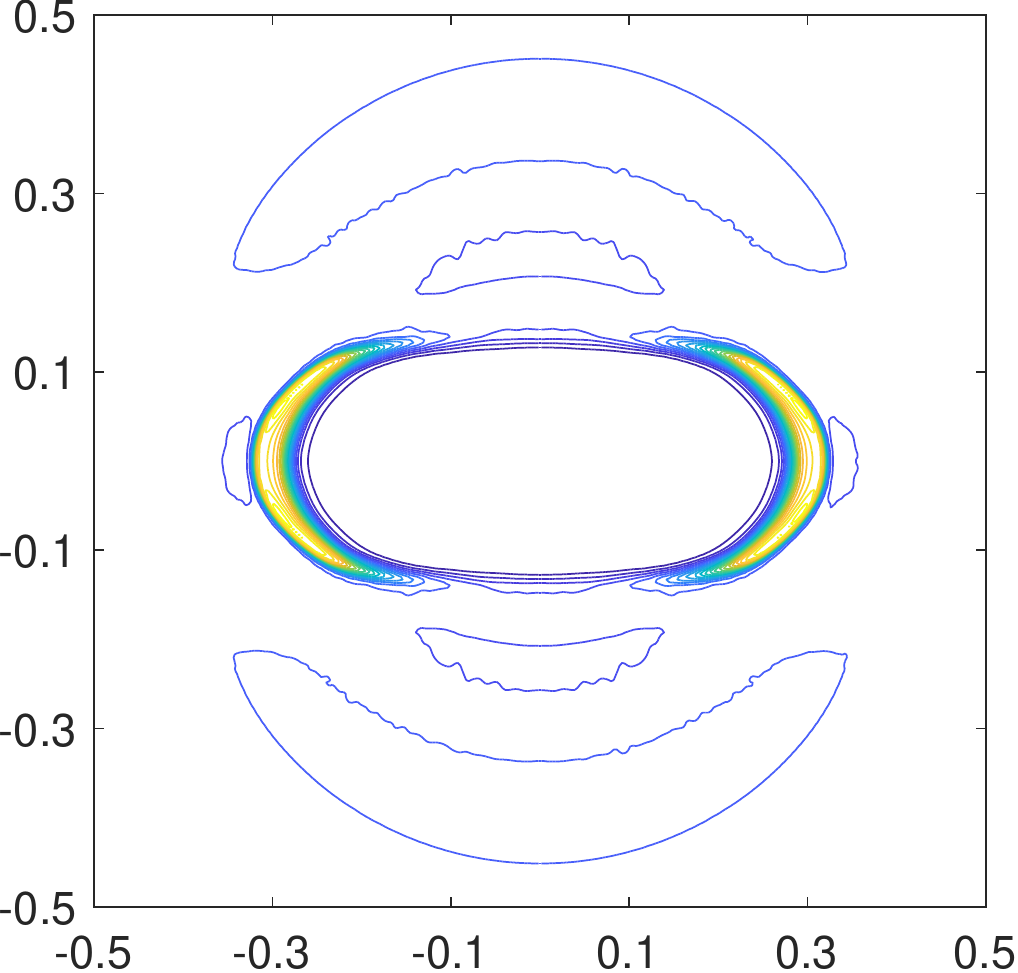}}
	{\includegraphics[width=0.49\textwidth]{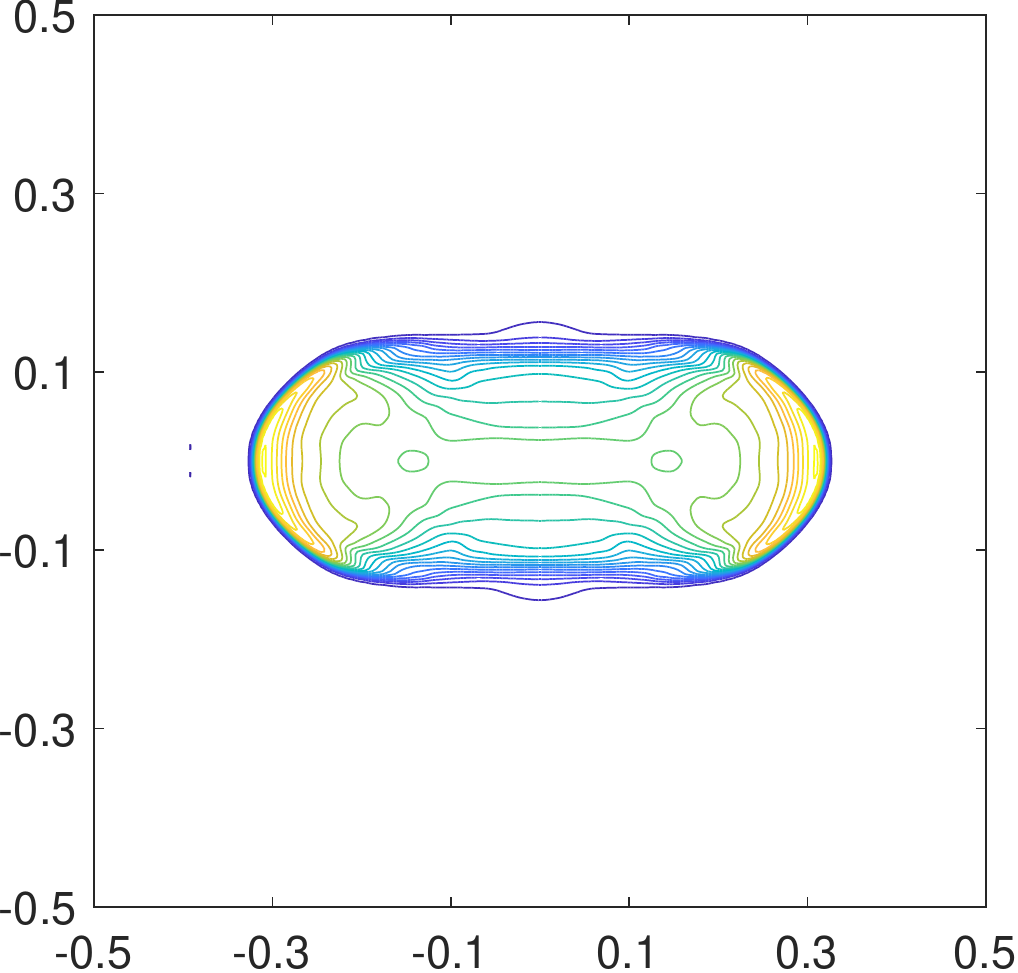}}
	{\includegraphics[width=0.49\textwidth]{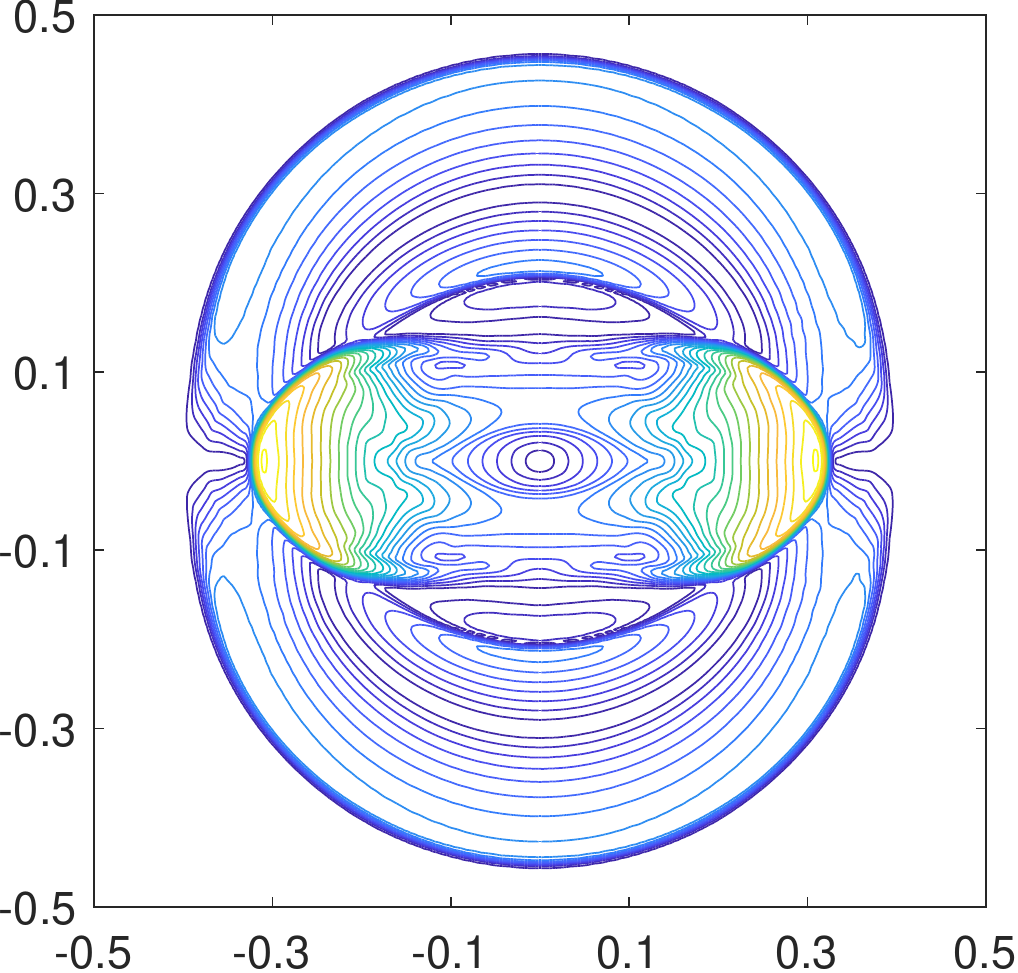}}
	{\includegraphics[width=0.49\textwidth]{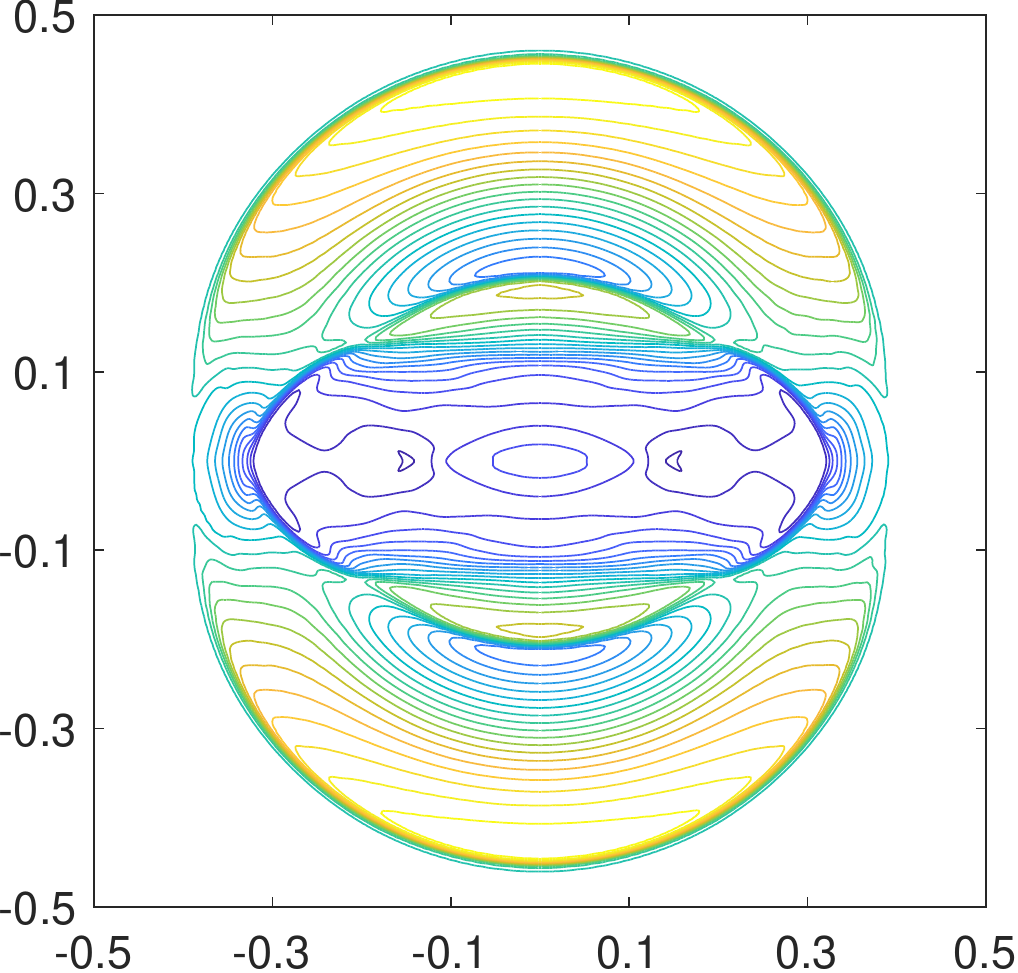}}
	\caption{\small Blast problem: 
		the contour plots of density $\rho$ (top left), 
		pressure $p$ (top right), velocity $|{\bf v}|$ (bottom left) 
		and magnetic pressure $p_m$ (bottom right) at $t=0.01$.} 
	\label{fig:BL1}
\end{figure}

To examine the PP property of the third-order DG scheme with the PP limiter, it is necessary to try more challenging test (rather than the standard tests). 
In a high Mach number jet with strong magnetic field, 
the internal energy is very
small compared to the huge kinetic and magnetic energy, 
negative pressure may easily appear in the numerical simulation. 
We consider the Mach 800 dense jet in \cite{Balsara2012}, and add a magnetic field so as to simulate the MHD jet flows. 
Initially, the computational domain $[-0.5,0.5]\times[0,1.5]$ is 
filled with a static uniform medium with density of $0.1\gamma$ and unit pressure, 
where the adiabatic index $\gamma=1.4$. 
A dense jet is injected in the $\tt y$-direction through the inlet part ($\left|{\tt x}\right|<0.05$) 
on the bottom boundary (${\tt y}=0$) with density of $\gamma$, unit pressure and speed of $800$.
The fixed inflow condition
is specified on the nozzle $\{{\tt y}=0,\left|{\tt x}\right|<0.05\}$, and the
other boundary conditions are outflow. 
A magnetic field with a magnitude
of $B_a$ is initialized along the $\tt y$-direction.  
The presence of magnetic field makes this test more extreme. 
A larger $B_a$ implies a larger value of $|{\bf v} \cdot {\bf B}|=800 B_a$, 
which more easily leads to 
negative numerical pressure when the DDF condition \eqref{eq:DivB:cst12} 
is violated seriously, as indicated by Theorem \ref{thm:PP:2DMHD:further}.  
Therefore, we have a strong motivation to 
examine the PP property by using this kind of problems.

We first consider a relatively mild setup with a weak magnetic field $B_a=\sqrt{20}$. 
The corresponding plasma-beta ($\beta=0.1$) is not very small. 
The locally divergence-free DG method with the PP limiter 
works well for this weak magnetized case, see Fig. \ref{fig:jet1}, which shows   
the results at $t=0.002$ on the mesh of $400 \times 600$ cells.
It is seen that the density, gas pressure and velocity profiles 
are very close to those displayed in \cite{Balsara2012} for the same jet but without 
magnetic field. 
This is not surprising, because the magnetic field is weak in our case. 
In this simulation, it is also necessary to enforce the condition 
\eqref{eq:2Dadmissiblity} by the PP limiter, otherwise the 
DG code will break down at $t \approx 5.58 \times 10^{-5}$.

\begin{figure}[htbp]
	\centering
	\includegraphics[width=0.98\textwidth]{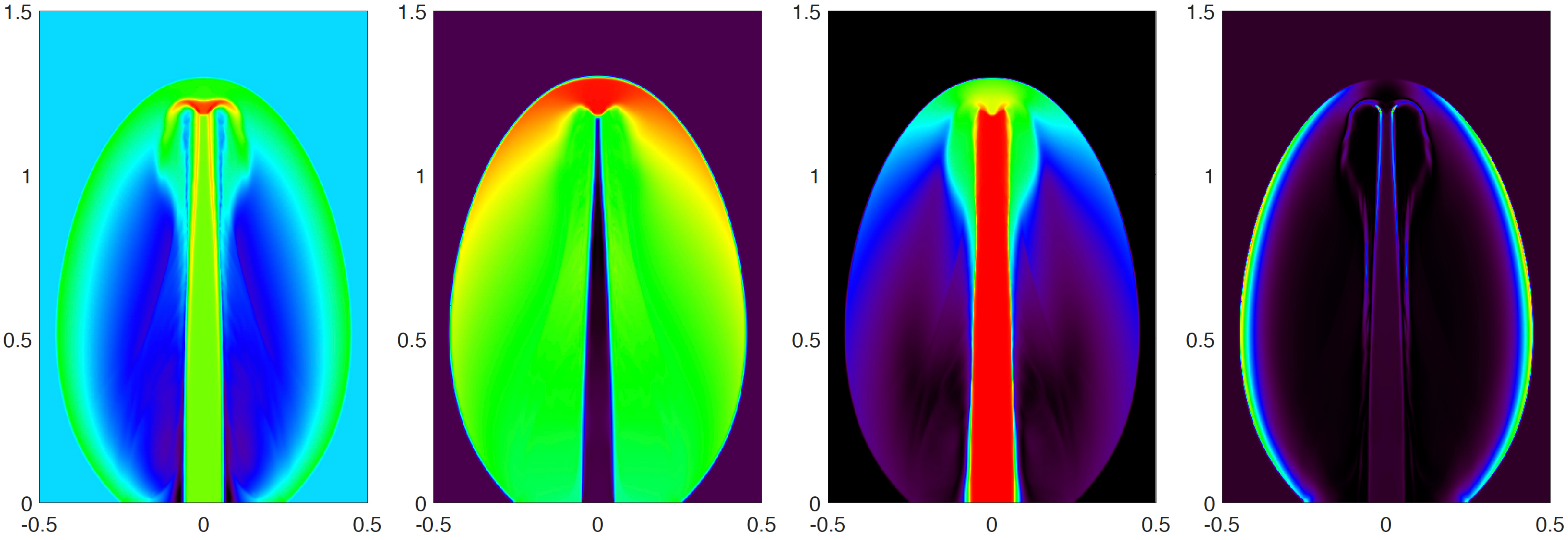}
	\caption{\small
		High Mach number jet with a weak magnetic field $B_a=\sqrt{20}$. 
		The schlieren images of density logarithm, gas pressure logarithm, 
		velocity $|{\bf v}|$ and magnetic pressure (from left to right) at 
		$t=0.002$.    	
	}\label{fig:jet1}
\end{figure}

\begin{figure}[htbp]
	\centering
	\includegraphics[width=0.80\textwidth]{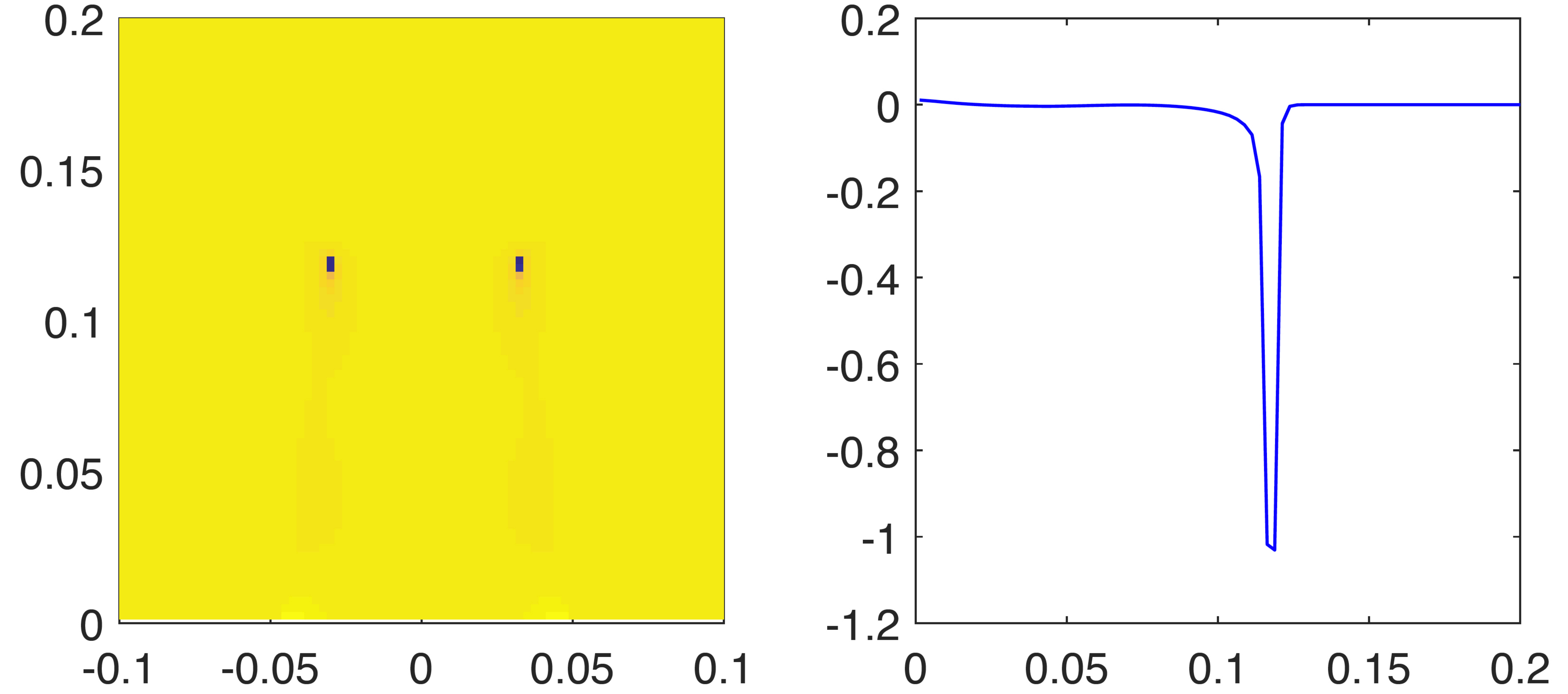}
	\caption{\small
		High Mach number jet with a strong magnetic field $B_a=\sqrt{200}$.
		the schlieren image of $\delta_{ij}$ (left)
		and its slice along ${\tt x}=0.03125$ (right). 
	}\label{fig:jet1b}
\end{figure}

To investigate the importance of the proposed DDF condition \eqref{eq:DivB:cst12} in 
Theorem \ref{thm:PP:2DMHD}, we now try to simulate the jet in a  
moderately magnetized case with $B_a=\sqrt{200}$ (the corresponding 
plasma-beta $\beta = 0.01$) on the mesh of $400 \times 600$ cells. 
In this case, the locally divergence-free third-order DG method with the PP limiter 
breaks down at $t\approx 0.00024$. 
This failure results from the computed inadmissible 
cell averages of conservative variables, detected in 
the four cells 
centered at points 
$(-0.03125, 0.11625)$, $(-0.03125, 0.11875)$, $(0.03125,0.11625)$, 
and $(0.03125, 0.11875)$, respectively. 
As expected from Theorem \ref{thm:PP:2DMHD:further}, 
these inadmissible 
cell averages correspond to negative numerical pressure (internal energy) 
due to the violation of DDF condition \eqref{eq:DivB:cst12}. 
We recall that Theorem \ref{thm:PP:2DMHD:further} implies for 
the inadmissible 
cell averages that  
$$
0 \ge	{\mathcal E} ( \bar{\bf U}_{ij}^{n+1} ) > -\Delta t_n \big(\bar{\bf v}_{ij}^{n+1} \cdot 
\bar{\bf B}_{ij}^{n+1} \big) {\rm div}_{ij}{\bf B}^n.
$$
If one defines 
\begin{align*}
	\begin{split}
		&{\mathcal P}_{ij} := 
		{\mathcal E} ( \bar{\bf U}_{ij}^{n+1} ) 
		+ \Delta t_n \big(\bar{\bf v}_{ij}^{n+1} \cdot 
		\bar{\bf B}_{ij}^{n+1} \big) {\rm div}_{ij}{\bf B}^n > 0,
		\\
		&{\mathcal N}_{ij} :=
		-\Delta t_n \big(\bar{\bf v}_{ij}^{n+1} \cdot 
		\bar{\bf B}_{ij}^{n+1} \big) {\rm div}_{ij}{\bf B}^n,
	\end{split}
\end{align*}
and $\delta_{ij} := {\mathcal N}_{ij} / {\mathcal P}_{ij} $, 
then $\delta_{ij}\le-1$ for inadmissible cell averages, 
and $\delta_{ij} > -1$ for admissible cell averages.  
As we see from the proof of Theorem \ref{thm:PP:2DMHD:further}, ${\mathcal N}_{ij}$ can be considered as the dominate negative part of 
${\mathcal E} ( \bar{\bf U}_{ij}^{n+1} )$ 
affected by the discrete divergence-error ${\rm div}_{ij}{\bf B}^n$, while 
${\mathcal P}_{ij}$ is the main positive part contributed by the condition 
\eqref{eq:2Dadmissiblity} enforced by the PP limiter. 
As evidences of these, Fig. \ref{fig:jet1b} gives the close-up of 
the schlieren image of $\delta_{ij}$ 
and its slice along ${\tt x}=0.03125$. 
It clearly shows the two subregions with small values $\delta_{ij}<-1$, 
and the four detected cells with inadmissible 
cell averages are exactly located in those two subregions. 
This further demonstrates our analysis in 
Theorem \ref{thm:PP:2DMHD:further} and that  
the DDF condition \eqref{eq:DivB:cst12} is really crucial in achieving 
completely PP schemes in 2D. 
It is observed that the code fails also on a refined mesh, and also for 
more strongly magnetized case.

More numerical results further supporting our analysis can be found in \cite{WuShu2018}, 
where the proposed theoretical techniques are applied  
to design multi-dimensional provably PP DG schemes via the discretization of symmetrizable ideal MHD equations.

\section{Conclusions}\label{sec:con}

We presented the rigorous PP analysis of 
conservative schemes with the LF flux for one- and multi-dimensional  ideal 
MHD equations. 
It was based on several important properties of admissible state set, 
including a novel equivalent form, convexity, orthogonal invariance and 
the generalized LF splitting properties. 
The analysis was focused on the finite volume or discontinuous Galerkin schemes on uniform Cartesian meshes. 
In the 1D case, we proved that the LF scheme 
with proper numerical viscosity is PP, and 
the high-order schemes are PP under accessible conditions. 
In the 2D case, 
our analysis revealed for the first time that a discrete divergence-free (DDF) condition is 
crucial for achieving 
the PP property of schemes for ideal MHD. 
We proved that 
the 2D LF scheme with proper numerical viscosity preserves the 
positivity and the DDF condition. 
We derived sufficient conditions for achieving 2D PP high-order schemes. Lower bound of the internal energy was derived when the 
proposed DDF condition was not satisfied, yielding that 
negative internal energy may be more easily computed in the cases with large $|{\bf v}\cdot {\bf B}|$ and large discrete divergence error. Our analyses were further confirmed by the numerical examples, and extended to the 3D case.

In addition, several usually-expected properties 
were disproved in this paper. Specifically, we rigorously showed that:  
(i) the LF splitting property does not always hold; 
(ii) the 1D LF scheme with standard numerical viscosity or 
piecewise constant $B_1$ is not PP in general, no matter how small the CFL number is;
(iii) the 2D LF scheme is not always PP under any CFL condition, unless  additional condition like the DDF condition is satisfied.  
As a result, some existing techniques for PP analysis become 
inapplicable in the MHD case. 
These, together with the technical challenges arising from the solenoidal magnetic field and
the intrinsic complexity of the MHD system, make the proposed analysis very nontrivial.

From the viewpoint of preserving positivity, our analyses provided  
a new understanding of the importance of divergence-free 
condition in robust MHD simulations. 
Our analyses and novel techniques as well as the provenly PP schemes can also be useful for investigating 
or designing 
other 
PP schemes for ideal MHD.  
In \cite{WuShu2018}, 
we applied the proposed analysis approach  
to develop multi-dimensional probably PP high-order 
methods for the symmetrizable version of the ideal
MHD equations. 
The extension of the PP analysis to less dissipative
numerical fluxes and on more general/unstructured meshes will be studied in a coming paper.

%
%

\appendix

\section{Additional Proofs}

\subsection[Proof of Proposition 2.5]{{Proof of Proposition \ref{lemma:2.8}}}
\label{sec:proof}
\begin{proof}
	We prove it by contradiction. 
	Assume that \eqref{eq:LFproperty} always holds. 
	Let $\alpha = \chi {\mathscr{R}}_i $ with the constant $\chi \ge 1$. 
	For any ${\tt p}\in \big(0,\frac1{\gamma}\big)$, 
	consider the admissible state 
	${\bf U}=(1,0,0,0,1,0,0, \frac{\tt p}{\gamma-1} + \frac{1}{2} )^\top$ for an ideal gas. Then \eqref{eq:LFproperty} implies that
	\begin{equation*}
		{\bf U} \pm \frac{ {\bf F}_1 ({\bf U}) }{\alpha} =
		\bigg(
		1,\pm   \frac{ {\tt p} - \frac{1}{2} } {\chi},0,0,1,0,0,\frac{\tt p}{\gamma-1} + \frac{1}{2}
		\bigg)^{\top}\in {\mathcal G}.
	\end{equation*}
	This means, for $\forall {\tt p}\in \big(0,\frac1{\gamma}\big)$, 
	$$
	0< {\mathcal E} ( {\bf U} \pm \alpha^{-1} {\bf F}_1 ({\bf U}) )  = \frac{\tt p}{\gamma-1} - { \frac{1}{2\chi^2} }  \bigg( {\tt p} - \frac{1}2 \bigg)^2=:{\mathcal E}_f({\tt p}) ,
	$$
	The continuity of ${\mathcal E}_f ({\tt p})$ further implies 
	$$
	0 \le \mathop {\lim }\limits_{ {\tt p} \to 0^+ } {\mathcal E}_f ( {\tt p} ) = - \frac{1}{ 8\chi^2} < 0,
	$$
	which is a contradiction. Hence \eqref{eq:LFproperty} does not always hold. 
	The proof is completed. 
\end{proof}

\subsection{Proof of Proposition \ref{lem:fornewGLF}}
\label{sec:proofwkl111}

\begin{proof}
	The proof is similar to that of Lemma \ref{theo:MHD:LLFsplit}. 
	Without loss of generality, we only 
	show \eqref{eq:MHD:LLFsplit:one-state} for $i=1$, 
	while the cases of $i=2$ and $i=3$ can be then proved by using 
	the
	orthogonal invariance in Lemma \ref{lem:MHD:zhengjiao} and similarly 
	to part (ii) of the proof of Lemma \ref{theo:MHD:LLFsplit}.

	For $i=1$, let define
	\begin{align*}
		& \hat \Pi_u =  {\bf U} \cdot {\bf n}^* + \frac{|{\bf B}^*|^2}{2}, \quad  
		\hat \Pi_f=   {\bf F}_1({\bf U}) \cdot {\bf n}^* +  v_i^* \frac{|{\bf B}^*|^2}{2} -   B_i({\bf v}^* \cdot {\bf B}^*) .
	\end{align*}
	Then it only needs to show
	\begin{equation}\label{eq:needproof1A}
	\frac{|\hat \Pi_f|}{\hat \Pi_u} \le \widehat \alpha_1 ({\bf U},{v}_1^*),
	\end{equation}
	by noting that
	\begin{equation}\label{eq:PiuA}
	\hat \Pi_u =  | \hat {\bm \theta}|^2 > 0 ,
	\end{equation}
	where the nonzero vector $\hat {\bm \theta} \in {\mathbb{R}}^{7}$ is defined as
	$$
	\hat {\bm \theta}= \frac1{\sqrt{2}} \Big(
	\sqrt{\rho} ( {\bf v} - {\bf v}^* ),~
	{\bf B}-{\bf B}^*,~
	\sqrt{2 \rho e}
	\Big)^\top.
	$$
	The proof of \eqref{eq:needproof1A} is divided into the following two steps.

	{\tt Step 1}. Reformulate $\hat \Pi_f$ into a quadratic form in the variables $\hat \theta_j, 1\le j \le 7$. 
	Unlike what is required in Lemma \ref{theo:MHD:LLFsplit}, one cannot 
	require that the coefficients of this quadratic form is independent on ${\bf v}^*$ and ${\bf B}^*$, 
	but can require those coefficients do not depend on 
	${ v}_2^*$, $v_3^*$ and ${\bf B}^*$.   
	Similar to the proof of Lemma \ref{theo:MHD:LLFsplit}, we technically arrange $\hat \Pi_f$  as 
	\begin{equation}\label{eq:3partsA}
	\hat \Pi_f = \frac12 v_1^* |{\bf B}-{\bf B}^*|^2 + 
	+ \frac{\rho v_1 }2 |{\bf v}-{\bf v}^*|^2    +  v_1 \rho e  +  p ( v_1 - v_1^* )
	+ \sum_{j=2}^3 ( B_j (v_1-v_1^*)
	- B_1( v_j -v_j^* ) ) ( B_j - B_j^* ).
	\end{equation}
	Then we immediately have  
	\begin{equation}\label{eq:B}
	\hat \Pi_f= v_1^* \sum_{j=4}^6 \hat \theta_j^2 + 
	v_1 \Big( \sum_{j=1}^3 \hat \theta_j^2 + \hat \theta_7^2 \Big)
	+ 2{\mathscr{C}}_s 
	\hat \theta_1 \hat \theta_7 +  \frac{ 2 B_2} { \sqrt{\rho} } \hat \theta_1 \hat \theta_5 +  \frac{ 2 B_3} { \sqrt{\rho} } \hat \theta_1 \hat \theta_6
	-  \frac{2B_1}{\sqrt{\rho}} (\hat \theta_2 \hat \theta_5 + \hat \theta_3 \hat \theta_6).
	\end{equation}
	
	{\tt Step 2}. Estimate the upper bound of $\frac{|\hat \Pi_f|}{\hat \Pi_u}$. 
	Note that
	\begin{align*}
		\hat \Pi_f  & = v_1^* \sum_{j=4}^6 \hat \theta_j^2 + v_1 \Big( \sum_{j=1}^3 \hat \theta_j^2 +  \hat \theta_7^2 \Big)
		+ \hat {\bm \theta}^\top {\bf A}_7 \hat {\bm \theta} ,
	\end{align*}
	where 
	$$
	{\bf A}_7 =  \begin{pmatrix}
	0 & 0 & 0 & 0 & B_2 \rho^{-\frac12} & B_3 \rho^{-\frac12} & {\mathscr{C}}_s  \\
	0 & 0 & 0 & 0 & -B_1 \rho^{-\frac12} & 0 & 0 \\
	0 & 0 & 0 & 0 & 0 & -B_1 \rho^{-\frac12} & 0 \\
	0 & 0 & 0 & 0 & 0 & 0 & 0 \\
	B_2 \rho^{-\frac12} & -B_1 \rho^{-\frac12} & 0 & 0 & 0 & 0 & 0 \\
	B_3 \rho^{-\frac12} & 0 & -B_1 \rho^{-\frac12} & 0 & 0 & 0 & 0 \\
	{\mathscr{C}}_s  & 0 & 0 & 0 & 0 & 0 & 0
	\end{pmatrix}.
	$$
	The spectral radius 
	of ${\bf A}_7$ is $ {\mathscr{C}}_1$. This gives the following estimate 
	\begin{equation} \label{eq:PI4A}
	\begin{aligned} 
	|\hat \Pi_f|  & 
	\le |v_1^*| \sum_{j=4}^6 \hat \theta_j^2
	+
	|v_1| \bigg( \sum_{j=1}^3
	\hat \theta_{j}^2  +  \hat \theta_7^2 \bigg)
	+ | \hat{\bm \theta}^\top {\bf A}_7 \hat {\bm \theta} |
	\\
	&\le \max\{|v_1|,| v_1^*|\}  |\hat {\bm \theta} |^2 
	+  {\mathscr{C}}_1 |\hat {\bm \theta} |^2
	= \widehat \alpha_1 ({\bf U},v_1^*)  |\hat{\bm \theta} |^2 
	=  \widehat \alpha_1 ({\bf U},v_1^*)  \hat \Pi_u. 
	\end{aligned}
	\end{equation}
	Therefore, the inequality \eqref{eq:needproof1A} holds. The proof is completed.
\end{proof}

\subsection{Deriving generalized LF splitting property by 
	Proposition \ref{lem:fornewGLF}}\label{sec:gLFnewproof}

\begin{proposition}\label{prop:n1DgLF}
	The 1D generalized LF splitting property in Theorem \ref{theo:MHD:LLFsplit1D} holds for any  $\alpha  $ 
	satisfying 
	$$
	\alpha > \widetilde\alpha_1 (\hat{\bf U},\check{\bf U}), 
	$$
	where $\widetilde\alpha_1$ is a different lower bound given by    
	$$
	\widetilde\alpha_1 (\hat{\bf U},\check{\bf U}) 
	:=
	\max \{ \hat{\mathscr{C}}_1, \check{\mathscr{C}}_1
	\} + 
	\max \{ |\hat v_1| + \hat{\mathscr{D}}_1, |\check v_1| + \check{\mathscr{D}}_1
	\}
	$$
	with 
	$$\hat{\mathscr{D}}_1 := \frac{|\hat p_{ tot}- \hat B_1^2|}{
		\hat \rho (A_1-\hat v_1 )},
	\qquad 
	\check{\mathscr{D}}_1 := \frac{|\check p_{ tot}- \check B_1^2|}{
		\check \rho (A_1+\check v_1 )},
	\qquad 
	A_1 := \max\{ |\hat v_1|, |\check v_1| \}+ \max\{  \hat{\mathscr{C}}_1, \check{\mathscr{C}}_1 \}.
	$$
	For the ideal EOS \eqref{eq:EOS}, the property also holds for 
	any  $\alpha > \widetilde {\widetilde  \alpha}_1 (\hat{\bf U},\check{\bf U}) $, where  
	$$
	\alpha > \widetilde {\widetilde  \alpha}_1 (\hat{\bf U},\check{\bf U}) 
	:=
	\max \{ \hat{\mathcal{C}}_1, \check{\mathcal{C}}_1
	\} + 
	\max \{ |\hat v_1| + \hat{\mathcal{D}}_1, |\check v_1| + \check{\mathcal{D}}_1
	\}
	$$
	with 
	$$\hat{\mathcal{D}}_1 := \frac{|\hat p_{ tot}- \hat B_1^2|}{
		\hat \rho ({\mathcal A}_1-\hat v_1 )},
	\qquad 
	\check{\mathcal{D}}_1 := \frac{|\check p_{ tot}- \check B_1^2|}{
		\check \rho ({\mathcal A}_1+\check v_1 )},
	\qquad 
	{\mathcal A}_1 := \max\{ |\hat v_1|, |\check v_1| \}+ \max\{  \hat{\mathcal{C}}_1, \check{\mathcal{C}}_1 \}.
	$$
\end{proposition}

\begin{proof}
	Let $\overline \rho$, $\overline {\bf v}$ 
	and $\overline {\bf B}$ respectively denote the density, velocity and magnetic field corresponding to 
	$\overline{\bf U}$. 
	Note that 
	$$\widetilde \alpha_1(\hat{\bf U},\check{\bf U}) > 
	\max\{ |\hat v_1|, |\check v_1| \},$$ 
	which implies 
	$$
	\overline \rho =
	\frac12 \bigg( \hat \rho \Big(1-\frac{\hat v_1}{\alpha}\Big) + \check\rho \Big(1+\frac{\check v_1}{\alpha}\Big) \bigg) > 0,$$
	for any $\alpha > \widetilde\alpha_1 (\hat{\bf U},\check{\bf U}) $. 
	Then it only needs to show ${\mathcal E} (\overline {\bf U}) >0$. 
	Define 
	$\overline{\bf n} = \big( \frac{|\overline{\bf v}|^2}2,~- \overline{\bf v},~-\overline{\bf B},~1 \big)^\top,$  
	then we can reformulate $2{\mathcal E} (\overline {\bf U}) $ as 
	\begin{align} \nonumber
		2{\mathcal E} (\overline {\bf U}) 
		& = 2\overline{\bf U} \cdot \overline{\bf n} 
		+ {|\overline{\bf B} |^2 }
		=  \bigg( \hat{\bf U} - \frac{ {\bf F}_1(\hat{\bf U})}{\alpha}
		+
		\check{\bf U} + \frac{ {\bf F}_1(\check{\bf U})}{\alpha} \bigg) 
		\cdot \overline{\bf n} 
		+ {|\overline{\bf B} |^2 } 
		\\ \nonumber
		&= \bigg[\bigg( \hat{\bf U} - \frac{ {\bf F}_1(\hat{\bf U})}{\alpha}
		\bigg) \cdot \overline{\bf n} + \frac{|\overline{\bf B}|^2}{2}
		-  \frac{  1 }{\alpha} \bigg( \overline v_1 \frac{|\overline{\bf B}|^2}{2} - \hat B_1(\overline{\bf v} \cdot \overline{\bf B}) \bigg)	\bigg]
		\\ \nonumber
		&\quad + \bigg[ \bigg( \check{\bf U} + \frac{ {\bf F}_1(\check{\bf U})}{\alpha}
		\bigg) \cdot \overline{\bf n} + \frac{|\overline{\bf B}|^2}{2}
		+  \frac{  1 }{\alpha} \bigg( \overline v_1 \frac{|\overline{\bf B}|^2}{2} - \check B_1(\overline{\bf v} \cdot \overline{\bf B}) \bigg) \bigg]
		\\ \label{eq:aaawu11}
		&=: \Pi_1 + \Pi_2,	
	\end{align}
	where the DDF condition \eqref{eq:descrite1DDIV} has been used. 
	We then use Proposition \ref{lem:fornewGLF} to prove $\Pi_i>0,i=1,2,$ by 
	verifying that
	$$
	\widetilde\alpha_1 (\hat{\bf U},\check{\bf U}) \ge  
	\widehat \alpha_1 (\hat{ \bf U }, \overline v_1),\qquad 
	\widetilde\alpha_1 (\hat{\bf U},\check{\bf U}) \ge  
	\widehat \alpha_1 (\check{ \bf U }, \overline v_1).
	$$
	It is sufficient to show that
	$$
	\widetilde\alpha_1 (\hat{\bf U},\check{\bf U}) \ge 
	\max\{ |\hat v_1|, |\overline v_1|, |\check v_1| \} 
	+ \max \{ \hat{\mathscr{C}}_1, \check{\mathscr{C}}_1
	\},
	$$
	or equivalently 
	$$
	\max \{ |\hat v_1| + \hat{\mathscr{D}}_1, |\check v_1| + \check{\mathscr{D}}_1
	\} \ge \max\{ |\hat v_1|, |\overline v_1|, |\check v_1| \},
	$$
	which can be verified easily by noting that 
	$$
	A_1 \le \widetilde\alpha_1 (\hat{\bf U},\check{\bf U})  < \alpha.
	$$	
	Therefore, $\Pi_i>0,i=1,2,$ by Proposition \ref{lem:fornewGLF}. 
	It follows from \eqref{eq:aaawu11} that ${\mathcal E} (\overline {\bf U})>0$. Hence $\overline {\bf U} \in {\mathcal G}$ for any $\alpha > {\widetilde\alpha}_1 (\hat{\bf U},\check{\bf U})$. 
	
	In the ideal EOS case, by noting that ${\mathcal C}_1 > {\mathscr{C}}_1$, 
	similar arguments   
	imply $\overline {\bf U} \in {\mathcal G}$ for any $\alpha > \widetilde{\widetilde\alpha}_1 (\hat{\bf U},\check{\bf U})$.  
	The proof is completed. 
\end{proof}

\begin{remark}
	It is worth mentioning that 
	the estimated lower bounds $\widetilde \alpha_1$ and $\widetilde{ \widetilde \alpha}_1$ are not as sharp as the bound  
	$\alpha_1$ in Theorem \ref{theo:MHD:LLFsplit1D}. 
	Note that here the lower bound of $\alpha$ is said to be {\em sharper} if it is {\em smaller}, indicating that the resulting generalized 
	LF splitting properties hold for a {\em larger} 
	range of $\alpha$. 
	The {\em sharper} (i.e., {\em smaller}) lower bound is {\em more desirable}, because it corresponds to a {\em less dissipative} LF flux (allowing smaller numerical viscosity) in our provably PP schemes.
\end{remark}

\begin{remark}
	Let define 
	$$ {\mathscr H }_i({\bf U}) :=| v_i| + \frac{| p_{ tot}-  B_i^2|}{
		\rho {\mathscr{C}}_i},\quad 
	{\mathcal H }_i({\bf U}) :=| v_i| + \frac{| p_{ tot}- \ B_i^2|}{
		\rho {\mathcal{C}}_i}.
	$$
	Since $\widetilde\alpha_1$ and $\widetilde{\widetilde\alpha}_1$ in Proposition \ref{prop:n1DgLF} only play the role of lower range bounds, 
	they can be replaced with some simpler but larger (not sharp) ones, 
	e.g., 
	\begin{align*} 
		&\widetilde\alpha_1 (\hat{\bf U},\check{\bf U}) 
		\qquad 
		\longleftrightarrow
		\qquad 
		\max \{ \hat{\mathscr{C}}_1, \check{\mathscr{C}}_1
		\} + 
		\max \big\{ {\mathscr H }_1(\hat{\bf U}), {\mathscr H }_1(\check{\bf U})
		\big\},
		\\
		&
		\widetilde{ \widetilde \alpha}_1 (\hat{\bf U},\check{\bf U}) 
		\qquad \longleftrightarrow \qquad 
		\max \{ \hat{\mathcal{C}}_1, \check{\mathcal{C}}_1
		\} + 
		\max \big\{ {\mathcal H }_1(\hat{\bf U}), {\mathcal H }_1(\check{\bf U})
		\big\}.
	\end{align*}
\end{remark}

The 2D and 3D generalized LF splitting properties 
can also be similarly derived by Proposition \ref{lem:fornewGLF}.

\section{3D Positivity-Preserving Analysis}\label{sec:3D}The extension of our PP analysis to 3D case 
is straightforward, and for completeness, also given as follows. 
We only present the main theorems, and omit the proofs, 
which are very similar to the 2D case except for using the 3D generalized LF splitting property in Theorem  \ref{theo:MHD:LLFsplit3D}.

To avoid confusing subscripts, 
the symbols $({\tt x},{\tt y},{\tt z})$ are used to denote the variables $(x_1,x_2,x_3)$ in \eqref{eq:MHD}. 
Assume that the 3D spatial domain is divided into a uniform cuboid mesh 
with cells $\big\{I_{ijk}=({\tt x}_{i-\frac{1}{2}},{\tt x}_{i+\frac{1}{2}})\times
({\tt y}_{j-\frac{1}{2}},{\tt y}_{j+\frac{1}{2}}) 
\times 
({\tt z}_{k-\frac{1}{2}},{\tt z}_{k+\frac{1}{2}}) \big\}$. 
The spatial step-sizes in ${\tt x},{\tt y},{\tt z}$ directions are denoted by
$\Delta x,\Delta y,\Delta z$ respectively. The time interval is also divided into the mesh $\{t_0=0, t_{n+1}=t_n+\Delta t_{n}, n\geq 0\}$
with the time step size $\Delta t_{n}$ determined by the CFL condition. We  use $\bar {\bf U}_{ijk}^n $
to denote the numerical approximation to the cell average of the exact solution over $I_{ijk}$ at time $t_n$.

\subsection{First-order scheme}
We consider the 3D first-order LF scheme
\begin{equation} \label{eq:3DMHD:LFscheme}
\begin{split}
\bar {\bf U}_{ijk}^{n+1} = \bar {\bf U}_{ijk}^n &- \frac{\Delta t_n}{\Delta x} \Big( 
\hat {\bf F}_1 ( \bar {\bf U}_{ijk}^n ,\bar {\bf U}_{i+1,j,k}^n)
- \hat {\bf F}_1 ( \bar {\bf U}_{i-1,j,k}^n ,\bar {\bf U}_{ijk}^n)  \Big) 
\\
&- \frac{\Delta t_n}{\Delta y} \Big( 
\hat {\bf F}_2 ( \bar {\bf U}_{ijk}^n ,\bar {\bf U}_{i,j+1,k}^n)
-  \hat {\bf F}_2 ( \bar {\bf U}_{i,j-1,k}^n ,\bar {\bf U}_{ijk}^n)
\Big)
\\
& 
- \frac{\Delta t_n}{\Delta z} \Big( 
\hat {\bf F}_3 ( \bar {\bf U}_{ijk}^n ,\bar {\bf U}_{i,j,k+1}^n)
-  \hat {\bf F}_3 ( \bar {\bf U}_{i,j,k-1}^n ,\bar {\bf U}_{ijk}^n)
\Big) ,
\end{split}
\end{equation}
where $\hat {\bf F}_\ell (\cdot,\cdot), \ell=1,2,3,$ are the LF fluxes in \eqref{eq:LFflux}. 
We have the following conclusions.

\begin{theorem}\label{theo:3DcounterEx}
	Let $\alpha_{\ell,n}^{\tt LF} = \chi \max_{ijk}{ {\mathscr{R}}_\ell (\bar{\bf U}^n_{ij})}$ with the 
	constant $\chi \ge 1$, and 
	$$\Delta t_n = \frac{ {\tt C} }{ \alpha_{1,n}^{\tt LF} /\Delta x + \alpha_{2,n}^{\tt LF}  / \Delta y 
		+ \alpha_{3,n}^{\tt LF}  / \Delta z}, $$
	where ${\tt C}>0$ is the CFL number. 
	For any given constants $\chi$ and $\tt C$, 
	there always exists a set of admissible states 
	$\{  \bar {\bf U}_{ijk}^{n},\forall i,j,k\}$ 
	such that the solution $\bar {\bf U}_{ijk}^{n+1}$ 
	of \eqref{eq:3DMHD:LFscheme} does not belong to $\mathcal G$. 
	In other words, for any given $\chi$ and $\tt C$, 
	the admissibility of $\{  \bar {\bf U}_{ijk}^{n}, \forall i,j,k \}$ does not always guarantee that $\bar {\bf U}_{ijk}^{n+1} \in {\mathcal G}$, $\forall i,j,k$. 
\end{theorem}

\begin{theorem} \label{theo:3DMHD:LFscheme}
	If for all $i,j,k$, $\bar {\bf U}_{ijk}^n \in {\mathcal G}$ and satisfies the following DDF condition
	\begin{equation}\label{eq:3DDisDivB}
	\begin{split}
	& \mbox{\rm div} _{ijk} \bar {\bf B}^n :=  \frac{ \left( \bar  B_1\right)_{i+1,j,k}^n - \left( \bar  B_1 \right)_{i-1,j,k}^n } {2\Delta x}
	\\
	& 
	+ \frac{ \left( \bar  B_2 \right)_{i,j+1,k}^n - \left( \bar B_2 \right)_{i,j-1,k}^n } {2\Delta y}
	+ \frac{ \left( \bar  B_3 \right)_{i,j,k+1}^n - \left( \bar B_3 \right)_{i,j,k-1}^n } {2\Delta z}  = 0,
	\end{split}
	\end{equation}	
	then the solution $ \bar {\bf U}_{ijk}^{n+1}$ of \eqref{eq:3DMHD:LFscheme} always belongs to ${\mathcal G}$ under the CFL condition
	\begin{equation}\label{eq:CFL:LF3D}
	0< \frac{ \alpha_{1,n}^{\tt LF} \Delta t_n}{\Delta x} + \frac{ \alpha_{2,n}^{\tt LF} \Delta t_n}{\Delta y}  
	+ \frac{ \alpha_{3,n}^{\tt LF} \Delta t_n}{\Delta z} \le  1,
	\end{equation}
	where the parameters $\{\alpha_{\ell,n}^{\tt LF}\}$ satisfy
	\begin{equation}\label{eq:Lxa123}
	\begin{gathered}
	\alpha_{1,n}^{\tt LF} > \max_{i,j,k} \alpha_1 ( \bar {\bf U}_{i+1,j,k}^n, \bar {\bf U}_{i-1,j,k}^n ),\quad
	\alpha_{2,n}^{\tt LF} > \max_{i,j,k} \alpha_2 ( \bar {\bf U}_{i,j+1,k}^n, \bar {\bf U}_{i,j-1,k}^n ),
	\\
	\alpha_{3,n}^{\tt LF} > \max_{i,j,k} \alpha_3 ( \bar {\bf U}_{i,j,k+1}^n, \bar {\bf U}_{i,j,k-1}^n ).
	\end{gathered}
	\end{equation}
\end{theorem}

\begin{theorem} \label{theo:3DDivB:LFscheme}
	For the LF scheme \eqref{eq:3DMHD:LFscheme},
	the divergence error 
	$$ \varepsilon_{\infty}^n := \max_{ijk} \left| {\rm div}_{ijk} \bar {\bf B}^{n} \right| ,$$ 
	does not grow with $n$
	under the condition \eqref{eq:CFL:LF3D}.
	Furthermore, the numerical solutions 
	$\{\bar {\bf U}_{ijk}^n\}$ 
	satisfy \eqref{eq:3DDisDivB}
	for all $i,j,k$ and $n \in \mathbb{N}$, if \eqref{eq:3DDisDivB} holds for the discrete initial data $\{\bar {\bf U}_{ijk}^0\}$.
\end{theorem}

\begin{theorem} \label{theo:3DFullPP:LFscheme}
	Assume that the discrete initial data $\{\bar {\bf U}_{ijk}^0\}$ are admissible and satisfy \eqref{eq:3DDisDivB}, which can be met by, e.g.,  the following second-order approximation
	\begin{align*}
		&\Big( \bar \rho_{ijk}^0, \bar {\bf m}_{ijk}^0, \overline {(\rho e)}_{ijk}^0 \Big)
		= \frac{1}{\Delta x \Delta y \Delta z} \iint_{I_{ijk}} \big( \rho,{\bf m }, \rho e \big) ({\tt x},{\tt y},{\tt z},0) d{\tt x} d{\tt y} d{\tt z},
		\\
		& \left( \bar B_1 \right)_{ijk}^0 = \frac{1}{4 \Delta y\Delta z} \int_{ {\tt y}_{j-1} }^{ {\tt y}_{j+1 } } 
		\int_{ {\tt z}_{k-1} }^{ {\tt z}_{k+1 } } B_1( {\tt x}_i,{\tt y},{\tt z},0) d{\tt y} d{\tt z},
		\\
		&
		\left( \bar B_2 \right)_{ijk}^0 = \frac{1}{4 \Delta x  \Delta z} \int_{ {\tt x}_{i-1} }^{ {\tt x}_{i+1 } } 
		\int_{ {\tt z}_{k-1} }^{ {\tt z}_{k+1 } } B_2({\tt x},{\tt y}_j,{\tt z},0) d {\tt x} d{\tt z},
		\\
		&
		\left( \bar B_3 \right)_{ijk}^0 = \frac{1}{4 \Delta x  \Delta y} \int_{ {\tt x}_{i-1} }^{ {\tt x}_{i+1 } } 
		\int_{ {\tt y}_{j-1} }^{ {\tt y}_{j+1 } } B_3({\tt x},{\tt y},{\tt z}_k,0) d {\tt x} d{\tt y},
		\\		
		&~\bar E_{ijk}^0 = \overline {(\rho e )}_{ijk}^0 + \frac12 \left( \frac{ |\bar{\bf m}_{ijk}^0|^2}{\bar \rho_{ijk}^0} + |\bar{\bf B}_{ijk}^0|^2 \right).
	\end{align*}
	If the parameters $\{\alpha_{\ell,n}^{\tt LF}\}$ satisfy \eqref{eq:Lxa123}, then under the CFL condition \eqref{eq:CFL:LF3D}, 
	the LF scheme \eqref{eq:3DMHD:LFscheme} preserve both $\bar {\bf U}_{ijk}^{n+1} \in {\mathcal G}$ and the DDF condition \eqref{eq:3DDisDivB} for all $i$, $j$, $k$, and $n \in \mathbb{N}$.
\end{theorem}

\subsection{High-order schemes} 

We focus on the forward Euler method for time discretization, and our analysis also works for  
high-order explicit time discretization using the SSP methods \cite{Gottlieb2009}. To achieve high-order accuracy,  
the approximate solution polynomials ${\bf U}_{ijk}^n ({\tt x},{\tt y},{\tt z})$ of degree $\tt K$
are also built, as approximation to the exact solution ${\bf U}({\tt x},{\tt y},{\tt z},t_n)$ within $I_{ijk}$. Such polynomial vector ${\bf U}_{ijk}^n ({\tt x},{\tt y},{\tt z})$
is, either reconstructed in the finite volume methods
from the cell averages $\{\bar {\bf U}_{ijk}^n\}$ or evolved in the DG methods. Moreover, the cell average of ${\bf U}_{ijk}^n({\tt x},{\tt y},{\tt z})$ over $I_{ijk}$ is $\bar {\bf U}_{ijk}^{n}$.

Let $\{  {\tt x}_i^{(\mu)} \}_{\mu=1}^{\tt Q}$, $\{  {\tt y}_j^{(\mu)} \}_{\mu=1}^{\tt Q}$ 
and $\{  {\tt z}_j^{(\mu)} \}_{\mu=1}^{\tt Q}$
denote the $\tt Q$-point Gauss quadrature nodes in the intervals $[ {\tt x}_{i-\frac12}, {\tt x}_{i+\frac12} ]$, $[ {\tt y}_{j-\frac12}, {\tt y}_{j+\frac12} ]$
and $[ {\tt z}_{k-\frac12}, {\tt z}_{k+\frac12} ]$, respectively. Let  $\{\omega_\mu\}_{\mu=1}^{\tt Q}$ be the associated weights satisfying
$\sum_{\mu=1}^{\tt Q} \omega_\mu = 1$.
With the 2D tensorized quadrature rule for approximating the integrals of numerical fluxes on cell interfaces,
a finite volume scheme or discrete equation for the cell average in the DG method can be written as 
\begin{equation}\label{eq:3DMHD:cellaverage}
\begin{split}
\bar {\bf U}_{ijk}^{n+1} & = \bar {\bf U}_{ijk}^{n}
- \frac{\Delta t_n}{\Delta x} \sum\limits_{\mu,\nu}  \omega_\mu 
\omega_\nu \left(
\hat {\bf F}_1( {\bf U}^{-,\mu,\nu}_{i+\frac{1}{2},j,k}, {\bf U}^{+,\mu,\nu}_{i+\frac{1}{2},j,k} ) -
\hat {\bf F}_1( {\bf U}^{-,\mu,\nu}_{i-\frac{1}{2},j,k}, {\bf U}^{+,\mu,\nu}_{i-\frac{1}{2},j,k})
\right)  \\
& \quad -  \frac{ \Delta t_n }{\Delta y} \sum\limits_{\mu,\nu} \omega_\mu 
\omega_\nu \left(
\hat {\bf F}_2( {\bf U}^{\mu,-,\nu}_{i,j+\frac{1}{2},k} , {\bf U}^{\mu,+,\nu}_{i,j+\frac{1}{2},k}   ) -
\hat {\bf F}_2( {\bf U}^{\mu,-,\nu}_{i,j-\frac{1}{2},k} , {\bf U}^{\mu,+,\nu}_{i,j-\frac{1}{2},k}  )
\right)
\\
& \quad -  \frac{ \Delta t_n }{\Delta z} \sum\limits_{\mu,\nu} \omega_\mu 
\omega_\nu \left(
\hat {\bf F}_3( {\bf U}^{\mu,\nu,-}_{i,j,k+\frac{1}{2}} , {\bf U}^{\mu,\nu,+}_{i,j,k+\frac{1}{2}}   ) -
\hat {\bf F}_3( {\bf U}^{\mu,\nu,-}_{i,j,k-\frac{1}{2}} , {\bf U}^{\mu,\nu,+}_{i,j,k-\frac{1}{2}}  )
\right),
\end{split}
\end{equation}
where $\hat {\bf F}_\ell,\ell=1,2,3$ are the LF fluxes in \eqref{eq:LFflux}, and the limiting values are given by
\begin{align*}
	&{\bf U}^{-,\mu,\nu}_{i+\frac{1}{2},j,k} = {\bf U}_{ijk}^n ({\tt x}_{i+\frac12},{\tt y}_j^{(\mu)},{\tt z}_k^{(\nu)}),\qquad
	{\bf U}^{+,\mu,\nu}_{i-\frac{1}{2},j,k} = {\bf U}_{ijk}^n ({\tt x}_{i-\frac12},{\tt y}_j^{(\mu)},{\tt z}_k^{(\nu)}),
	\\
	&{\bf U}^{\mu,-,\nu}_{i,j+\frac{1}{2},k} = {\bf U}_{ijk}^n ({\tt x}_i^{(\mu)},{\tt y}_{j+\frac12},{\tt z}_k^{(\nu)}),\qquad
	{\bf U}^{\mu,+,\nu}_{i,j-\frac{1}{2},k} = {\bf U}_{ijk}^n ({\tt x}_i^{(\mu)},{\tt y}_{j-\frac12},{\tt z}_k^{(\nu)}),
	\\
	&{\bf U}^{\mu,\nu,-}_{i,j,k+\frac{1}{2}} = {\bf U}_{ijk}^n ({\tt x}_i^{(\mu)},{\tt y}_j^{(\nu)},{\tt z}_{k+\frac12}),\qquad
	{\bf U}^{\mu,\nu,+}_{i,j,k-\frac{1}{2}} = {\bf U}_{ijk}^n ({\tt x}_i^{(\mu)},{\tt y}_j^{(\nu)},{\tt z}_{k-\frac12}).
\end{align*}
For the accuracy requirement,  $\tt Q$ should satisfy:
${\tt Q} \ge {\tt K}+1$ for a $\mathbb{P}^{\tt K}$-based DG method, 
or ${\tt Q} \ge ({\tt K}+1)/2$ for a $({\tt K}+1)$-th order finite volume scheme.

We denote 
\begin{align*}
	&	\overline{(B_1)}_{i+\frac{1}{2},j,k}^{\mu,\nu} := \frac12 \left( (B_1)_{i+\frac{1}{2},j,k}^{-,\mu,\nu} + (B_1)_{i+\frac{1}{2},j,k}^{+,\mu,\nu} \right), 
	\\
	& \overline{ ( B_2) }_{i,j+\frac{1}{2},k}^{\mu,\nu}  := \frac12 \left( ( B_2)_{i,j+\frac{1}{2},k}^{\mu,-,\nu} + ( B_2)_{i,j+\frac{1}{2},k}^{\mu,+,\nu} \right),
	\\
	& \overline{ ( B_3) }_{i,j,k+\frac{1}{2}}^{\mu,\nu}  := \frac12 \left( ( B_3)_{i,j,k+\frac{1}{2}}^{\mu,\nu,-} + ( B_3)_{i,j,k+\frac{1}{2}}^{\mu,\nu,+} \right),
\end{align*}
and  
define the discrete divergences of the numerical magnetic field ${\bf B}^n( {\tt x},{\tt y},{\tt z})$ as 
\begin{align*}
	&	{\rm div} _{ijk} {\bf B}^n :=  \frac{\sum\limits_{\mu,\nu} \omega_\mu 
		\omega_\nu \left(   \overline{(B_1)}_{i+\frac{1}{2},j,k}^{\mu,\nu} 
		- \overline{(B_1)}_{i-\frac{1}{2},j,k}^{\mu,\nu}  \right)}{\Delta x}   
	\\
	&		
	+ \frac{\sum \limits_{\mu,\nu} \omega_\mu \omega_\nu \left(  \overline{ ( B_2)}_{i,j+\frac{1}{2},k}^{\mu,\nu} 
		-  \overline{ ( B_2)}_{i,j-\frac{1}{2},k}^{\mu,\nu}    \right)}{\Delta y}
	+ \frac{\sum \limits_{\mu,\nu} \omega_\mu \omega_\nu \left(  \overline{ ( B_3)}_{i,j,k+\frac{1}{2}}^{\mu,\nu} 
		-  \overline{ ( B_3)}_{i,j,k-\frac{1}{2}}^{\mu,\nu}    \right)}{\Delta z}. 
\end{align*} 
Let $\{ \hat {\tt x}_i^{(\delta) }\}_{\delta=1} ^ {\tt L}$, $\{ \hat {\tt y}_j^{(\nu)} \}_{\delta=1} ^{\tt L}$ 
and $\{ \hat {\tt z}_k^{(\delta)} \}_{\delta=1} ^{\tt L}$ 
be the $\tt L$-point Gauss-Lobatto quadrature nodes in the intervals
$[{\tt x}_{i-\frac{1}{2}},{\tt x}_{i+\frac{1}{2}}]$, $[{\tt y}_{j-\frac{1}{2}},{\tt y}_{j+\frac{1}{2}} ]$ 
and $[{\tt z}_{k-\frac{1}{2}},{\tt z}_{k+\frac{1}{2}} ]$, respectively, and
$ \{\hat \omega_\delta\}_{\delta=1} ^ {\tt L}$ be  associated weights satisfying $\sum_{\delta=1}^{\tt L} \hat\omega_\delta = 1$, where  ${\tt L}\ge \frac{{\tt K}+3}2$ such that the
associated quadrature has algebraic precision of at least degree ${\tt K}$. 
Then we have the following sufficient
conditions for that the high-order scheme \eqref{eq:3DMHD:cellaverage} is PP.

\begin{theorem} \label{thm:PP:3DMHD}
	If the polynomial vectors $\{{\bf U}_{ijk}^n({\tt x},{\tt y},{\tt z})\}$ satisfy:
	\begin{align}\label{eq:DivB:cst123}
		&
		\mbox{\rm div} _{ijk} {\bf B}^n  = 0, \quad \forall i,j,k,
		\\
		&{\bf U}_{ijk}^n ({\bf x}) \in {\mathcal G},\quad \forall{\bf x} \in {\Theta}_{ijk}, \forall i,j,k,
		\label{eq:3Dadmissiblity}
	\end{align}
	with 
	${\Theta}_{ijk}:=\big\{( \hat {\tt x}_i^{(\delta)},{\tt y}_j^{(\mu)},{\tt z}_k^{(\nu)} ),~( {\tt x}_i^{(\mu)},  \hat {\tt y}_j^{(\delta)},{\tt z}_k^{(\nu)} ),~( {\tt x}_i^{(\mu)}, 
	{\tt y}_j^{(\nu)}, \hat{\tt z}_k^{(\delta)} ), \forall \mu,\nu,\delta \big\},$ 
	then the scheme \eqref{eq:3DMHD:cellaverage} always preserves $\bar{\bf U}_{ijk}^{n+1} \in {\mathcal G}$ under the CFL condition
	\begin{equation}\label{eq:CFL:3DMHD}
	0< \frac{\alpha_{1,n}^{\tt LF} \Delta t_n}{\Delta x} + \frac{ \alpha_{2,n}^{\tt LF} \Delta t_n}{\Delta y} 
	+ \frac{ \alpha_{3,n}^{\tt LF} \Delta t_n}{\Delta z} \le \hat \omega_1,
	\end{equation}
	where the parameters $\{\alpha_{\ell,n}^{\tt LF}\}$ satisfy
	\begin{equation}\label{eq:3DhighLFpara}
	\begin{split}
	&\alpha_{1,n}^{\tt LF} >
	\max_{ i,j,k,\mu,\nu}  \alpha_1 \big(  { \bf U }_{i+\frac12,j,k}^{\pm,\mu,\nu} ,  { \bf U }_{i-\frac12,j,k}^{\pm,\mu,\nu}  \big)
	,\\
	&\alpha_{2,n}^{\tt LF} >  \max_{ i,j,k,\mu,\nu}  \alpha_2 \big(  { \bf U }_{i,j+\frac12,k}^{\mu,\pm,\nu} ,  { \bf U }_{i,j-\frac12,k}^{\mu,\pm,\nu}  \big),
	\\
	&
	\alpha_{3,n}^{\tt LF} >  \max_{ i,j,k,\mu,\nu}  \alpha_3 \big(  { \bf U }_{i,j,k+\frac12}^{\mu,\nu,\pm} ,  { \bf U }_{i,j,k-\frac12}^{\mu,\nu,\pm}  \big).
	\end{split}
	\end{equation}
	
\end{theorem}

Lower bound of the internal energy can also be estimated when the 
DDF condition \eqref{eq:DivB:cst123} is not satisfied. 

\begin{theorem} \label{thm:PP:3DMHD:further}
	Assume the polynomial vectors $\{{\bf U}_{ijk}^n({\tt x},{\tt y},{\tt z})\}$ satisfy 
	\eqref{eq:3Dadmissiblity}, and the parameters $\{\alpha_{\ell,n}^{\tt LF}\}$ satisfy 
	\eqref{eq:3DhighLFpara}. 
	Then under the CFL condition \eqref{eq:CFL:3DMHD}, 
	the solution $\bar{\bf U}_{ijk}^{n+1}$ 
	of the scheme \eqref{eq:3DMHD:cellaverage} satisfies  
	that $\bar \rho_{ijk}^{n+1} > 0$, 
	and  
	\begin{equation*}
		{\mathcal E} ( \bar{\bf U}_{ijk}^{n+1} ) > -\Delta t_n \big(\bar{\bf v}_{ijk}^{n+1} \cdot 
		\bar{\bf B}_{ijk}^{n+1} \big) {\rm div}_{ijk}{\bf B}^n, 
	\end{equation*}
	where 
	$\bar{\bf v}_{ijk}^{n+1} :=\bar{\bf m}_{ijk}^{n+1}/\bar{\rho}^{n+1}_{ijk}$. 
\end{theorem}

\bibliographystyle{siamplain}
\bibliography{references}

\end{document}